\documentclass[12pt]{article}

\usepackage{mathrsfs}
\usepackage{color}
\usepackage{amssymb, amsthm, amsfonts, amsxtra, amsmath}

\usepackage{latexsym}
\usepackage[all]{xy}
\usepackage{graphics}
\usepackage{latexsym}
\usepackage{makeidx}
\usepackage{rotating}

\theoremstyle{change}

\newtheorem{Theorem}{Theorem}[section]
\newtheorem{Def}[Theorem]{Definition}
\newtheorem{Lem}[Theorem]{Lemma}
\newtheorem{Prop}[Theorem]{Proposition}
\newtheorem{Cor}[Theorem]{Corollary}

\title{Galois objects and cocycle twisting for locally compact quantum groups}
\author{Kenny De Commer\footnote{Research Assistant of the Research Foundation - Flanders (FWO -
Vlaanderen).}\\\small Department of Mathematics,  K.U. Leuven\\
\small Celestijnenlaan 200B, 3001 Leuven, Belgium\\ \\ \small e-mail: kenny.decommer@wis.kuleuven.be}
\date{}
\begin{document}

\newcommand{\acnabla}{\nabla\!\!\!{^\shortmid}}
\newcommand{\acsigma}{\sigma\!\!\!{^\shortmid}}

\maketitle

\abstract{\noindent In this article, we investigate the notion of a Galois object for a locally compact quantum group. Such an object consists of a von Neumann algebra $N$ with an ergodic integrable coaction $\alpha$ of $M$ on $N$, such that the crossed product is a type I factor. We show how to construct from $(N,\alpha)$ a possibly different locally compact quantum group $(\widehat{P},\Delta_{\widehat{P}})$. By way of application, we prove the following statement: any twisting of a locally compact quantum group by a unitary 2-cocycle is again a locally compact quantum group.\\


 }

\section*{Introduction}

\noindent In commutative geometry, the importance of principal fiber bundles can hardly be overestimated. When passing to non-commutative geometry, they become even more intriguing: one can have interesting principal bundles over the point! In this article, we investigate this phenomenon in the framework of locally compact quantum groups.\\

\noindent In Hopf algebra theory, `non-commutative principal bundles' are known under the name `\emph{faithfully flat Hopf-Galois extensions}'. A \emph{Hopf-Galois extension} consists of the following data: a Hopf algebra $(H,\Delta_H)$ (say over a field $k$), a unital $k$-algebra $A$, and a coaction $\alpha: A\rightarrow A\underset{k}{\otimes} H$. These have to satisfy the following property: with $B$ the fixed point algebra of $\alpha$, the map \[G: A\underset{B}{\otimes} A\rightarrow A\underset{k}{\otimes} H: x\otimes y\rightarrow \alpha(x)(y\otimes 1),\] called \emph{the Galois map}, must be a bijection. Here the surjectivity corresponds geometrically to the freeness of the action, while the injectivity corresponds to the action being proper (actually, `to the action being Cartan' is the more accurate analogy). Saying that the Hopf Galois extension is \emph{faithfully flat}, means that $A$ is faithfully flat as a right $B$-module (in which case the injectivity of the map $G$ comes for free). This corresponds to the local triviality of the bundle. Finally, if we want to have a fiber bundle over a point, we should ask that $B=k\cdot 1_A$ (the condition of being `faithfully flat' becoming obsolete). The couple $(A,\alpha)$ is then called a \emph{(right) Galois object for $(H,\Delta_H)$}. See \cite{Sch3} for a nice overview of these concepts.\\

\noindent We now briefly indicate how the above definitions have to be adapted in the setting of locally compact quantum groups. We will work only in the von Neumann algebra framework. While this is certainly not sufficient to study `locally compact quantum principal fiber bundles', it turns out to be sufficient if one considers a bundle over a point (i.e., there is automatically a C$^*$-algebraic picture available). So let $(M,\Delta)$ be a von Neumann algebraic quantum group (see \cite{Kus3} and \cite{VDae2}). Let $N$ be a von Neumann algebra, and $\alpha: N\rightarrow N\otimes M$ a right coaction. Denote by $N^{\alpha}$ the subalgebra of fixed points. The map $\alpha$ is called \emph{integrable}, if the operator-valued weight $(\iota\otimes \varphi)\alpha$ from $N$ to $N^{\alpha}$ is semi-finite, where $\varphi$ denotes the left invariant nsf weight for $(M,\Delta)$. This is our non-commutative analogue of the action being proper. (In fact, one would also like to use the notion of integrability to define properness on the level of C$^*$-algebras, but the situation there is much more subtle, see e.g. \cite{Rie2}.) When this condition is satisfied, one is able to construct an analogue of the Galois map on the level of $\mathscr{L}^2$-spaces. It will automatically be isometric. When it is actually a unitary, then we call $\alpha$ a \emph{Galois coaction}. Finally, when also $N^{\alpha}=\mathbb{C}$, i.e. when $\alpha$ is ergodic, we call $(N,\alpha)$ a Galois object. This turns out to be equivalent with the condition given in the abstract. \\

\noindent One reason which makes Galois objects so interesting, is that in general they carry with them not one, but two Hopf algebras: if $(A,\alpha)$ is a (right) Galois object for a Hopf algebra $(H,\Delta_H)$, then one can construct from this a \emph{second} `reflected' Hopf algebra $(L,\Delta_L)$ and a (left) coaction $\gamma$ of $L$ on $A$, such that $(A,\gamma)$ becomes a left Galois object, and such that $\gamma$ and $\alpha$ commute. This turns out to be a (part of a) non-commutative generalization of the \emph{Ehresmann construction}, where one lets a locally compact group act freely and properly on a locally compact space, and constructs from this a locally compact groupoid with an action on this same space, commuting with the group action (see e.g. \cite{Mac1}, Example 1.1.5).\\

\noindent We show in this article that such a reflected quantum group also exists when dealing with Galois objects for locally compact quantum groups. While the new locally compact quantum group can be \emph{constructed} more or less as on the algebraic level, there is one technical point which is much less straightforward to establish: namely, the construction gives a priori only a Hopf-von Neumann algebra, and one still has to see if there are invariant weights available. The existence of these weights is the main theorem of this paper (Theorem \ref{corinv}). In fact, we prefer an approach \emph{dual} to the one in Hopf algebra theory, so we rather construct a locally compact quantum group $(\widehat{P},\Delta_{\widehat{P}})$, whose dual then plays the role of $(L,\Delta_L)$.\\

\noindent An important corollary of our results, is that any cocycle twist of a locally compact quantum group is again a locally compact quantum group. I.e.: if $\Omega$ is a unitary 2-cocycle for a locally compact quantum group $(\widehat{M},\widehat{\Delta})$ (so $\Omega\in \widehat{M}\otimes \widehat{M}$ and $(\Omega\otimes 1)(\widehat{\Delta}\otimes \iota)(\Omega) = (1\otimes \Omega)(\iota\otimes \widehat{\Delta})(\Omega)$), then one can show that the cocycle twisted convolution algebra $\widehat{M}\underset{\Omega}{\ltimes} \mathbb{C}$ together with its dual coaction constitutes a Galois object for $(M,\Delta)$, and the new locally compact quantum group $(\widehat{P},\Delta_{\widehat{P}})$ which we obtain is precisely $\widehat{M}$ itself with the new coproduct $\widehat{\Delta}_{\Omega} := \Omega \widehat{\Delta}(\cdot )\Omega^*$. We want to note that in \cite{Fim1}, a special type of such cocycle deformations is discussed: a cocycle on a classical subobject, satisfying certain conditions, is lifted to the whole locally compact quantum group. In this case, more concrete formulas are available for describing the weights on the twisted locally compact quantum group.\\

%\noindent In the article \cite{DeC4}, we will use Galois objects to develop a theory of monoidal equivalence for locally compact quantum groups, and we will give a more detailed comparison with the Hopf algebra theory. We provide there as mentioned a C$^*$-algebraic picture of Galois objects. We also provide a non-trivial example.\\

\noindent As mentioned already, the theory of Galois objects is well-developed for Hopf algebras. It was also investigated for compact quantum groups in \cite{BDV1}, which was in turn based on the work of Wassermann on ergodic actions of compact groups on von Neumann algebras (\cite{Was1},\cite{Was2}). We then investigated this notion for algebraic quantum groups in \cite{DeC1}. It can be shown that the $^*$-Galois objects of \cite{DeC1} can be completed to analytic objects of the kind discussed in this paper (similar to the completion of $^*$-algebraic quantum groups to locally compact quantum groups, as is done in \cite{Kus}), although we have not included a detailed exposition of this fact in this paper.\\

\noindent The specific content of this paper is as follows: in the first two sections, we treat the notion of a (right) \textit{Galois coaction} for a von Neumann algebraic quantum group, as briefly explained above. This notion already appeared implicitly at various places in the literature, for it turns out to be equivalent with the following property: with $\alpha$ denoting the coaction of a locally compact quantum group $(M,\Delta)$ on a von Neumann algebra $N$, being Galois is the same as saying that $N\rtimes M$ can be represented \textit{faithfully} on $\mathscr{L}^2(N)$ by a certain canonical map $\rho$. Our general references for this part are the first four chapters of \cite{Tak1} for the theory of non-commutative integration, section 10 of \cite{Eno1} for some results about inclusions of von Neumann algebras, \cite{Kus3} en \cite{VDae2} for the theory of locally compact quantum groups in the von Neumann algebraic setting, and \cite{Vae1} for the theory of coactions for locally compact quantum groups (which are there just termed `actions').\\

\noindent In the third section, we study \textit{Galois objects}, i.e. Galois coactions $(N,\alpha)$ for which $\alpha$ is \emph{ergodic}. We show that a Galois object has as rich a structure as a locally compact quantum group: we can associate to $N$ certain invariant weights, related by a modular element, and also a one-parameter scaling group. The references for this part are \cite{Kus3}, \cite{VDae2} and \cite{Vae1}. \\

\noindent In the fourth section, we use these results to construct a (possibly new) locally compact quantum group $(\widehat{P},\Delta_{\widehat{P}})$ from such a Galois object, with underlying von Neumann algebra $\widehat{P}=\rho(1\otimes \widehat{M}')'$. The reference for this part is section IX.3 of \cite{Tak1}.\\

\noindent In the fifth section, we consider the special case of cocycle twisted locally compact quantum groups. In the sixth section, we define the notion of a projective corepresentation for a locally compact quantum group, and we show the connection with coactions on type $I$-factors. \\

\noindent \qquad \textit{Preliminaries and notations}\\

\noindent The scalar product of a Hilbert space will be anti-linear in the second argument. If $\mathscr{H},\mathscr{K}$ are Hilbert spaces, we denote by $B(\mathscr{H},\mathscr{K})$ the Banach space of all bounded operators between $\mathscr{H}$ and $\mathscr{K}$, by $B(\mathscr{H})$ the algebra of all bounded operators on $\mathscr{H}$, and by $B_0(\mathscr{H})$ the algebra of all compact operators. If $\xi,\eta\in \mathscr{H}$, we write \[\omega_{\xi,\eta}: B(\mathscr{H})\rightarrow \mathbb{C}: x\rightarrow \langle x\xi,\eta\rangle.\] If $u$ is a unitary on  $\mathscr{H}$, we will denote \[\textrm{Ad}(u):B(\mathscr{H})\rightarrow B(\mathscr{H}):x\rightarrow uxu^*.\] If $\mathscr{H}_1,\mathscr{H}_2$ are two Hilbert spaces, we will denote by $\Sigma$ the flip map \[\mathscr{H}_1\otimes \mathscr{H}_2\rightarrow \mathscr{H}_2\otimes \mathscr{H}_1: \xi\otimes \eta\rightarrow \eta\otimes \xi.\]

\noindent We will also frequently use leg numbering notation: if $\mathscr{H}_i$ are Hilbert spaces and \[u:\mathscr{H}_1\otimes \mathscr{H}_2\rightarrow \mathscr{H}_3\otimes \mathscr{H}_4\] is an operator, we denote for example by $u_{12}$ the operator \[u\otimes 1:\mathscr{H}_1\otimes \mathscr{H}_2\otimes \mathscr{H}_5\rightarrow \mathscr{H}_3\otimes \mathscr{H}_4\otimes \mathscr{H}_5,\] and by $u_{13}$ the operator \[\Sigma_{23}u_{12}\Sigma_{23}:\mathscr{H}_1\otimes \mathscr{H}_5\otimes \mathscr{H}_2\rightarrow \mathscr{H}_3\otimes \mathscr{H}_5\otimes \mathscr{H}_4.\] If $u$ is already indexed, say $u=u_1$, then we write $u_{1,13}$ for $u_{13}$.\\

\noindent If $N$ is a von Neumann algebra, we denote by $N_*$ its predual. We denote by $\mathscr{L}^2(N)$ the universal Hilbert space for GNS-constructions. We denote the spatial tensor product of two von Neumann algebras by $\otimes$.\\

\noindent Let $\varphi_N$ be a fixed normal semi-finite faithful (nsf) weight on $N$. We will then sometimes index the modular structure by $N$ instead of $\varphi_N$ (so the modular automorphism group for example is written as $\sigma_t^N$). We will then write the modular operator as $\nabla_N$ (since the symbol $\Delta$ will be used for the comultiplication of a quantum group).  When we work with another weight $\psi_N$, we will then always use $\psi_N$ as an index. No confusion should arise as to what is meant. We will always write $\mathscr{N}_{\varphi_N}=\{n\in N\mid \varphi_N(n^*n)<\infty\}$ for the space of square integrable elements for $\varphi_N$, we write $\mathscr{M}^+_{\varphi_N}=\{n\in N^+\mid \varphi_N(n)<\infty\}$ for the space of positive integrable elements, and $\mathscr{M}_{\varphi_N}= \textrm{span}\{\mathscr{M}^+_{\varphi_N}\}=\mathscr{N}_{\varphi_N}^*\mathscr{N}_{\varphi_N}$ for the space of integrable elements. The GNS map $\mathscr{N}_{\varphi_N}\rightarrow \mathscr{L}^2(N)$ for $\varphi_N$ is denoted by $\Lambda_N$. We denote by $\mathscr{T}_{\varphi_N}$ the canonical Tomita algebra for $\varphi_N$ (inside $N$): \[\mathscr{T}_{\varphi_N}=\{x\in N\mid x\textrm{ analytic for } \sigma_t^N \textrm{ and } \sigma^N_z(x)\in \mathscr{N}_{\varphi_N}\textrm{ for all }z\in \mathbb{C}\}.\] We then also call $\Lambda_N(\mathscr{T}_{\varphi_N})$ the Tomita algebra for $\varphi_N$ (inside $\mathscr{L}^2(N)$). \\

\noindent The opposite weight of $\varphi_N$ will be denoted by $\varphi_N^{\textrm{op}}$. We see it as a weight on the commutant $N'\subseteq B(\mathscr{L}^2(N))$. It has a natural GNS-construction in $\mathscr{L}^2(N)$: with $J_N$ denoting the modular conjugation of $\varphi_N$, we have a GNS map \[\Lambda_N^{\textrm{op}}: \mathscr{N}_{\varphi_N^{\textrm{op}}}\rightarrow \mathscr{L}^2(N): J_N n^*J_N\rightarrow J_N\Lambda_N(n^*).\] Sometimes however, we will also allow elements of $N$ as input of $\Lambda_N^{\textrm{op}}$: then in fact we first identify $N$ with the opposite von Neumann algebra $N^{\textrm{op}}$ as a linear space (an operation we will write as $n\rightarrow n^{\textrm{op}}$), and then we identify $N^{\textrm{op}}$ with $N'$ as a $^*$-algebra by sending $n^{\textrm{op}}$ to $J_Nn^*J_N$. So for $n\in \mathscr{N}_{\varphi_N}^*$, we will also write $\Lambda_N^{\textrm{op}}(n)=J_N\Lambda_N(n^*)$. This notation is consistent, since $J_Nn^*J_N=n$ for elements in the center.\\

\noindent When $N_1$ and $N_2$ are two von Neumann algebras, and $\varphi_{N_i}$ an nsf weight on $N_i$, we denote by $\varphi_{N_1}\otimes \varphi_{N_2}$ their tensor product (Definition 4.2 in \cite{Tak1}), which is an nsf weight on $N_1\otimes N_2$. We denote its GNS-map with $\Lambda_{N_1}\otimes \Lambda_{N_2}$. One can show that \[\varphi_{N_1} \otimes \varphi_{N_2} = \varphi_{N_1} \circ (\iota\otimes \varphi_{N_2}),\] where $(\iota\otimes \varphi_{N_2})$ is an nsf operator valued weight from $N_1\otimes N_2$ to $N_1$, defined as \[\omega((\iota\otimes \varphi_2)(x)) := \varphi_2 ((\omega\otimes \iota)(x))\] for $x\in (N_1\otimes N_2)^+$ and $\omega\in (N_1\otimes N_2)_*^+$. By symmetry, this gives us a Fubini theorem.\\

\noindent We recall the definition of the Connes-Sauvageot tensor product.\\

\noindent If $\mathscr{H}$ is a left $N$-module, by which we mean a Hilbert space carrying a unital normal representation $\pi_l$ of $N$, and $\varphi_N$ is a nsf weight on $N$, a vector $\xi \in \mathscr{H}$ is called \emph{right bounded} w.r.t. $\varphi_N$ if the map \[ \Lambda_N(\mathscr{N}_{\varphi_N})\rightarrow \mathscr{H}: \Lambda_{N}(x)\rightarrow \pi_l(x)\xi\] is bounded, in which case we denote its closure by $R^{\pi_l,\varphi_N}(\xi)$ (or $R_\xi$ if $\pi_l$ and $\varphi_N$ are fixed). We denote by $_{\varphi_N}\mathscr{H}$ the space of right bounded vectors for $\pi_l$. Similarly, if $\mathscr{H}$ is a right $N$-module, by which we mean a Hilbert space carrying a unital normal anti-representation $\pi_r$ of $N$, a vector $\xi \in \mathscr{H}$ is called \emph{left bounded} w.r.t. $\varphi_N$ if the map \[ \Lambda_N^{\textrm{op}}(\mathscr{N}_{\varphi_N}^*)\rightarrow \mathscr{H}: J_N\Lambda_N(x^*)\rightarrow \pi_r(x)\xi\] is bounded, in which case we denote its closure by $L^{\pi_r,\varphi_N}(\xi)$ (or $L_\xi$ if $\pi_r$ and $\varphi_N$ are fixed). We denote by $\mathscr{H}_{\varphi_N}$ the space of left bounded vectors for $\pi_r$. Remark that when we regard $\mathscr{H}$ as a left $N^{\textrm{op}}$-module in the natural way, then the right bounded vectors with respect to $\varphi_N^{\textrm{op}}$ are exactly the left bounded vectors with respect to $\varphi_N$.\\

\noindent If $(\mathscr{H}^r,\pi_r)$ is a faithful right $N$-module, $(\mathscr{H}^l,\pi_l)$ a faithful left $N$-module, and $\varphi_{N}$ a nsf weight on $N$, we denote by $\mathscr{H}^ r\underset{\varphi_{N}}{_{\pi_r}\otimes_{\pi_l}} \mathscr{H}^l$ (or simply $\mathscr{H}^r \underset{\varphi_N}{\otimes} \mathscr{H}^l$ when $\pi_l,\pi_r$ are clear) their Connes-Sauvageot tensor product with respect to $\pi_l,\pi_r$ and $\varphi_{N}$. It is the Hilbert space closure of the algebraic tensor product of $\mathscr{H}^r_{\varphi_N}$ and $\mathscr{H}^l$ with respect to the scalar product \begin{eqnarray*}  \langle \xi_1\otimes \xi_2,\eta_1\otimes \eta_2\rangle  &=& \langle \pi_l(L_{\eta_1}^*L_{\xi_1}) \xi_2,\eta_2\rangle,\end{eqnarray*} modulo vectors of norm zero. In fact, we could as well start with the algebraic tensor product of $\mathscr{H}^r_{\varphi_N}$ and $_{\varphi_N}\mathscr{H}^l$, since the image of this tensor product in the previous Hilbert space will be dense. On elementary tensors of the last space, we can give a different form of the scalar product, namely \begin{eqnarray*}  \langle \xi_1\otimes \xi_2,\eta_1\otimes \eta_2\rangle  &=& \langle \pi_r(R_{\eta_2}^* R_{\xi_2})\xi_1,\eta_1\rangle.\end{eqnarray*} The image of such an elementary tensor in $\mathscr{H}^r \underset{\varphi_N}{\otimes} \mathscr{H}^l$ will then be denoted by the same symbol, with $\otimes$ replaced by $\underset{\varphi_{N}}{_{\pi_r}\otimes_{\pi_l}}$ or simply $\underset{\varphi_N}{\otimes}$.\\

\noindent Note that these spaces carry faithful normal left representations $\pi_r'$ and $\pi_l'$ of respectively $\pi_r(N)'$ and $\pi_l(N)'$, determined by \[\pi_r'(n_1)\pi_l'(n_2)(\xi_1\underset{\varphi_N}{\otimes} \xi_2)=(n_1 \xi_1)\underset{\varphi_N}{\otimes} (n_2\xi_2),\qquad n_1\in \pi_r(N)',n_2\in \pi_l(N)', \xi_1 \in \mathscr{H}^r_{\varphi_N}, \xi_2 \in  {_{\varphi_N}\mathscr{H}^l}.\] If $N_1 \subseteq B(\mathscr{H}^r)$ is a von Neumann algebra containing $\pi_r(N)$, and $N_2 \subseteq B(\mathscr{H}^l)$ is a von Neumann algebra containing $\pi_l(N)$, the von Neumann algebra $N_1 \underset{N}{_{\pi_r}* _{\pi_l}} N_2:= ( \pi_r'(N_1')\cup\pi_l'(N_2'))'$ is called the \emph{fiber product} of $N_1$ and $N_2$. As an abstract von Neumann algebra, it only depends on $N,N_1,N_2$ and the maps $\pi_r:N\rightarrow N_1$ and $\pi_l:N\rightarrow N_2$. For further properties of the fiber product, see \cite{Eno4}. \\

\noindent We will also need the notion of \emph{intertwiners} and a \emph{linking algebra}. Suppose  $(\mathscr{H}_2,\pi_{r,2})$ and $(\mathscr{H}_1,\pi_{r,1})$ are two right $N$-modules. Denote $Q_{ij}=\{x\in B(\mathscr{H}_j,\mathscr{H}_i)\mid x\pi_{r,j}(n)=\pi_{r,i}(n)x \textrm{ for all }n\in N\}$. We call $Q_{12}$ the space of $\emph{intertwiners}$ between the right $N$-modules $(\mathscr{H}_2,\pi_{r,2})$ and $(\mathscr{H}_1,\pi_{r,1})$. In fact, it is a self-dual $Q_{11}$-$Q_{22}$-Hilbert $W^*$-bimodule (see \cite{Pas1}). The \emph{linking algebra} between $(\mathscr{H}_2,\pi_{r,2})$ and $(\mathscr{H}_1,\pi_{r,1})$ is the von Neumann algebra $Q=\left(\begin{array}{ll} Q_{11} & Q_{12}\\ Q_{21}& Q_{22}\end{array} \right)$, acting on $\left(\begin{array}{l} \mathscr{H}_1\\ \mathscr{H}_2\end{array}\right)=\mathscr{H}_1\oplus \mathscr{H}_2$ in the obvious way. It is the commutant of the direct sum right representation $\pi_{r,1}\oplus \pi_{r,2}$. Most of the time, we will identify the $Q_{ij}$ as subspaces of $Q$, indexing the units of the $Q_{ii}$ then to emphasize that we consider them as projections in $Q$. If $\theta_1$ is a weight on $Q_{11}$ and $\theta_2$ a weight on $Q_{22}$, the \emph{balanced weight} $\theta_1\oplus \theta_2$ is the weight $Q^+\rightarrow \lbrack 0,+\infty\rbrack:\left(\begin{array}{ll} x & y\\ z& w\end{array} \right)\rightarrow \theta_{1}(x)+\theta_{2}(w)$.\\

\noindent  We now briefly recall the definition of a locally compact quantum group, mainly to fix notation.\\

\noindent Let $M$ be a von Neumann algebra, and $\Delta$ a faithful normal unital $^*$-homomorphism $M\rightarrow M\otimes M$, which satisfies coassociativity: \[(\Delta\otimes \iota)\circ \Delta= (\iota\otimes \Delta)\circ \Delta,\] where $\iota$ denotes the identity map. Then the pair $(M,\Delta)$ is called a \textit{Hopf-von Neumann algebra}. A Hopf von Neumann algebra is called \textit{coinvolutive} if there exists an anti-multiplicative $^*$-involution $R:M\rightarrow M$ such that \[\Delta\circ R=(R\otimes R)\circ \Delta^{\textrm{op}},\] where $\Delta^{\textrm{op}}= \textrm{Ad}(\Sigma)\circ \Delta$. A Hopf-von Neumann algebra is called a \textit{locally compact quantum group} if there exist nsf weights $\varphi$ and $\psi$ on $M$ such that \[(\iota\otimes \varphi)\circ\Delta = \varphi,\]\[(\psi\otimes \iota)\circ \Delta= \psi.\] These identities should be interpreted as follows: for any $\omega\in M_*^+$, the weight $\psi\circ (\iota\otimes \omega)\Delta$ should equal the weight $\omega(1)\psi$, and similarly for $\varphi$. (These are in fact the \textit{strong forms} of invariance, and they follow from weaker ones (see Proposition 3.1 of \cite{Kus3}).) A locally compact quantum group will automatically be a coinvolutive Hopf-von Neumann algebra for a canonical map $R$.\\

\noindent We refer to \cite{Kus3} and \cite{VDae2} for further definitions and formulas. We shall also use notations as in those papers. Specifically, we denote by $\varphi$ a (fixed) left invariant nsf weight, by $S$ the antipode, by $\tau_t$ the one-parameter scaling group and by $R$ the unitary antipode (so that $S=R\circ \tau_{-i/2}$). We scale the right invariant weight $\psi$ such that $\psi=\varphi\circ R$. We establish the GNS-constructions for $\varphi$ in the standard form $\mathscr{L}^2(M)$, writing just $\Lambda$ for the GNS-map associated with $\varphi$. We follow the convention of \cite{Kus3} by taking a GNS-construction $\Lambda_{\delta}$ for $\psi$ in $\mathscr{L}^2(N)$ by defining $\Lambda_{\delta}(x) := \Lambda(x\delta^{1/2})$ for $x\in M$ a left multiplier of the square root of the modular element $\delta$ \emph{such that} $x\delta^{1/2} \in \mathscr{N}_{\varphi}$, and then closing $\Lambda_{\delta}$. If $\nu$ is the scaling constant of $(M,\Delta)$, then $\Lambda_{\delta}$ and $\Lambda_{\psi}$ are related by $\Lambda_{\delta} = \nu^{i/4}\Lambda_{\psi}$. The modular one-parametergroup for $\varphi$ is denoted simply by $\sigma_t$, and its corresponding modular operator by $\nabla$. The modular one-parametergroup for $\psi$ is denoted by $\sigma'_t=\textrm{Ad}(\delta^{it})\circ \sigma_t$, its modular operator by $\acnabla$. The canonical (self-dual) one-parametergroup of unitaries implementing the scaling group will be denoted by $P^{it}$.\\

\noindent We denote the dual locally compact quantum group by $(\widehat{M},\widehat{\Delta})$, and also all its other structures are denoted as for $(M,\Delta)$, but with a $\,\,\widehat{}\,\,$ on top. By $W$ and $\widehat{W}=\Sigma W^*\Sigma$ we denote the left regular corepresentation of respectively $(M,\Delta)$ and $(\widehat{M},\widehat{\Delta})$. We write $V$ and $\widehat{V}$ for the right regular corepresentation of respectively $(M,\Delta)$ and $(\widehat{M},\widehat{\Delta})$. We will also from time to time work with the commutant locally compact quantum groups $(M',\Delta')$ and $(\widehat{M}',\widehat{\Delta}')$. For the relationship between all these quantum groups, we refer again to \cite{Kus3}.\\

\noindent For most of the paper, we will work with a fixed locally compact quantum group $(M,\Delta)$. When $(P,\Delta_P)$ is the von Neumann algebraic realization of another locally compact quantum group, we will use the same notations but with a subscript $P$. Note that we also use the symbol $P$ for the scaling operator, since this is standard notation, but in any case, there should not arise any occasion where a von Neumann algebra could get mixed up with an operator!\\

\section{Preliminaries on the basic construction for operator valued weights}

\noindent We collect in this preliminary section some results about operator valued weights. While they are well-known to specialists, we have chosen to present them here in considerable detail, as we do not know a convenient reference for the specific results we need.\\

\noindent Let $N_0\subseteq N$ be a unital inclusion of von Neumann algebras, and
$T$ a normal semi-finite faithful operator valued weight from $N^+$ to the positive extended cone $(N_0)^{+,\textrm{ext}}$ of $N_0$. Let $\mu$ be a fixed nsf
weight on $N_0$, and denote by $\varphi_{N}$ the nsf weight
$\mu\circ T$. Denote the semi-cyclic representation associated to $\varphi_N$ by
$(\mathscr{L}^2(N),\Lambda_{N},\pi_l)$, realized in the standard form. Most of the time, we will write $n$ instead of $\pi_l(n)$ for $n\in N$. Denote
by $\pi_r$ (or $\pi_r^N$ for emphasis) the anti-representation $n\rightarrow J_Nn^*J_N$ of $N$, where $J_N$ is the modular conjugation.
Denote $N_2=\pi_r(N_0)'$,
then $N_0\subseteq N\subseteq N_2$ is called the basic construction. We will also use $\pi_l$ for the natural representation of $N_2$ on $\mathscr{L}^2(N)$, and $\theta_r$ for the natural anti-representation $\theta_r(x)=J_Nx^*J_N$ on $\mathscr{L}^2(N)$ (but this will of course \emph{not} turn $\mathscr{L}^2(N)$ into a $N_2$-$N_2$-bimodule in general).\\

\noindent Consider $x\in \mathscr{N}_{T}=\{n\in N\mid T(n^*n)<\infty\}$. Then $xn\in \mathscr{N}_{\varphi_N}$ when $n\in \mathscr{N}_\mu$, and $\Lambda_{\mu}(n)\rightarrow
\Lambda_N(xn)$ extends from $\Lambda_{\mu}(\mathscr{N}_{\mu})$
to a bounded operator $\mathscr{L}^2(N_0)\rightarrow
\mathscr{L}^2(N)$, which we will denote by
$\Lambda_T(x)$ (following the notations of Theorem 10.6 of \cite{Eno1}). Its adjoint is then determined by
$\Lambda_T(x)^*\Lambda_N(y) =
\Lambda_\mu(T(x^*y))$ for $y\in \mathscr{N}_{\varphi_N}\cap
\mathscr{N}_{T}$. The operators of the form $\Lambda_T(x)\Lambda_T(y)^*$, with $x,y\in\mathscr{N}_T$, will generate a $\sigma$-weakly-dense sub-$^*$-algebra of $N_2$, and if we denote by $T_2$ the canonical operator valued weight from $N_2$ onto $N$ associated to $T$, then $\Lambda_T(x)\Lambda_T(y)^*\in \mathscr{M}_{T_2}:=\mathscr{N}_{T_2}^*\mathscr{N}_{T_2}$ with $T_2(\Lambda_T(x)\Lambda_T(y)^*)=xy^*$  (\cite{Eno1}, Theorem 10.7).\\

\noindent Consider now $\mathscr{L}^2(N)$ as an $N_2$-$N_0$-bimodule, and denote by $\overline{\mathscr{L}^2(N)}$ the conjugate bimodule. Then it is well-known that there is a unitary $N_2$-$N_2$-bimodule map \[\mathscr{L}^2(N)\underset{\mu}{\otimes}  \overline{\mathscr{L}^2(N)}\rightarrow \mathscr{L}^2(N_2):\Lambda_N(x)\underset{\mu}{\otimes} \overline{\Lambda_N^{\textrm{op}}(y)}\rightarrow \Lambda_{\varphi_2}(\Lambda_T(x)\Lambda_T(y)^*)\] for $x,y\in \mathscr{N}_{\varphi_N}\cap\mathscr{N}_{\varphi_N}^*\cap  \mathscr{N}_T$, where $\varphi_2=\varphi_N\circ T_2$. As said, since we will need some more information about this statement, of which we know no appropriate reference in the literature, we will give a proof of it. \\

\noindent We first prove a lemma about interchanging the analytic continuation of a modular one-parametergroup with an operator valued weight.\\

\begin{Lem} Let $Q$ be the linking algebra between the right $N_0$-modules $\mathscr{L}^2(N)$ and $\mathscr{L}^2(N_0)$, and consider the balanced weight $\varphi_2\oplus \mu$ on $Q$. Let $x\in N$ be such that $x$ is analytic for $\sigma_t^N$ and $\sigma_z^N(x)\in \mathscr{N}_T$ for all $z\in \mathbb{C}$. Then $\Lambda_T(x)$ is analytic for $\sigma_t^{Q}$, with $\sigma_z^{Q}(\Lambda_T(x))=\Lambda_T(\sigma_z^{N}(x))$ for all $z\in \mathbb{C}$.

\end{Lem}

\noindent Recall that the notion of a linking algebra between two right von Neumann modules, and the notion of balanced weight, were given in the preliminaries' section of the introduction, of which we also use the notation.

\begin{proof}

First remark that $\Lambda_T(x)\in Q_{12}$ by \cite{Eno1}, Lemma 10.6.(i). Choose $y\in \mathscr{N}_\mu$ and $u,v\in \mathscr{N}_{\varphi_N}$ with $v$ in the Tomita algebra $\mathscr{T}_{\varphi_N}\subseteq N$ for $\varphi_N$. Denote $f(z)=\langle \Lambda_T(\sigma^N_z(x)) \Lambda_\mu(y),J_N\sigma^N_{i/2}(v)J_N\Lambda_N(u)\rangle$ for $z\in \mathbb{C}$. Then \begin{eqnarray*} f(z) &=& \langle J_N\sigma_{i/2}^N(v)^*J_N \Lambda_N(\sigma_z^N(x)y),\Lambda_N(u)\rangle \\ &=& \langle \sigma_z^N(x)\Lambda_N(yv),\Lambda_N(u)\rangle,\end{eqnarray*} and so $f$ is analytic. Moreover, if $z=r+is$ with $ r,s\in \mathbb{R}$, then since $\sigma_t^\mu = (\sigma^N_{t})_{\mid N_0}$,  \begin{eqnarray*} |f(z)| &=& |\langle \sigma_{is}^N(x)\nabla_N^{-ir}\Lambda_N(yv),\nabla_N^{-ir}\Lambda_N(u)\rangle |\\ &=& |\langle \Lambda_N(\sigma_{is}^N(x)\sigma_{-r}^\mu(y)),J_N\sigma_{i/2}^N(\sigma_{-r}^N(v))J_N\Lambda_N(\sigma_{-r}^N(u))\rangle |\\ &=& |\langle \Lambda_T(\sigma_{is}^N(x))\nabla_{\mu}^{-ir}\Lambda_\mu(y),\nabla_N^{-ir}J_N\sigma_{i/2}^N(v)J_N\Lambda_N(u)\rangle |,\end{eqnarray*} and so we can conclude, by the Phragm\'{e}n-Lindel\"{o}f principle, that the modulus of $f$ is bounded on every horizontal strip by $M_x\|\omega\|$, where \[\omega =\omega_{ \Lambda_\mu(y),J_N\sigma^N_{i/2}(v)J_N\Lambda_N(u)} \in B(\mathscr{L}^2(N_0),\mathscr{L}^2(N))_*,\] and $M_x$ is a number depending only on $x$ and the chosen strip. The same is of course true for linear combinations of such $\omega$, and since these span a dense subspace of $B(\mathscr{L}^2(N_0),\mathscr{L}^2(N))_*$, we get that $z\rightarrow \Lambda_T(\sigma_z^N(x))$ is bounded on compact sets. But then this function is analytic (for example by condition $A.1.(iii)$ in the appendix of $\cite{Tak1}$). Since $\sigma_t^{Q}$ is implemented by $\nabla_N^{it}\oplus \nabla_\mu^{it}$ and $\nabla_N^{it}\Lambda_T(x)\nabla_\mu^{-it}=\Lambda_T(\sigma_t^N(x))$, the result follows.

\end{proof}

\noindent We can now provide a convenient Tomita algebra for $\varphi_2$. Let $\mathscr{T}_{\varphi_N}\subseteq N$ be the Tomita algebra for $\varphi_N$, and denote \[\mathscr{T}_{\varphi_N,T}= \{ x\in \mathscr{T}_{\varphi_N}\cap \mathscr{N}_T\cap \mathscr{N}_T^*\mid \sigma_z^N(x)\in \mathscr{N}_T\cap \mathscr{N}_T^* \textrm{ for all } z\in \mathbb{C}\}.\] (This space is called the Tomita algebra for $\varphi_N$ and $T$ in Proposition 2.2.1 of \cite{Eno3}.) Denote the linear span of $\{\Lambda_T(x)\Lambda_T(y)^*\mid x,y\in\mathscr{T}_{\varphi_N,T}\}$ by $\mathfrak{A}_2$, and further denote by $(\mathscr{L}^2(N_2),\Lambda_{N_2},\pi_l^{N_2})$ the natural semi-cyclic representation for $\varphi_2$.

\begin{Prop}\label{tex11} We have $\mathfrak{A}_2\subseteq \mathscr{D}(\Lambda_{N_2})$,
and $\mathfrak{A}_2$ is a Tomita algebra for $(N_2,\varphi_{2})$.

\end{Prop}

\noindent By the second statement, we mean that $\Lambda_{N_2}(\mathfrak{A}_2)$ is a sub-Tomita algebra of the natural Tomita algebra $\Lambda_{N_2}(\mathscr{T}_{\varphi_2})$ for $\varphi_2$, closed in $\mathscr{L}^2(N_2)$, which still has $N_2$ as its left von Neumann algebra, and also with the corresponding weight on $N_2$ coinciding with $\varphi_2$.

\begin{proof}

\noindent For $x,y\in \mathscr{T}_{\varphi_N,T}$, we know that $\Lambda_T(x)\Lambda_T(y)^*\in \mathscr{M}_{T_2}$, with $T_2(\Lambda_T(x)\Lambda_T(y)^*)=xy^*$. Since $x,y\in \mathscr{T}_{\varphi_N}$, also $xy^*\in \mathscr{M}_{\varphi_N}$. Hence $\mathfrak{A}_2\subseteq \mathscr{M}_{\varphi_2}$, and so certainly $\mathfrak{A}_2\subseteq \mathscr{D}(\Lambda_{N_2})$.\\

\noindent It is clear that $\mathfrak{A}_2$ is closed under the $^*$-involution. Now choose $x,y,u,v\in \mathscr{T}_{\varphi_N,T}$. Then \[(\Lambda_T(u)\Lambda_T(v)^*) (\Lambda_T(x)\Lambda_T(y)^*)=\Lambda_T(uT(v^*x))\Lambda_T(y)^*.\] We want to show that $uT(v^*x)\in \mathscr{T}_{\varphi_N,T}$. It is clear that $uT(v^*x) \in \mathscr{N}_{\varphi_N}^*\cap \mathscr{N}_T\cap \mathscr{N}_T^*$. By the previous lemma, we have, using notation as there, that $\Lambda_T(v)$ and $\Lambda_T(x)$ are analytic for $\sigma_t^Q$, with $\sigma_z^{Q}(\Lambda_T(v))=\Lambda_T(\sigma_z^N(v))$ and $\sigma_z^{Q}(\Lambda_T(x))=\Lambda_T(\sigma_z^N(x))$ for all $z\in \mathbb{C}$. But then also $\Lambda_T(v)^*\Lambda_T(x)=T(v^*x)$ analytic for $\sigma_t^Q$, with $\sigma_z^Q(T(v^*x))= T(\sigma_{\overline{z}}^N(v)^*\sigma_z^N(x))$ for all $z\in \mathbb{C}$. Since $\sigma_t^Q$ restricts to $\sigma_t^\mu$ on $N_0$, and also $\sigma_t^N$ restricts to $\sigma_t^\mu$ on $N_0$, we get that $uT(v^*x)$ is analytic for $\sigma_t^{N}$, with $\sigma_z^N(uT(v^*x))=\sigma_z^N(u)T(\sigma_{\overline{z}}^N(v)^*\sigma_z^N(x))$ for $z\in \mathbb{C}$. Since $\mathscr{T}_{\varphi_N,T}$ is invariant under all $\sigma_z^N$ with $z\in \mathbb{C}$, we get that $\sigma_z^N(uT(v^*x))\in \mathscr{N}_{\varphi_N}^*\cap \mathscr{N}_T\cap \mathscr{N}_T^*$ for all $z\in \mathbb{C}$. Hence $uT(v^*x)\in \mathscr{T}_{\varphi_N,T}$, and thus $(\Lambda_T(u)\Lambda_T(v)^*) (\Lambda_T(x)\Lambda_T(y)^*)\in \mathfrak{A}_2$. \\

\noindent We have shown so far that $\Lambda_{N_2}(\mathfrak{A}_2)$ is a sub-left Hilbert algebra of $\Lambda_{N_2}(\mathscr{N}_{\varphi_2}\cap \mathscr{N}_{\varphi_2}^*)$. But by the previous lemma, $\mathfrak{A}_2$ consists of analytic elements for $\sigma^{Q}_t$, which restricts to $\sigma^{N_2}_t$ on $N_2$. So in fact $\Lambda_{N_2}(\mathfrak{A}_2)$ is a sub-Tomita algebra of $\Lambda_{N_2}(\mathscr{T}_{\varphi_2})$.\\

\noindent Now we show that $\mathfrak{A}_2$ is $\sigma$-weakly dense in $N_2$. For this, it is enough to show that $\Lambda_T(\mathscr{T}_{\varphi_N,T})$ is strongly dense in $Q_{12}$. Note that $\Lambda_T(\mathscr{T}_{\varphi_N,T})$ is closed under right multiplication with elements from $\mathscr{T}_\mu\subseteq N_0$, which are $\sigma$-weakly dense in $N_0$. Then by a similar argument as in the proof of Theorem 10.6.(ii), it is sufficient to prove that if $z\in Q_{12}$ and $z^*\Lambda_T(x)=0$ for all $x\in \mathscr{T}_{\varphi_N,T}$, then $z=0$. So suppose $z$ satisfies this condition. Choose $y\in \mathscr{N}_{\mu}$ analytic for $\sigma_t^\mu$. Then \begin{eqnarray*} \pi_r^{N_0}(\sigma_{i/2}^\mu (y))z^*\Lambda_N(x) &=& z^*\pi_r^N(\sigma_{i/2}^N (y))\Lambda_N(x) \\  &=& z^* \Lambda_N(xy)\\ &=& z^* \Lambda_T(x)\Lambda_\mu(y)\\ &=& 0.\end{eqnarray*} Letting $\pi_r^{N_0}(\sigma_{i/2}^\mu(y))$ tend to 1, we see that $z^*$ vanishes on $\Lambda_N(\mathscr{T}_{\varphi_N,T})$. Now choose $x\in \mathscr{M}_{\varphi_N}\cap \mathscr{M}_T$. Then $x_n=\sqrt{\frac{n}{\pi}}\int_{-\infty}^{+\infty} e^{-nt^2}\sigma_t^N(x)dt$ is in $\mathscr{T}_{\varphi_N,T}$ by Lemma 10.12 of \cite{Eno1}, and $\Lambda_N(x_n)$ converges to $\Lambda_N(x)$. Hence $z^*$ vanishes on $\Lambda_N(\mathscr{M}_{\varphi_N}\cap \mathscr{M}_T)$. Since $\mathscr{N}_{\varphi_N}\cap \mathscr{N}_T$ is weakly dense in $N$ and $\Lambda_N(\mathscr{N}_{\varphi_N}\cap \mathscr{N}_T)$ is normdense in $\mathscr{L}^2(N)$, we get that $z^*=0$, and the density claim follows.\\

\noindent Now let $\mathscr{G}$ be the closure of $\Lambda_{N_2}(\mathfrak{A}_2)$. Then for $x\in N_2$, we get that $\pi_l^{N_2}(x)$ will restrict to an operator $\pi^{\mathfrak{A}_2}_l(x):\mathscr{G}\rightarrow \mathscr{G}$, since $\mathfrak{A}_2$ is dense in $N_2$. Then the left von Neumann algebra associated with $\Lambda_{N_2}(\mathfrak{A}_2)$ is $\pi^{\mathfrak{A}_2}_l(N_2)$. If we denote by $\varphi_{1\frac{1}{2}}$ the weight on $\pi^{\mathfrak{A}_2}_l(N_2)$ associated to $\Lambda_{N_2}(\mathfrak{A}_2)$, then it is clear that $\varphi_2$, the weight $(\varphi_{1\frac{1}{2}}\circ \pi_l^{\mathfrak{A}_2})$ and $\mathfrak{A}_2$ satisfy the conditions of Proposition VIII.3.15 of \cite{Tak1}, hence $\varphi_2=\varphi_{1\frac{1}{2}}\circ \pi_l^{\mathfrak{A}_2}$, which finishes the proof.

\end{proof}

\noindent \textit{Remark:} It also follows easily from Lemma 10.12 of \cite{Eno1} that $\mathscr{T}_{\varphi_N,T}$ itself is $\sigma$-weakly dense in $N$.\\

\noindent Let $\mathscr{L}^2(N)\underset{\mu}{\otimes}
\mathscr{L}^2(N)$ denote the Connes-Sauvageot tensor product, with its natural $N_2$-$N_2$-bimodule structure. Denote by $\mathscr{K}$ the natural image of the algebraic tensor product $\Lambda_N(\mathscr{T}_{\varphi_N,T})\odot \Lambda_N(\mathscr{T}_{\varphi_N,T})$ inside $\mathscr{L}^2(N)\underset{\mu}{\otimes}
\mathscr{L}^2(N)$.

\begin{Theorem}\label{dens} The space $\mathscr{K}$ is dense in $\mathscr{L}^2(N)\underset{\mu}{\otimes}
\mathscr{L}^2(N)$, and the map \[\mathscr{K}\rightarrow \mathscr{L}^2(N_2): \Lambda_N(x)\underset{\mu}{\otimes}
\Lambda_N(y)  \rightarrow
\Lambda_{N_2}(\Lambda_T(x)\Lambda_T(y^*)^*)\] extends to a unitary equivalence of $N_2$-$N_2$-bimodules.\end{Theorem}

\begin{proof}

\noindent First note that the expression on the left is well-defined by Theorem 10.6.(v) of \cite{Eno1}, and then by definition, we have for $x,y,z,w\in \mathscr{T}_{\varphi_N,T}$ that \begin{eqnarray*} \langle \Lambda_N(x)\underset{\mu}{\otimes}
\Lambda_N(y),\Lambda_N(z)\underset{\mu}{\otimes}
\Lambda_N(w)\rangle &=& \langle (\Lambda_T(z)^*\Lambda_T(x))\Lambda_N(y),\Lambda_N(w)\rangle \\ &=& \varphi_N(w^*T(z^*x)y)\\ &=&  \langle \Lambda_{N_2}(\Lambda_T(x)\Lambda_T(y^*)^*),\Lambda_{N_2}(\Lambda_T(z)\Lambda_T(w^*)^*)\rangle,\end{eqnarray*} so that the given map extends to a well-defined partial isometry. Since $\Lambda_N(\mathscr{T}_{\varphi_N,T})$ is dense in $\mathscr{L}^2(N)$ (which was proven in the course of the previous proposition), we have that  $\mathscr{K}$ is dense in $\mathscr{L}^2(N)\underset{\mu}{\otimes}
\mathscr{L}^2(N)$. Since also $\Lambda_{N_2}(\mathfrak{A}_2)$ is dense in $\mathscr{L}^2(N_2)$, the extension is in fact a unitary.\\

\noindent The fact that it is a bimodule map follows from a straightforward computation (since we only have to check the bimodule property for operators in $\mathfrak{A}_2$ and vectors in $\mathscr{K}$ and $\Lambda_{N_2}(\mathfrak{A}_2)$).

\end{proof}

\noindent \textit{Remark:} If we identify $\mathscr{L}^2(N)$ with $\overline{\mathscr{L}^2(N)}$ as an $N_0$-$N_2$-bimodule by the unitary $\Lambda_N(y)\rightarrow \overline{\Lambda_N^{\textrm{op}}(y^*)}$, we get the isomorphism $\mathscr{L}^2(N)\underset{\mu}{\otimes}  \overline{\mathscr{L}^2(N)}\rightarrow \mathscr{L}^2(N_2)$ mentioned before. In some sense, this is a more natural unitary, but in our specific setting, the former one is easier to work with. \\

\noindent In the following, we will hence identify $\mathscr{L}^2(N)\underset{\mu}{\otimes}
\mathscr{L}^2(N)$ and $\mathscr{L}^2(N_2)$ in this manner.

\begin{Lem}\label{lem12} Let $x,y$ be elements of $\mathscr{T}_{\varphi_N,T}$, and let $p$ be an element of $\mathscr{N}_{\varphi_2}$. Then \[\langle \Lambda_N(x)\underset{\mu}{\otimes} \Lambda_N(y),\Lambda_{N_2}(p)\rangle = \langle \Lambda_{N}(x),p\Lambda_N(\sigma_{-i}^N(y^*))\rangle.\] Conversely, if $p\in N_2$ and $\xi \in \mathscr{L}^2(N_2)$ are such that \[\langle \Lambda_N(x)\underset{\mu}{\otimes} \Lambda_N(y),\xi\rangle = \langle \Lambda_{N}(x),p\Lambda_N(\sigma_{-i}^N(y^*))\rangle\] for all $x,y\in \mathscr{T}_{\varphi_N,T}$, then $p\in \mathscr{N}_{\varphi_2}$ and $\Lambda_{N_2}(p)=\xi$. \end{Lem}

\begin{proof} Suppose $p= \Lambda_T(z)\Lambda_T(w^*)^*$ for some $z,w\in \mathscr{T}_{\varphi_N,T}$. Then since $w^*T(z^*x) \in \mathscr{N}_{\varphi_N}\cap\mathscr{N}_{\varphi_N}^*$, we have \begin{eqnarray*} \langle  \Lambda_N(x)\underset{\mu}{\otimes} \Lambda_N(y),\Lambda_{N_2}(p) \rangle &=&  \langle \Lambda_N(x)\underset{\mu}{\otimes} \Lambda_N(y),\Lambda_N(z)\underset{\mu}{\otimes} \Lambda_N(w)\rangle \\ &=& \varphi_N(w^*T(z^*x)y) \\ &=&
\varphi_N(\sigma_i^N (y)w^*T(z^*x))\\ &=& \langle \Lambda_N(w^*T(z^*x)),\Lambda_N(\sigma_{-i}^N(y^*))\rangle \\ &=& \langle \Lambda_T(w^*)\Lambda_T(z)^*\Lambda_N(x),\Lambda_N(\sigma_{-i}^N(y^*))\rangle \\ &=& \langle \Lambda_N(x),p\Lambda_N(\sigma_{-i}^N(y^*))\rangle.\end{eqnarray*}
As $\mathfrak{A}_2$, being a left Hilbert algebra for $\varphi_2$, is a strong-norm core for $\Lambda_{N_2}$, the result holds true for any $p\in \mathscr{N}_{\varphi_2}$.\\

\noindent Now we prove the converse statement. So let $p\in N_2$ and $\xi \in \mathscr{L}^2(N_2)$ be such that \[\langle \Lambda_N(x)\underset{\mu}{\otimes} \Lambda_N(y),\xi\rangle = \langle \Lambda_{N}(x),p\Lambda_N(\sigma_{-i}^N(y^*))\rangle\] for all $x,y\in \mathscr{T}_{\varphi_N,T}$. Then, since $\mathfrak{A}_2^{\textrm{op}}$ is a strong-norm core for $\Lambda_{N_2}^{\textrm{op}}$ (whose notation was introduced in the preliminaries' section of the introduction), it is enough to prove that $p\Lambda_{N_2}^{\textrm{op}}(a)=\pi_r^{N_2}(a)\xi$ for all $a\in \mathfrak{A}_2$. Now if $a=\Lambda_T(x)\Lambda_T(y^*)^*$, then \begin{eqnarray*} \Lambda_{N_2}^{\textrm{op}}(a)&=&J_{N_2}\Lambda_{N_2}(a^*)\\&=&\Lambda_N(\sigma_{-i/2}^N(x))\underset{\mu}{\otimes} \Lambda_N(\sigma^N_{-i/2}(y)).\end{eqnarray*} So if also $b\in \mathfrak{A}_2$ with $b=\Lambda_T(z)\Lambda_T(w^*)^*$, $w,z\in \mathscr{T}_{\varphi_N,T}$, then \[\langle \Lambda_{N_2}(b),p\Lambda_{N_2}^{\textrm{op}}(a)\rangle =\langle \Lambda_N(z),p\Lambda_T(\sigma_{-i/2}^N(x))\Lambda_T(\sigma^N_{-i/2}(y)^*)^*\Lambda_N(\sigma^N_{-i}(w^*))\rangle\] by the first part of the lemma. On the other hand, we have \begin{eqnarray*} \langle\Lambda_{N_2}(b),\pi_r^{N_2}(a)\xi\rangle &=& \langle \pi_r^{N_2}(a)^*\Lambda_{N_2}(b),\xi\rangle \\ &=& \langle \Lambda_{N_2}(b\sigma_{i/2}^{N_2}(a)^*),\xi\rangle\\ &=& \langle \Lambda_{N_2}(\Lambda_T(z)\Lambda_T(w^*)^*\Lambda_T(\sigma^N_{i/2}(y)^*)\Lambda_T(\sigma_{i/2}^N(x))^*),\xi\rangle \\ &=& \langle\Lambda_{N_2}(\Lambda_T(z)\Lambda_T(\sigma_{i/2}^N(x)T(\sigma_{i/2}^N(y)w^*))^*),\xi\rangle \\ &=&\langle \Lambda_N(z),p\Lambda_{N}(\sigma_{-i/2}^N(x)T(\sigma_{-i/2}^N(y)\sigma_{-i}^N(w^*)))\rangle,\end{eqnarray*} which equals our earlier expression, hence proving $p\Lambda_{N_2}^{\textrm{op}}(a)=\pi_r^{N_2}(a)\xi$ for all $a\in \mathfrak{A}_2$.

\end{proof}

\noindent We prove two further results which naturally belong in this section, but of which only the second one will be used in this paper (the first one will be used in another article).

\begin{Lem}\label{lemint} Let $N_0\subseteq N$ be a unital inclusion of von Neumann algebras, $T: N\rightarrow N_0$ an nsf operator valued weight, $\mu$ an nsf weight on $N_0$, and $\varphi_N$ the nsf weight $\mu\circ T$. Suppose $x\in N$ and $z\in B(\mathscr{L}^2(N_0),\mathscr{L}^2(N))$ are such, that for any $y\in \mathscr{N}_\mu$, we have $xy\in \mathscr{N}_{\varphi_N}$ and $\Lambda_N(xy)=z\Lambda_\mu(y)$. Then $x\in \mathscr{N}_T$ with $\Lambda_T(x)=z$.\end{Lem}

\begin{proof} Choose $y,w\in \mathscr{N}_\mu$ with $w$ in the Tomita algebra of $\mu$. Then \begin{eqnarray*} \pi_r(w)z\Lambda_\mu(y) &=& \pi_r(w)\Lambda_N(xy)\\ &=& \Lambda_N(xy\sigma_{-i/2}^N(w))\\ &=& z\Lambda_\mu(y\sigma_{-i/2}^\mu(w))\\ &=& z\pi_r^{N_0}(w)\Lambda_{\mu}(y),\end{eqnarray*} so that $z$ is a right $N_0$-module map. It follows that $z^*z\in N_0$. \\

\noindent Now by Lemma 4.7 of \cite{Tak1}, there exists a closed positive (possibly unbounded) operator $A$, such that $\Lambda_\mu(y)\in \mathscr{D}(A)$ and $\omega_{\Lambda_\mu(y),\Lambda_\mu(y)}(T(x^*x)) = \langle A\Lambda_\mu(y),A\Lambda_\mu(y)\rangle$. Also, since for any element $u\in N_0^{+,\textrm{ext}}$, one can find a sequence $u_n\in N_0^+$ such that $u_n \nearrow u$ pointwise on $(N_0)_*^+$ (see the proof of Proposition 4.17.(ii) in \cite{Tak1}), we get that $\omega_{\Lambda_\mu(y),\Lambda_\mu(y)}(T(x^*x)) = \mu (y^*T(x^*x)y)$, using Corollary 4.9 of \cite{Tak1} (which allows us to extend weights to the extended positive cone). Using the bimodularity of $T$, we get  \begin{eqnarray*} \langle A\Lambda_\mu(y),A\Lambda_\mu(y)\rangle &=& \omega_{\Lambda_\mu(y),\Lambda_\mu(y)}(T(x^*x)) \\&=&  \mu(y^*T(x^*x)y) \\ &=& \mu(T(y^*x^*xy)) \\ &=& \langle \Lambda_N(xy),\Lambda_N(xy)\rangle \\ &=& \langle z\Lambda_\mu(y),z\Lambda_\mu(y)\rangle,\end{eqnarray*} from which we conclude that $A$ is bounded. Hence $T(x^*x)$ is bounded, and then of course $\Lambda_T(x)=z$ follows.

\end{proof}

\begin{Lem}\label{propincl} Let $\begin{array}{lll} N_{10} &\subseteq  & N_{11}\\ \;\,\textrm{\begin{sideways} $\subseteq$\end{sideways}} & & \;\,\textrm{\begin{sideways} $\subseteq$ \end{sideways}}\\
   N_{00} &\subseteq & N_{01}\end{array}$ be unital normal inclusions of von Neumann algebras. Denote, for $i\in\{0,1\}$, by $Q_{i}$ the linking algebra between the right $N_{i0}$-modules $\mathscr{L}^2(N_{i0})$ and $\mathscr{L}^2(N_{i1})$. Suppose $T_1$ is an nsf operator valued weight $N_{11}^+\rightarrow N_{10}^{+,\textrm{ext}}$ whose restriction $T_0$ to $N_{01}^+$ is an nsf operator valued weight $N_{01}^+\rightarrow N_{00}^{+,\textrm{ext}}$. Then there is a natural normal embedding of $Q_0$ into $Q_{1}$, determined by $\Lambda_{T_0}(x)\rightarrow \Lambda_{T_1}(x)$ for $x\in \mathscr{N}_{T_0}$.

\end{Lem}

\noindent\emph{Remark:} The inclusion will in general \emph{not} be unital. Consider for example the case where $N_{11}=M_2(\mathbb{C})$ and all other algebras equal to $\mathbb{C}$.

\begin{proof} By assumption, if $x,y\in \mathscr{N}_{T_0}$, then $x,y\in \mathscr{N}_{T_1}$, and $T_0(x^*y)=T_1(x^*y)$. Denote by $\tilde{\mathscr{Q}}_1$ the $^*$-algebra generated by the $\Lambda_{T_1}(x)$, $x\in \mathscr{N}_{T_0}$, and by $\tilde{Q}_1$ its $\sigma$-weak closure. Denote by $\mathscr{Q}_0$ the $^*$-algebra generated by the $\Lambda_{T_0}(x)$, $x\in \mathscr{N}_{T_0}$. We want to show that $Q_0$ and $\tilde{Q}_1$ are isomorphic in the indicated way.\\

\noindent Now for $a_i,b_i\in \mathscr{N}_{T_0}$, it is easy to check that $\sum_{i} \Lambda_{T_1}(a_i)\Lambda_{T_1}(b_i)^* = 0$ iff $\sum_i \Lambda_{T_0}(a_i)\Lambda_{T_0}(b_i)^*=0$, so we already have an isomorphism $F$ at the level of $\mathscr{Q}_0$ and $\mathscr{\tilde{Q}}_1$. Denote by $e_0$ the unit of $N_{00}$, seen as a projection in $Q_0$, and denote by $e_1$ the unit of $N_{00}$ as a projection in $\tilde{Q}_1$. Suppose that $x_i$ is a bounded net in $\mathscr{Q}_0$ which converges to $0$ in the $\sigma$-weak topology. Then for any $a,b\in \mathscr{Q}_0$, we have that $e_0ax_ibe_0$ converges to 0 $\sigma$-weakly. Applying $F$, we get that $e_1F(a)F(x_i)F(b)e_1$ converges $\sigma$-weakly to 0, and then also $ce_1F(a)F(x_i)F(b)e_1d$, for any $c,d\in \tilde{Q}_1$. Since $\tilde{Q}_1 e_1 \tilde{\mathscr{Q}}_1$ is $\sigma$-weakly dense in $\tilde{Q}_1$, we get that $F(x_i)$ converges $\sigma$-weakly to 0. Since the same argument applies to $F^{-1}$, we see that $F$ extends to a $^*$-isomorphism between $Q_0$ and $\tilde{Q}_1$, and we are done.
\end{proof}

\noindent\textit{Remark:} We could also have used the results from \cite{Pas1} concerning self-dual Hilbert $W^*$-modules.\\

\section{Galois coactions}

\noindent Let $(M,\Delta)$ be the von Neumann algebraic realization of a locally compact quantum group. Let $N$ be a von Neumann algebra
equipped with a right coaction $\alpha$ of $(M,\Delta)$, by which we mean a
faithful normal unital $^*$-homomorphism $\alpha:N\rightarrow N\otimes M$
such that $(\iota\otimes \Delta)\alpha=(\alpha\otimes\iota)\alpha$.
Denote by $N^{\alpha}$ the von Neumann algebra of coinvariants: \[N^{\alpha}=\{x\in N\mid \alpha(x)=x\otimes 1\}.\] \emph{In this paper, we will only work with integrable
coactions,} so the
normal faithful operator valued weight
$T=(\iota\otimes\varphi) \alpha$ from $N^+$ to $(N^{\alpha})^{+,\textrm{ext}}$, where $\varphi$ is the left invariant weight for $(M,\Delta)$, is
assumed to be semi-finite (\cite{Vae1}, Proposition 1.3 and Definition 1.4).  Let $\mu$ be a fixed nsf
weight on $N^{\alpha}$, and denote by $\varphi_N$ the nsf weight
$\mu\circ T$. It will be $\delta$-invariant (see Definition III.1 of \cite{Eno2} and Definition 2.3 of \cite{Vae1}). With the exception that $N_0$ is now written $N^\alpha$, we will use notation as in the previous section.\\

\noindent Recall from Theorem 5.3 of \cite{Vae1} that the integrability of $\alpha$ is equivalent with the existence of a canonical map \[\rho:N\rtimes M\rightarrow B(\mathscr{L}^2(N))\] (which we will explicitly write down a bit later on), where $N\rtimes M=(\alpha(N)\cup (1\otimes \widehat{M}'))''$ denotes the crossed product of $N$ with respect to the coaction $\alpha$. We can also consider the map \[\mathscr{K}\rightarrow \mathscr{L}^2(N)\otimes
\mathscr{L}^2(M): \Lambda_N(x)\underset{\mu}{\otimes} \Lambda_N(y)\rightarrow (\Lambda_N\otimes \Lambda)(\alpha(x)(y\otimes 1)) \] for $x,y\in \mathscr{T}_{\varphi_N,T}$, where $\mathscr{K}$ was introduced just before Theorem \ref{dens}, and $\mathscr{T}_{\varphi_N,T}$ just before Proposition \ref{tex11}. Then this is easily seen to be a well-defined isometry. Denote its extension by \[G:\mathscr{L}^2(N)\underset{\mu}{\otimes}
\mathscr{L}^2(N)\rightarrow \mathscr{L}(N)\otimes
\mathscr{L}^2(M).\]\\

\noindent The main goal of this section is to prove

\begin{Theorem}\label{Theo1} The map $\rho$ is faithful iff $G$ is a unitary.
\end{Theorem}

\noindent The result will follow from the following set of lemmas and propositions, which conclude with Lemma \ref{lem222}.\\

\noindent Consider the dual weight
$\varphi_{N\rtimes M}$ of $\varphi_N$ on $N\rtimes M$ (\cite{Vae1}, Definition 3.1). Then there is a natural
semi-cyclic representation $(\mathscr{L}^2(N)\otimes \mathscr{L}^2(M),\Lambda_{N\rtimes M},\pi^{N\rtimes M}_l)$
 for $\varphi_{N\rtimes M}$, determined
 by \[\Lambda_{N\rtimes M}((1\otimes m)\alpha(x)) = \Lambda_N(x)\otimes \widehat{\Lambda}^{\textrm{op}}(m)\]
 for $x\in \mathscr{N}_{\varphi_N}$ and $m\in \mathscr{N}_{\widehat{\varphi}^{\textrm{op}}}$. Most of the time, we will suppress the symbol $\pi^{N\rtimes M}_l$. Note that we use here the results of \cite{Vae1}, adapted to the setting of right coactions.\\

\noindent Denote by $U \in B(\mathscr{L}^2(N))\otimes M$ the unitary implementation of $\alpha$ (i.e., the unitary implementation for $\alpha^{\textrm{op}}$ in the sense of \cite{Vae1}, with its legs interchanged).
By Proposition 4.3 and Theorem 4.4 of \cite{Vae1}, it can be defined as
$U={J}_{N\rtimes M}(J_N\otimes \widehat{J})$, with $J_{N\rtimes M}$ the modular conjugation of the dual weight $\varphi_{N\rtimes M}$, as
well as by the formula \begin{equation}\label{form}(\iota\otimes
\omega_{\xi,\eta})(U)\Lambda_{N}(z)
= \Lambda_{N}((\iota\otimes
\omega_{\delta^{-1/2}\xi, \eta})\alpha(z)),\end{equation}
where $\xi,\eta\in \mathscr{L}^2(M)$ with $\xi\in \mathscr{D}(\delta^{-1/2})$, $z\in
\mathscr{N}_{\varphi_N}$. The surjective
normal $^*$-homomorphism $\rho$ from $N\rtimes M$ to $N_2$ mentioned before is then given on
the generators of $N\rtimes M$ by \[\left\{\begin{array}{ll} \rho(\alpha(x))=\pi_l(x) &\textrm{for }
x\in N,\\ \rho(1\otimes (\iota\otimes \omega)(V))=(\iota\otimes
\omega)(U)& \textrm{for }\omega\in M_*,\end{array}\right.\] where we recall that $\pi_l$ is just the standard representation for $N$, that $V$ is the right regular multiplicative unitary for $(M,\Delta)$, and that $N_2$ is the von Neumann algebra in the basic construction $N^\alpha\subseteq N\subseteq N_2$. \\

\noindent In the following Proposition, we also use the associated basic construction for the weight $T$, i.e. $N^{\alpha}\underset{T}{\subseteq} N\underset{T_2}{\subseteq} N_2$ denotes the basic construction obtained from $N^{\alpha}\underset{T}{\subseteq} N$, as explained in the beginning of the previous section. We also denote again $\varphi_2 = \varphi_N\circ T_2$.\\

\begin{Prop}\label{prop2222} If $m\in \mathscr{N}_{\widehat{\varphi}^{\textrm{op}}}$ and $z\in \mathscr{N}_{\varphi_N}$, then $\rho((1\otimes m)\alpha(z))\in \mathscr{N}_{\varphi_2}$ and \[G^*(\Lambda_{N}(z)\otimes \widehat{\Lambda}^{\textrm{op}}(m))= \Lambda_{N_2}(\rho((1\otimes m)\alpha(z))).\]\end{Prop}

\begin{proof} Choose $m\in \mathscr{N}_{\widehat{\varphi}^{\textrm{op}}}$ of the form
$(\iota\otimes \omega)(V)$, with $\omega$ such that $x\rightarrow \overline{\omega(S(x)^*)}$ coincides with a normal functional $\omega^*$ on $\mathscr{D}(S)$, and such that moreover $x\rightarrow \omega^*(x \delta^{-1/2})$ coincides with a normal functional $\omega^*_{\delta}$ on the set of left multipliers of $\delta^{-1/2}$ in $M$. Then, since $(\iota\otimes \omega)(V)^*= (\iota\otimes \omega^*)(V)$, we have, for $x,y\in \mathscr{T}_{\varphi_N,T}$ and
$z\in \mathscr{N}_{\varphi_N}$, that
\begin{eqnarray*} \langle \Lambda_{N}(x),\rho(1\otimes
m) z\Lambda_{N}(\sigma_{-i}^{N}(y^*))\rangle &=&
\langle \Lambda_{N}((\iota\otimes
\omega^*_{\delta})(\alpha(x))),\Lambda_{N}(z\sigma_{-i}^{N}(y^*))\rangle
\\&=& \varphi_N(\sigma_{i}^{N}(y)z^*(\iota\otimes
\omega^*_{\delta})(\alpha(x)))\\ &=& \varphi_N(z^*(\iota\otimes
\omega^*_{\delta})(\alpha(x))y).\end{eqnarray*} But since for $a\in
\mathscr{N}_{\varphi}$, we have $\langle
\Lambda(a),\widehat{\Lambda}^{\textrm{op}}(m)\rangle
 = \omega^*_{\delta}(a)$, this last expression equals \[\langle G(\Lambda_N(x)\underset{\mu}{\otimes} \Lambda_N(y)),\Lambda_N(z)\otimes \widehat{\Lambda}^{\textrm{op}}(m)\rangle.\]  Since such $m$ form a strong-norm core for $\widehat{\Lambda}^{\textrm{op}}$ (by standard smoothing arguments), we have \[\langle \Lambda_{N}(x),\rho(1\otimes
m) z\Lambda_{N}(\sigma_{-i}^{N}(y^*))\rangle =\langle G(\Lambda_N(x)\underset{\mu}{\otimes} \Lambda_N(y)),\Lambda_N(z)\otimes \widehat{\Lambda}^{\textrm{op}}(m)\rangle,\]for all $m\in \mathscr{N}_{\widehat{\varphi}^{\textrm{op}}}$. By the second part of Lemma \ref{lem12}, we then get $\rho((1\otimes m)\alpha(z))\in \mathscr{N}_{\varphi_2}$ and \[\Lambda_{N_2}(\rho((1\otimes m)\alpha(z)))=G^*(\Lambda_N(z)\otimes \widehat{\Lambda}^{\textrm{op}}(m))\] for all $m\in \mathscr{N}_{\widehat{\varphi}^{\textrm{op}}}$ and $z\in \mathscr{N}_{\varphi_N}$.

\end{proof}

\begin{Lem}\label{lem25} The map $G$ is a left $N\rtimes M$-module morphism.\end{Lem}
\begin{proof}
Denoting again by $\pi_l^{N_2}$ the natural representation of $N_2$ on $\mathscr{L}^2(N)$, it is easy to see that $G\pi_l^{N_2}(x)=\alpha(x)G$ for all $x\in \mathscr{T}_{\varphi_N,T}$, hence this is true for all $x\in N$. Further, if $m\in \widehat{M}'$, $n\in \mathscr{N}_{\widehat{\varphi}^{\textrm{op}}}$ and $z\in \mathscr{N}_{\varphi_N}$, then $\rho((1\otimes mn)\alpha(z))\in \mathscr{N}_{\varphi_2}$ by the previous lemma, and we have  \begin{eqnarray*} \pi_l^{N_2}(\rho(1\otimes m)) G^*(\Lambda_{N}(z)\otimes \widehat{\Lambda}^{\textrm{op}}(n)) &=& \Lambda_{N_2}(\rho((1\otimes mn)\alpha(z)))\\ &=& G^*(\Lambda_{N}(z)\otimes \widehat{\Lambda}^{\textrm{op}}(mn)),\end{eqnarray*}
hence $G\pi_l^{N_2}(\rho(1\otimes m))=(1\otimes m)G$ for all $m\in \widehat{M}'$. Since $N\rtimes M$ is generated by $1\otimes \widehat{M}'$ and $\alpha(N)$, the lemma is proven.\end{proof}

\noindent \textit{Remark:} This implies that $\pi_l^{N_2}(\rho(x))=G^*xG$ for $x\in N\rtimes M$, as $G$ is an isometry.\\

\begin{Lem}\label{lem26} The following commutation relations hold: \begin{enumerate}
\item $\nabla_{N\rtimes M}^{it}G=G\nabla_{N_2}^{it}$,
\item $J_{N\rtimes M}G=GJ_{N_2}$.
\end{enumerate}
\end{Lem}

\noindent Here $\nabla_{N\rtimes M}$ denotes the modular operator for $\varphi_{N\rtimes M}$.

\begin{proof}
By the earlier identification of $\mathscr{L}^2(N_2)$ with $\mathscr{L}^2(N)\underset{\mu}{\otimes}\mathscr{L}^2(N)$, it's easy to see that \[\nabla_{N_2}^{it}(\Lambda_N(x)\underset{\mu}{\otimes} \Lambda_N(y))= \Lambda_N(\sigma_t^{N}(x))\underset{\mu}{\otimes} \Lambda_N(\sigma_t^N(y))\] for $x,y\in \mathscr{T}_{\varphi_N,T}$, so for the first commutation relation, we must show that for all $x,y\in \mathscr{T}_{\varphi_N,T}$, we have \[ \nabla_{N\rtimes M}^{it}((\Lambda_N\otimes \Lambda)(\alpha(x)(y\otimes 1)))= (\Lambda_N\otimes \Lambda)(\alpha(\sigma_t^N(x))(\sigma_t^N(y)\otimes 1)).\]

\noindent Define the one-parametergroup $\kappa_t$ on $M$ by $\kappa_t(a)=\delta^{-it}\tau_{-t}(a)\delta^{it}$ for $a\in M$. As in the proof of Proposition 4.3 in \cite{Vae1}, one can show that \[\nabla_{N\rtimes M}^{it}=\nabla_N^{it}\otimes q^{it},\] where $q^{it}\Lambda(a)=\Lambda(\kappa_t(a))$ for $a\in \mathscr{N}_{\varphi}$. Since $\sigma_t^{\varphi_{N\rtimes M}}\circ \alpha= \alpha\circ \sigma_t^N$ by Proposition 3.7.2 of \cite{Vae1}, we have for $x,y\in \mathscr{T}_{\varphi_N,T}$ and $\xi\in \mathscr{L}^2(M)$ that \[\nabla_{N\rtimes M}^{it}(\alpha(x)(\Lambda_N(y)\otimes \xi))= \alpha(\sigma_t^N(x))(\Lambda_N(\sigma_t^N(y))\otimes q^{it}\xi).\] Now let $a\in \mathscr{N}_{\varphi}$ be analytic for $\sigma_t$. Since $\sigma_t$ commutes with $\kappa_t$, we have that $\kappa_t(a)$ is then also analytic for $\sigma_t$, with $\sigma_z(\kappa_t(a))=\kappa_t(\sigma_z(a))$ for $t\in \mathbb{R},z\in \mathbb{C}$. Hence for such $a$, and $x,y\in \mathscr{T}_{\varphi_N,T}$, we get \begin{eqnarray*} \nabla_{N\rtimes M}^{it}(1\otimes J\sigma_{i/2}(a)^*J)((\Lambda_N\otimes \Lambda)(\alpha(x)(y\otimes 1))) &=& \nabla_{N\rtimes M}^{it} (\Lambda_N\otimes \Lambda)(\alpha(x)(y\otimes a))\\ &=& (\Lambda_N\otimes \Lambda)(\alpha(\sigma_t^N(x))(\sigma_t^N(y)\otimes \kappa_t(a)))\\ &=& (1\otimes J\kappa_t(\sigma_{i/2}(a))^*J)(\Lambda_N\otimes \Lambda)(\alpha(\sigma_t^N(x))(\sigma_t^N(y)\otimes 1)),\end{eqnarray*} and letting $\sigma_{i/2}(a)$ tend to 1, we see that \[\nabla_{N\rtimes M}^{it}((\Lambda_N\otimes \Lambda)(\alpha(x)(y\otimes 1)))= (\Lambda_N\otimes \Lambda)(\alpha(\sigma_t^N(x))(\sigma_t^N(y)\otimes 1)),\] which proves the first commutation relation.\\

\noindent From this, it follows that $G^*\nabla_{N\rtimes M}^{1/2}$ will equal the restriction of $\nabla_{N_2}^{1/2}G^*$ to $\mathscr{D}(\nabla_{N\rtimes M}^{1/2})$. Denote $t_{N\rtimes M}=J_{N\rtimes M}\nabla_{N\rtimes M}^{1/2}$ and $t_{N_2}= J_{N_2}\nabla_{N_2}^{1/2}$. Then $t_{N_2}G^*=J_{N_2}G^*\nabla_{N\rtimes M}^{1/2}$ on $\mathscr{D}(\nabla_{N\rtimes M}^{1/2})$. So to prove the second commutation relation (in the form $G^*J_{N\rtimes M}=J_{N_2}G^*$), we only have to find a subset $K$ of $\mathscr{D}(\nabla_{N\rtimes M}^{1/2})=\mathscr{D}(t_{N\rtimes M})$ whose image under $\nabla_{N\rtimes M}^{1/2}$ (or $t_{N\rtimes M}$) is dense in $\mathscr{L}^2(N\rtimes M)$, and on which $t_{N_2}G^*$ and $G^*t_{N\rtimes M}$ agree. But take \[K= \textrm{span}\{\alpha(x) \Lambda_{N\rtimes M}((1\otimes m)\alpha(y))\mid x,y\in \mathscr{T}_{\varphi_N,T}, m\in \mathscr{N}_{\widehat{\varphi}^{\textrm{op}}}\cap \mathscr{N}_{\widehat{\varphi}^{\textrm{op}}}^*\}.\] Then clearly $K\subseteq \mathscr{D}(t_{N\rtimes M})$ and $t_{N\rtimes M}(K)=K$, since \[t_{N\rtimes M}(\alpha(x) \Lambda_{N\rtimes M}((1\otimes m)\alpha(y)))= \alpha(y^*) \Lambda_{N\rtimes M}((1\otimes m^*)\alpha(x^*)).\] Furthermore, if $x,y\in \mathscr{T}_{\varphi_N,T}$ and $m\in \mathscr{N}_{\widehat{\varphi}^{\textrm{op}}}\cap \mathscr{N}_{\widehat{\varphi}^{\textrm{op}}}^*$, we get from Proposition \ref{prop2222} and Lemma \ref{lem25} that $\rho(\alpha(x)(1\otimes m)(\alpha(y)))$ and $\rho(\alpha(y^*)(1\otimes m^*)(\alpha(x^*)))$ are both in $\mathscr{D}(\Lambda_{N_2})$, and that $G^*\alpha(x) \Lambda_{N\rtimes M}((1\otimes m)\alpha(y))\in \mathscr{D}(t_{N_2})$, with \begin{eqnarray*} t_{N_2} G^*\alpha(x) \Lambda_{N\rtimes M}((1\otimes m)\alpha(y)) &=& t_{N_2} \Lambda_{N_2}(\rho(\alpha(x)(1\otimes m)\alpha(y)))\\ &=& \Lambda_{N_2}(\rho(\alpha(y^*)(1\otimes m^*)\alpha(x^*))) \\ &=& G^*\alpha(y^*) \Lambda_{N\rtimes M}((1\otimes m^*)\alpha(x^*))\\ &=& G^* t_{N\rtimes M} \alpha(x) \Lambda_{N\rtimes M}((1\otimes m)\alpha(y)) .\end{eqnarray*} Since $K$ is dense in $\mathscr{L}^2(N\rtimes M)$, the second commutation relation is proven.\end{proof}

\noindent Denote by $p$ the central projection in $N\rtimes M$ such that ker$(\rho)=(1-p)(N\rtimes M)$. Denote by $\rho_p$ the restriction of $\rho: N\rtimes M\rightarrow N_2$ to $p(N\rtimes M)$, and by $\tilde{\varphi}_2$ the nsf weight $\varphi_{N\rtimes M}\circ \rho_p^{-1}$ on $N_2$.

\begin{Lem}\label{lem222} The projection $GG^*$ equals $p$.\end{Lem}

\begin{proof} By Lemma \ref{lem25}, $G$ is a left $N\rtimes M$-module morphism, hence $GG^*\in (N\rtimes M)'$, and $GG^* \leq p$ since $G^*pG=\rho(p)=1$. By the previous lemma, $GG^*$ commutes with $J_{N\rtimes M}$, hence $GG^*$ is in the center $\mathscr{Z}(N\rtimes M)$. Since $\rho(GG^*)=G^*(GG^*)G=1$, we must have $GG^*=p$. \\
\end{proof}

\noindent As mentioned, Theorem \ref{Theo1} follows immediately from this, since $G$ is unitary iff $p=1$ iff $\rho$ is faithful.\\

\begin{Prop}\label{propprop} The weight $\tilde{\varphi}_2$ equals $\varphi_2$.\end{Prop}

\begin{proof}

\noindent If $m\in \mathscr{N}_{\widehat{\varphi}^{\textrm{op}}}$ and $z\in \mathscr{N}_{\varphi_N}$, then $\rho((1\otimes m)\alpha(z))\in \mathscr{N}_{\tilde{\varphi}_2}$, and we can make a GNS-map $\Lambda_{\tilde{\varphi}_2}$ for $\tilde{\varphi}_2$ in $p(\mathscr{L}^2(N)\otimes \mathscr{L}^2(M))$ by \begin{eqnarray*} \Lambda_{\tilde{\varphi}_2}(\rho((1\otimes m)\alpha(z)))&=&p(\Lambda_{N\rtimes M}((1\otimes m)\alpha(z)))\\&=&p(\Lambda_N(z)\otimes \widehat{\Lambda}^{\textrm{op}}(m)),\end{eqnarray*} since by the results of \cite{Vae1}, the linear span of the $(1\otimes m)\alpha(z)$ forms a $\sigma$-strong$^*$-norm core for $\Lambda_{N\rtimes M}$. By Proposition \ref{prop2222} and the previous lemma, \[\Lambda_{\tilde{\varphi}_2}(\rho((1\otimes m)\alpha(z)))= G(\Lambda_{\varphi_2}(\rho((1\otimes m)\alpha(z)))).\] Since $G$ is a left $N\rtimes M$-module map, we obtain that also $(\mathscr{L}^2(N_2),G^*\circ\Lambda_{\tilde{\varphi}_2},\pi_l^{N_2} )$ is a GNS-construction for $\tilde{\varphi}_2$, and that $(G^*\circ\Lambda_{\tilde{\varphi}_2})\subseteq \Lambda_{\varphi_2}$.\\

\noindent By the first commutation relation of Lemma \ref{lem26}, it also follows that the modular operators for the GNS-constructions $\Lambda_{\varphi_2}$ and $G^*\circ \Lambda_{\tilde{\varphi}_2}$ are the same. Hence $\varphi_2=\tilde{\varphi}_2$ by Proposition VIII.3.16 of \cite{Tak1}.

\end{proof}

\noindent \textit{Remark:} This implies that $T_2$ equals $T_{N\rtimes M}\circ \rho_p^{-1}$ with $T_{N\rtimes M}$ the canonical operator valued weight $N\rtimes M\rightarrow N$, by Theorem IX.4.18 of \cite{Tak1}. Note that this result was obtained in Proposition 5.7 of \cite{Vae1} under the hypothesis that $\rho$ was faithful.\\

\noindent It follows from Proposition \ref{propprop} that $G^*$ coincides with the map $Z:\mathscr{L}^2(N\rtimes M)\rightarrow \mathscr{L}^2(N_2)$ which sends $\Lambda_{N\rtimes M}(z)$ to $
\Lambda_{\tilde{\varphi}_2}(\rho(z))$ (cf. the proof of Theorem 5.3 in \cite{Vae1}). So we can summarize our results by saying that the following square of $N\rtimes M$-bimodules and bimodule morphisms commutes:
\begin{equation}\label{eqn2} \xymatrix{
\ar[d]^{\cong}\mathscr{L}^2(N_2) \ar[r]^{Z^*}  &
\mathscr{L}^2(N\rtimes M) \ar[d]^{\cong} \\
\mathscr{L}^2(N)\underset{\mu}{\otimes} \mathscr{L}^2(N) \ar[r]_{G} &
\mathscr{L}^2(N)\otimes \mathscr{L}^2(M)}
\end{equation}\\

\begin{Def} Let $\alpha$ be an integrable coaction of $(M,\Delta)$ on a von Neumann algebra $N$. We call the associated map \[\rho:N\rtimes M\rightarrow N_2\] the \emph{Galois homomorphism} for $\alpha$. We call the operator \[\tilde{G}=\Sigma G: \mathscr{L}^2(N)\underset{\mu}{\otimes} \mathscr{L}^2(N)\rightarrow \mathscr{L}^2(M)\otimes \mathscr{L}^2(N)\] the \emph{Galois map} or the \emph{Galois isometry} for $(N,\alpha)$. We call the coaction $\alpha$ \emph{Galois} if the Galois homomorphism is bijective, or equivalently, if the Galois isometry is a unitary (in which case we call it of course the \emph{Galois unitary}).\end{Def}

\noindent \emph{Remark:} The reason for putting a flip map in front of $G$, is to make it right $N$-linear in such a way that this is just right $N$-linearity on the second factors of the domain and range, so that `the second leg of $\tilde{G}$ is in $N$'. See the section on Galois objects for more.\\

\noindent Note that the notion of a Galois coaction already appeared, as far as we know, nameless at various
places in the literature. The property of $G$ being surjective is the
motivation for the terminology, as the bijectivity of the above map
$N\underset{N^{\alpha}}{\otimes} N\rightarrow N\otimes M$ (in the
algebraic context of Hopf algebras) is precisely the condition to have a Galois
coaction of a Hopf algebra. Also note that as these are the
non-commutative generalizations of principal fiber bundles, we could
call the space pertaining to a Galois coaction a
\textit{measured quantum principal fiber bundle} (with $(M,\Delta)$ as the principal fiber), an object which is quite trivial in the commutative setting! (There probably is no need to account for the `local triviality', as the functor $ \mathscr{L}^2(N)\underset{\mu}{\otimes} - $ is automatically an equivalence between the categories of respectively left $N^{\alpha}$ and left $N_2$-modules (so the `faithful flatness' condition in the algebraic setup is automatically fulfilled).) Further note that the above square was essentially constructed in the setting of algebraic quantum groups in \cite{VDae1}.\\

\noindent We give a further characterization of Galois coactions in the following corollary. Given an integrable coaction $\alpha$ of $(M,\Delta)$ on $N$, write $N_{00}=N^{\alpha}\otimes \mathbb{C}$, $N_{01}=\alpha(N)$, $N_{10}=N\otimes \mathbb{C}$ and $N_{11}=N\otimes M$. Write $T_1$ for the operator valued weight $(\iota\otimes \varphi)$ from $(N\otimes M)^+$ to $(N\otimes \mathbb{C})^{+,\textrm{ext}}$. Then $T_1$ restricts to the canonical operator valued weight $T_0=T\circ \alpha^{-1}$ from $\alpha(N)^+$ to $(N^{\alpha}\otimes \mathbb{C})^{+,\textrm{ext}}$, so we are in the situation of Lemma \ref{propincl}. Then also denote again by $Q_0$ and $Q_1$ the corresponding linking algebras. We regard $\mathscr{L}^2(N)$ and $\mathscr{L}^2(N)\otimes \mathscr{L}^2(M)\cong \mathscr{L}^2(N\rtimes M)$ as right $N\rtimes M$-modules in the natural way, using the Galois homomorphism for the first one.\\

\begin{Cor} The following statements are equivalent: \begin{enumerate} \item The coaction $\alpha$ is Galois.
\item The inclusion $Q_0\subseteq Q_1$ is unital.
\item The image of $(Q_0)_{12}$ is exactly the space of $N\rtimes M$-intertwiners.
\end{enumerate}
\end{Cor}

\begin{proof} We will write $\tilde{Q}$ for the linking algebra between the right $N\rtimes M$-modules $\mathscr{L}^2(N)$ and $\mathscr{L}^2(N\rtimes M)$. Denote explicitly the inclusion $Q_0 \subseteq Q_1$ by $F$. We first show that $F(Q_0)\subseteq \tilde{Q}$. Take $x\in \mathscr{N}_T$, then $\alpha(x)\in \mathscr{N}_{T_1}$, and it is easily seen that $\Lambda_{T_1}(\alpha(x))\Lambda_N(y)=(\Lambda_N\otimes \Lambda)(\alpha(x)(y\otimes 1))$ for $y\in \mathscr{N}_{\varphi_N}$. Hence $\Lambda_{T_1}(\alpha(x))= G\circ l_x$, where we denote by $l_x$ the map $\mathscr{L}^2(N)\rightarrow \mathscr{L}^2(N)\underset{\mu}{\otimes} \mathscr{L}^2(N)$ which sends $\xi \in \mathscr{L}^2(N)$ to $\Lambda_N(x) \underset{\mu}{\otimes} \xi$. But $l_x$ is a right $N\rtimes M$-intertwiner, and we know that $G$ is a right $N\rtimes M$-intertwiner by the diagram (\ref{eqn2}). Hence $\Lambda_{T_1}(\alpha(x))\in \tilde{Q}_{12}$, and then it follows that $F(Q_0)\subseteq \tilde{Q}$.\\

\noindent Next, we show that $\rho\circ F_{11}\circ \tilde{F}=\iota$, where $F_{11}$ denotes the restriction of $F$ to $(Q_0)_{11}$, and $\tilde{F}$ is the isomorphism $N_2\rightarrow (Q_0)_{11}$ determined by $\Lambda_T(x)\Lambda_T(y)^*\rightarrow \Lambda_{T_0}(\alpha(x))\Lambda_{T_0}(\alpha(y))^*$ for $x,y\in \mathscr{N}_{T}$. Namely: for $x,y\in \mathscr{N}_{T}$, we have \begin{eqnarray*} (\rho\circ F_{11})(\Lambda_{T_0}(\alpha(x))\Lambda_{T_0}(\alpha(y))^*) &=& (\pi_l^{N_2})^{-1}(G^*(G l_xl_y^* G^*)G)\\ &=& (\pi_l^{N_2})^{-1}(l_xl_y^*)\\ &=& \Lambda_{T}(x)\Lambda_{T}(y)^*,\end{eqnarray*} again by using the diagram (\ref{eqn2}).\\

\noindent By these observations, the equivalence of the first and third statement is immediate. Since we have also shown that in fact $F_{11}\circ \tilde{F}=G(\pi_l^{N_2}(\cdot))G^*$, the equivalence of the first and second statement follows.

\end{proof}

\noindent We now present some natural examples of Galois coactions. \\

\noindent First, every dual coaction, or more generally, every semidual coaction (i.e. the ones for which there exists a unitary $v\in B(\mathscr{L}^2(M))\otimes N$ with $(\iota\otimes \alpha)(v) = \widehat{W}_{13}v_{12}$)  is Galois, by Proposition 5.12 of \cite{Vae1}.\\

\noindent Also, whenever $(N,\alpha)$ is an integrable \emph{outer} coaction (i.e. $N\rtimes M \cap \alpha(N)' = \mathbb{C}1$), the coaction is automatically Galois (since $N\rtimes M$ is then a factor).\\

\noindent Next, suppose $(M_1,\Delta_1)$ and $(M,\Delta)$ are locally compact quantum groups, with $(\widehat{M}_1,\widehat{\Delta}_1)$ a closed quantum subgroup of $(\widehat{M},\widehat{\Delta})$: we mean by this that $\widehat{M}_1$ is a unital sub-von Neumann algebra of $\widehat{M}$ such that the restriction of $\widehat{\Delta}$ to $\widehat{M}_1$ coincides with $\widehat{\Delta}_1$ (\cite{Vae5}, Definition 2.9).\\

\noindent Associated to $(M_1,\Delta_1)$ and $(M,\Delta)$, there is a canonical right coaction $\Gamma_r$ of $(M_1,\Delta_{M_1})$ on $M$ by right translation and a left coaction $\Gamma_l$ by left translation (see for example the first paragraphs of section 4 in \cite{Vae5} for the left setting).

\begin{Prop}\label{propsubg} If $(M_1,\Delta_1)$ and $(M,\Delta)$ are locally compact quantum groups, with $(\widehat{M}_1,\widehat{\Delta}_1)$ a closed quantum subgroup of $(\widehat{M},\widehat{\Delta})$, the associated coaction $\Gamma_r$ is Galois. Conversely, if $(M,\Delta)$ and $(M_1,\Delta_1)$ are locally compact quantum groups for which there is a right Galois coaction $\Gamma_r$ of $(M_1,\Delta_1)$ on $M$ such that $(\iota\otimes \Gamma_r)\Delta = (\Delta\otimes \iota)\Gamma_r$, then $(\widehat{M}_1,\widehat{\Delta}_1)$ can be identified with a closed quantum subgroup of $(\widehat{M},\widehat{\Delta})$ in such a way that $\Gamma_r$ is precisely the coaction by right translations.\end{Prop}

\begin{proof} First suppose that $(\widehat{M}_1,\widehat{\Delta}_1)$ is a closed quantum subgroup of $(\widehat{M},\widehat{\Delta})$. Then we can also embed $\widehat{M}_1'$ into $\widehat{M}'$ by a normal map $F$ which respects the comultiplications. Denote $V_{\Gamma}=(F\otimes \iota)(V_1)$, where $V_1 \in \widehat{M}_1'\otimes M_1$  is the right regular representation of $(M_1,\Delta_1)$. The aforementioned coaction $\Gamma_r$ is then explicitely given as $\Gamma_r(x)=V_{\Gamma}(x\otimes 1)V_{\Gamma}^*$ for $x\in M$. We can make the following sequence of isomorphisms: \begin{eqnarray*}  M\rtimes M_1 &=& (\Gamma_r(M)\cup (1\otimes \widehat{M}_1'))''\\ &\cong & ((M\otimes 1)\cup V_{\Gamma}^*(1\otimes \widehat{M}_1')V_{\Gamma})''\\  &= & ((M\otimes 1)\cup (F\otimes \iota)(\widehat{\Delta}_1'(\widehat{M}_1')))''\\ &\cong & ((M\otimes 1)\cup \widehat{\Delta}'(F(\widehat{M}_1')))''\\ &\cong & (M\cup F(\widehat{M}_1')),\end{eqnarray*} where we have used that $V_1$ is also the left regular representation for $(\widehat{M}_1',\widehat{\Delta}'_1)$. Since it's easy to see that the resulting isomorphism satisfies the requirements for the Galois homomorphism (using that $V_{\Gamma}$ is actually the corepresentation implementing $\Gamma_r$), the coaction is Galois.\\

\noindent Now suppose that we have a Galois coaction $\Gamma_r$ such that $(\iota\otimes \Gamma_r)\Delta = (\Delta\otimes \iota)\Gamma_r$. Denote by $(\widehat{A}'_u,\widehat{\Delta}'_u)$ the universal locally compact quantum group associated with $(\widehat{M}',\widehat{\Delta}')$, and similarly for $(\widehat{M}_1',\widehat{\Delta}_1')$ (cf. \cite{Kus2}). By the results in section 12 of \cite{Kus2} (in the setting of right coactions), we get that there is a canonical non-degenerate $^*$-homomorphism $F_u: \widehat{A}_{1,u}'\rightarrow M(\widehat{A}_u')$ which intertwines the comultiplications. Since $\Gamma_r$ is Galois, we also have a faithful normal homomorphism $F:\widehat{M}_1'\rightarrow B(\mathscr{L}^2(M))$. Denote by $U$ the corepresentation associated with $\Gamma_r$. By the results of \cite{Kus2}, it is an element of $\widehat{M}'\otimes M_1$. Denote by $\pi_u$ and $\pi_{1,u}$ respectively the canonical homomorphisms $M(\widehat{A}'_u)\rightarrow \widehat{M}'$ and $M(\widehat{A}'_{1,u})\rightarrow \widehat{M}_1'$ from the multiplier C$^*$-algebras to the von Neumann algebras. Identify $M_*$ and $(M_1)_*$ with their images in respectively $\widehat{A}'_u$ and $\widehat{A}'_{1,u}$ (noting that the dual of $(\widehat{M}',\widehat{\Delta}')$ is $(M,\Delta^{\textrm{op}})$). Then we can deduce again from \cite{Kus2} that for $\omega\in (M_1)_*$, we have $\pi_u(F_u(\omega)) = (\iota\otimes \omega)(U)$. Since the $\pi_u$ also commute with the comultiplications, we deduce that $F(\widehat{M}_1')\subseteq \widehat{M}'$, and that this embedding respects the comultiplications. Hence $(\widehat{M}_1',\widehat{\Delta}'_1)$ is a closed quantum subgroup of $(\widehat{M}',\widehat{\Delta}')$ (and then of course also $(\widehat{M}_1,\widehat{\Delta}_1)$ is a closed quantum subgroup of $(\widehat{M},\widehat{\Delta})$). \\

\noindent Finally, we should show that $\Gamma_r$ is just the coaction naturally associated with this closed quantum subgroup. But this is clear, as $\Gamma_r$ is implemented by the corepresentation $U$, which equals $V_{\Gamma}$ since also $\pi_u(F_u(\omega)) = (\iota\otimes \omega)(V_{\Gamma})$.\\
\end{proof}

\section{Galois objects}

\noindent We will now treat in detail the case of ergodic Galois coactions, i.e. $N^{\alpha}=\mathbb{C}$.
\begin{Def} If $N$ is a von Neumann algebra, $(M,\Delta)$ a locally compact quantum group and $\alpha$ an \emph{ergodic} Galois coaction of $(M,\Delta)$ on $N$, we call $(N,\alpha)$ a (right) \emph{Galois object} for $(M,\Delta)$.\end{Def}

\noindent In this case, the constructions of the previous sections greatly simplify. First of all,
$T=(\iota\otimes \varphi)\alpha$ itself will already be an nsf weight on $N$ (identifying $\mathbb{C}$ with $\mathbb{C}\cdot 1_N$), so we denote it by $\varphi_N$. Then $\mathscr{N}_T=\mathscr{N}_{\varphi_N}$. There is a slight ambiguity of notation then, as $\Lambda_N(x)$ denotes either an element of $\mathscr{H}$ or a linear operator $\mathbb{C}\rightarrow \mathscr{H}$, but this ambiguity disappears if we identify the Hilbert spaces $B(\mathbb{C},\mathscr{H})$ and $\mathscr{H}$ by sending $x$ to $x\cdot 1$. Next, $N\rtimes M\overset{\rho}{\cong} N_2$ becomes the whole of $B(\mathscr{L}^2(N))$, and $\varphi_2=\textrm{Tr}(\cdot \nabla_N)$. Further, $\mathscr{L}^2(N_2)$ will be identified with $\mathscr{L}^2(N)\otimes \mathscr{L}^2(N)$ by the map \[ \Lambda_{N_2}(\Lambda_N(x)\Lambda_N(y^*)^*) \rightarrow \Lambda_N(x)\otimes \Lambda_N(y) \qquad \textrm {for } x,y\in \mathscr{N}_{\varphi_N}\cap \mathscr{N}_{\varphi_N}^*.\] For $x\in B(\mathscr{L}^2(N))$, we have $\pi_l^{N_2}(x)=x\otimes 1$, $\pi_r^{N_2}(x)= 1\otimes \pi_r(x)$ (where $\pi_r(x)= J_Nx^*J_N$), $\nabla_{N_2}^{it}=\nabla_N^{it}\otimes \nabla_N^{it}$ and $J_{N_2}= \Sigma (J_N\otimes J_N)$. In the following, we will now also use the symbol $\widehat{\pi}_l$ to denote the left representation of $\widehat{M}'$ on $\mathscr{L}^2(N)$ (so $\widehat{\pi}_l(m)=\rho(1\otimes m)$ for $m\in \widehat{M}'$), and we will write $\widehat{\theta}_r(m)$ for $\theta_r(\rho(1\otimes m))=J_N\widehat{\pi}_l(m)^*J_N$ when $m\in \widehat{M}'$. In fact, it's not difficult to see that for any integrable coaction, we have then $\widehat{\theta}_r(m)=\widehat{\pi}_l(\widehat{R}'(m))$: just use that $(J_N\otimes \widehat{J})U(J_N\otimes \widehat{J})=U^*$ and $(J\otimes \widehat{J})V(J\otimes \widehat{J})=V^*$.\\

\noindent The aim of this section is to show that there is much extra structure on a Galois object $(N,\alpha)$, closely resembling the
one of $(M,\Delta)$ itself. In particular, we are able to show that there
exists an nsf invariant weight on $N$. To find it, we will search a 1-cocycle to deform $\varphi_N$.\\

\noindent Let $(N,\alpha)$ be a fixed Galois object, with Galois unitary $\tilde{G}:\mathscr{L}^2(N)\otimes
\mathscr{L}^2(N)\rightarrow \mathscr{L}^2(M)\otimes
\mathscr{L}^2(N)$. To begin with, we will write down some commutation relations. In the following, $\alpha^{\textrm{op}}(x)=\Sigma \alpha(x)\Sigma$ for $x\in N$, and $\widehat{\pi}_r(m)$ will denote the operator $\widehat{J}m^*\widehat{J}$ on $\mathscr{L}^2(M)$.

\begin{Lem}\label{com1} For all $x\in N$ and $m\in \widehat{M}'$, we have
\begin{enumerate}
\item $\tilde{G}(x\otimes 1)=\alpha^{\textrm{op}}(x)\tilde{G}$,
\item $\tilde{G}(\widehat{\pi}_l(m)\otimes 1)=(m\otimes 1)\tilde{G}$,
\item $\tilde{G}(1\otimes \pi_r(x))=(1\otimes \pi_r(x))\tilde{G}$,
\item $\tilde{G}(1\otimes \widehat{\theta}_r(m))= (\widehat{\pi}_r\otimes \widehat{\theta}_r)((\widehat{\Delta}')^{\textrm{op}}(m))\tilde{G}$.\end{enumerate}
    \end{Lem}

\begin{proof} These equalities follow directly from the fact that $G$ is a $N\rtimes M$-bimodule map. For the fourth one, we remark that the right representation $\pi^{N\rtimes M}_r$ of $N\rtimes M$ on $\mathscr{L}^2(N)\otimes \mathscr{L}^2(M)$ is given by $\pi^{N\rtimes M}_r(\alpha(x))= \pi_r(x)\otimes 1$ and $\pi^{N\rtimes M}_r(1\otimes m)=U(1\otimes \widehat{\pi}_r(m))U^*$, a fact which is easy to recover using that $U=J_{N\rtimes M}(J_N\otimes \widehat{J})$. Now use that also $U=(\widehat{\pi}_l\otimes \iota)(V)$, that $V$ is the left multiplicative unitary for $(\widehat{M}',\widehat{\Delta}')$, and that $V(J\otimes \widehat{J})=(J\otimes \widehat{J})V^*$.\end{proof}

\noindent Note that $\mathscr{L}^2(N)$ is a natural right $\widehat{M}$-module, by an anti-representation $\widehat{\pi}_r:m\rightarrow \widehat{\pi}_l(\widehat{J}m^*\widehat{J})$ for $m\in \widehat{M}$. (This is \emph{not} the same map as the one above the previous lemma. However, no ambiguities should arise by the use of this double notation (of which the first one will not get used much anyhow). For example, the two maps agree on their common domain when $(N,\alpha)=(M,\Delta)$.) Denote by $\widehat{Q}$ the linking algebra between the right $\widehat{M}$-modules $\mathscr{L}^2(M)$ and $\mathscr{L}^2(N)$. We will write $\widehat{Q}=\left(\begin{array}{cc} \widehat{Q}_{11} & \widehat{Q}_{12}\\
\widehat{Q}_{21} & \widehat{Q}_{22}\end{array}\right)$ as well as $\widehat{Q}=\left(\begin{array}{cc} \widehat{P} & \widehat{N}\\
\widehat{O} & \widehat{M}\end{array}\right)$.

\begin{Cor}\label{corr}\begin{enumerate} \item  $\tilde{G}\in \widehat{O}\otimes N$.
\item $\tilde{G}_{12} U_{13} = V_{13}\tilde{G}_{12}$.\end{enumerate}
\end{Cor}

\begin{proof} The first statement follows by the second and third commutation relation in the previous lemma. Since for $\omega \in M_*$, we have $(\iota\otimes \omega)(U)=\widehat{\pi}_l((\iota\otimes \omega)(V))$, the second statement also follows from the second commutation relation of the previous lemma. \end{proof}

\noindent The following is just a restatement of Lemma \ref{lem26}.

\begin{Lem} \label{lem2} The map $\tilde{G}$ satisfies the
identity $\tilde{G}(J_N\otimes J_N)\Sigma = \Sigma U\Sigma
(\widehat{J}\otimes J_N)\tilde{G}$.\end{Lem}

\noindent Now we prove a kind of pentagon equation:

\begin{Prop}\label{lem4}
\[\widehat{W}_{12}\tilde{G}_{13}\tilde{G}_{23}= \tilde{G}_{23}\tilde{G}_{12}.\]
\end{Prop}
\begin{proof}
\noindent For $x\in \mathscr{N}_{\varphi_N}$ and
$\omega\in B(\mathscr{L}^2(N))_*$, we have $(\omega\otimes \iota)(\alpha(x)) \in
\mathscr{N}_{\varphi}$,
 and \[(\iota\otimes \omega)(\tilde{G})\Lambda_{N}(x) =
 \Lambda ((\omega\otimes \iota)(\alpha(x))).\]  This follows
  by first considering $\omega$ of the form $\omega_{\Lambda_{N}(y),\Lambda_{N}(z)}$ with $y,z\in \mathscr{N}_{\varphi_N}$,
   and then using the closedness of the map $\Lambda$ to conclude that it holds in general.
    Now for $x\in \mathscr{N}_{\varphi_N}$, $\omega\in M_*$ and $\omega'\in N_*$, we have, using $\widehat{W}=\Sigma W^*\Sigma$,
     \begin{eqnarray*} (\iota\otimes \omega)(\widehat{W})(\iota\otimes \omega')(\tilde{G}) \Lambda_{N} (x)
     &=& \Lambda ((\omega'\otimes \omega\otimes \iota)((\iota\otimes \Delta)(\alpha(x))))\\ &=&
     \Lambda((\omega'\otimes \omega\otimes \iota)((\alpha\otimes \iota)(\alpha(x)))) \\ &= &
     \Lambda((((\omega\otimes \omega')\circ \alpha^{\textrm{op}})\otimes \iota)(\alpha(x)))\\ &=&
     (\iota\otimes ((\omega\otimes \omega')\circ \alpha^{\textrm{op}}))(\tilde{G})\Lambda_N(x),
     \end{eqnarray*} from which we conclude $\widehat{W}_{12}\tilde{G}_{13} = (\iota\otimes \alpha^{\textrm{op}})(\tilde{G})$.
      Since $(\iota\otimes \alpha^{\textrm{op}})(\tilde{G})= \tilde{G}_{23}\tilde{G}_{12} \tilde{G}_{23}^*$,
      the result follows. \\

\end{proof}
\noindent \emph{Remarks:} \begin{enumerate} \item Note that if $N$ and $M$ have separable preduals, then, choosing a unitary $u:\mathscr{L}^2(M)\rightarrow \mathscr{L}^2(N)$, the unitary $v=\tilde{G}(u\otimes 1)$ in $B(\mathscr{L}^2(M))\otimes N$ will satisfy $(\iota\otimes \alpha)(v)=\widehat{W}_{13}v_{12}$. So in this case there is a one-to-one correspondence between Galois objects and ergodic semi-dual coactions.
\item Note that for the trivial right Galois object $(M,\Delta)$ for $(M,\Delta)$, the map $\tilde{G}$ is exactly $\widehat{W}$, while the map $U$ becomes the right regular representation $V$.\\\end{enumerate}

\noindent We have the following density results:

\begin{Prop}\label{prop412} \begin{enumerate} \item The space $L=\{(\omega\otimes \iota)(\tilde{G})\mid \omega\in B(\mathscr{L}^2(N),\mathscr{L}^2(M))_*\}$ is $\sigma$-weakly dense in $N$.
\item  The space $K=\{(\iota\otimes \omega)(\tilde{G})\mid \omega\in B(\mathscr{L}^2(N))_*\}$ is $\sigma$-weakly dense in $\widehat{O}$.
\end{enumerate}\end{Prop}

\begin{proof}

\noindent By the pentagon equation, the linear span of the $(\omega\otimes \iota)(\tilde{G})$ will be an algebra. Further, for any $x\in \mathscr{N}_{\varphi_N}$ and $m\in \mathscr{N}_{\varphi}$, we have $(1\otimes m^*)\alpha(x) \in \mathscr{M}_{(\iota\otimes \varphi)}$ and $(\omega_{\Lambda_N(x),\Lambda(m)}\otimes \iota)(\tilde{G})=(\iota\otimes \varphi)((1\otimes m^*)\alpha(x))$. From this, we can conclude that the $\sigma$-weak closure of $L$ also is the $\sigma$-weak closure of the span of
$\{(\iota\otimes \omega)(\alpha(x))\mid \omega\in
M_*, x\in N \}$, so that this $\sigma$-weak closure will be a unital sub-von Neumann algebra of $N$ (see also the proof of Proposition 1.21 of \cite{VDae2}). Now suppose $\omega\in N_*$ is orthogonal to $L$. By the bi-duality
theorem (see \cite{Eno2}, and also Theorem  2.6 of \cite{Vae1}), we have that $(\alpha(N)\cup (1\otimes
B(\mathscr{L}^2(M))))''$ equals $N\otimes B(\mathscr{L}^2(M))$. So for any $x\in N\otimes B(\mathscr{L}^2(N))$ and $\omega'\in B(\mathscr{L}^2(N))_*$, $(\iota\otimes \omega')(x)$ can be $\sigma$-weakly approximated by elements of the form $(\iota\otimes \omega')(x_n)$ with $x_n$ in the algebra generated by $\alpha(N)$ and $1\otimes B(\mathscr{H})$, and any such element can in turn be approximated by an element in the algebra generated by elements of the form $(\iota\otimes \omega'')(\alpha(x_{nm}))$, $\omega''\in B(\mathscr{L}^2(M))_*$ and $x_{nm}\in N$, by using an orthogonal basis argument. It follows that $\omega$ vanishes on the whole of $N$, and hence $L$ is $\sigma$-weakly dense in $N$.\\

\noindent For the second statement, note that, by the pentagon equation, $K$ is closed under left multiplication with elements of the form $(\iota\otimes \omega)(\widehat{W})$ for $\omega\in M_*$. Hence, as in the proof of Proposition \ref{tex11}, it is enough to show that if $z\in \widehat{N}$ satisfies $K\cdot z= 0$, then $z= 0$. But take $x,y\in \mathscr{T}_{\varphi_N,T}$ (which is now just the Tomita algebra $\mathscr{T}_{\varphi_N}$ for $\varphi_N$), and $m\in \mathscr{N}_{\widehat{\varphi}^{\textrm{op}}}$. Then $(\iota\otimes \omega_{\Lambda_N(x),\Lambda_N(y)})(\tilde{G}^*)\widehat{\Lambda}^{\textrm{op}}(m) = \widehat{\pi}_l(m) x\Lambda_N(\sigma_{-i}(y ^*))$ by  Lemma \ref{lem12} and Proposition \ref{prop2222}. Hence $K^* \cdot \mathscr{L}^2(M)$ is dense in $\mathscr{L}^2(N)$, and necessarily $z=0$.
\end{proof}

\begin{Prop} For any $m\in M'$, the operator $\tilde{G}^*(m\otimes 1)\tilde{G}$ lies in $N'\otimes N$.\end{Prop}

\begin{proof} Clearly, the second leg lies in $N$.
Since $\tilde{G}(y\otimes 1)\tilde{G}^*= \alpha^{\textrm{op}}(y)$ for
$y\in N$, the first leg of $\tilde{G}^*(m\otimes 1)\tilde{G}$ must
be inside $N'$.
\end{proof}

\noindent Recall that $\nabla_{N\rtimes M}^{it}=\nabla_N^{it}\otimes q^{it}$, where we can also write $q^{it}=\delta^{-it}\widehat{\nabla}^{-it}$. Then $\kappa_t=q^{it}xq^{-it}$ defines a one-parametergroup of automorphisms on $M$, and $\gamma_t(x)=q^{it}xq^{-it}$ defines a one-parametergroup of automorphisms on $\widehat{M}'$.

\begin{Lem}\label{lem3} \begin{enumerate}\item  For $x\in N$,
we have $\alpha(\sigma_t^N(x))=
 (\sigma_t^N\otimes \kappa_t)(\alpha(x))$.
 \item  For $m\in \widehat{M}'$, we have $\sigma_t^{N_2}(\widehat{\pi}_l(m))= \widehat{\pi}_l(\gamma_t(m))$.
\item For $m\in \widehat{M}'$, we have $\widehat{\theta}_r(m)= \widehat{\pi}_l(\widehat{R}'(m))$. \end{enumerate}\end{Lem}

\begin{proof} The first two statements follow straightforwardly from Lemma \ref{lem26} and Lemma \ref{com1}. The final statement was noted at the beginning of this section.
 \end{proof}

\noindent In particular, $\sigma_t^{N_2}(\widehat{\pi}_l(\widehat{J}\widehat{\delta}^{is}\widehat{J})) = \widehat{\pi}_l(\widehat{J}\widehat{\delta}^{is}\widehat{J})$ for each $s,t\in \mathbb{R}$, since an easy computation shows that each $\widehat{J}\widehat{\delta}^{is}\widehat{J}$ is invariant under $\gamma_t$.

\begin{Cor} The one-parametergroups $\nabla_N^{it}$ and
 $\widehat{\pi}_l(\widehat{J}\widehat{\delta}^{it}\widehat{J})$ commute.\end{Cor}

 \noindent We denote the resulting one-parametergroup of unitaries by
 $P_N^{it}=\nabla_N^{it}\widehat{\pi}_l(\widehat{J}\widehat{\delta}^{-it}\widehat{J})$.

\begin{Prop}  $N$ is invariant under Ad$(P_N^{it})$.\end{Prop}

\begin{proof} We only have to show that $N$ is invariant under
Ad$(\widehat{\pi}_l(\widehat{J}\widehat{\delta}^{-it}\widehat{J}))$. But for any group-like
element $u\in \widehat{M}'$, we have, denoting by $\widehat{\alpha}$ the
dual coaction, that \begin{eqnarray*}
((\rho\otimes \iota)\widehat{\alpha} \rho^{-1})(\widehat{\pi}_l(u)x\widehat{\pi}_l(u)^*)&=& (\widehat{\pi}_l(u)\otimes u)(x\otimes 1)(\widehat{\pi}_l(u)^*\otimes u^*)\\ &=& \widehat{\pi}_l(u)x\widehat{\pi}_l(u)^*\otimes 1 \end{eqnarray*} for
$x\in N$, and so, by the bi-duality theorem of \cite{Eno2}, we get
$\widehat{\pi}_l(u)x\widehat{\pi}_l(u)^*\in N$. \end{proof}

\begin{Def} We call the resulting one-parametergroup \[\tau^N_t:N\rightarrow N: x\rightarrow P_N^{it}xP_N^{-it}\] the \emph{scaling group} of $(N,\alpha)$.\end{Def}

\begin{Prop}\label{com2} The following identities hold for $x\in N$:
\[\alpha(\tau_t^{N}(x)) = (\tau_t^{N}\otimes \tau_t)(\alpha(x)),\]
\[\alpha(\tau_t^{N}(x))= (\sigma_t^N\otimes \sigma'_{-t})(\alpha(x)),\]
\[\alpha(\sigma_t^N(x)) = (\tau_t^N\otimes \sigma_t)(\alpha(x)).\]
\end{Prop}

\noindent Recall that $\tau_t$ denotes the scaling group of $(M,\Delta)$, while $\sigma'_t$ denotes the modular one-parametergroup of the right invariant weight $\psi$.

\begin{proof} By Lemma \ref{lem3}, we have \[\alpha\circ \sigma_t^{N}=(\sigma_t^{N} \otimes \textrm{Ad}(\delta^{-it}) \tau_{-t})\circ \alpha.\] Further, since $G$ is a left $N\rtimes M$-module map, we have \begin{eqnarray*} \alpha( \textrm{Ad}(\widehat{\pi}_l(\widehat{J}\widehat{\delta}^{-it}\widehat{J}))(x)) &=& G ((\textrm{Ad}(\widehat{\pi}_l(\widehat{J}\widehat{\delta}^{-it}\widehat{J}))(x))\otimes 1)G^* \\ &=& (\iota\otimes \textrm{Ad}(\widehat{\pi}_l(\widehat{J}\widehat{\delta}^{-it}\widehat{J})))(G(x\otimes 1)G^*)\\ &=& (\iota\otimes \textrm{Ad}(\widehat{\pi}_l(\widehat{J}\widehat{\delta}^{-it}\widehat{J})))(\alpha(x)).\end{eqnarray*} Now by the first formula of Theorem 4.17 in \cite{VDae2}, we have $(\widehat{J}\widehat{\delta}^{-it}\widehat{J})P^{-it}=\nabla^{-it}$, where $P^{it}$ denotes the standard unitary implementation of the scaling group of $(M,\Delta)$, so $\textrm{Ad}(\delta^{-it}) \tau_{-t}\textrm{Ad}(\widehat{J}\widehat{\delta}^{-it}\widehat{J})$ reduces to $\sigma'_{-t}$ on $M$. This proves the second formula. \\

\noindent As for the first identity, we have, using the second identity, the coaction property of $\alpha$ and the identity $\Delta\circ \sigma'_{-t} = (\sigma_{-t}'\otimes \tau_t)\circ \Delta$, that \begin{eqnarray*} (\alpha\otimes \iota)\circ (\tau_t^N\otimes \tau_t)\circ \alpha &=& (\sigma^N_t\otimes \sigma_{-t}'\otimes \tau_t)\circ (\iota\otimes \Delta)\circ \alpha \\ &=& (\iota\otimes \Delta)\circ (\sigma^N_t\otimes \sigma_{-t}')\circ \alpha \\ &=& (\alpha\otimes \iota)\circ \alpha\circ \tau_t^N.\end{eqnarray*} Thus the first identity follows by the injectivity of $\alpha$.\\

\noindent The third identity now easily follows from the first identity and \[\alpha(\textrm{Ad}(\widehat{\pi}_l(\widehat{J}\widehat{\delta}^{it}\widehat{J}))(x))= (\iota\otimes \sigma_{t}\tau_{-t})(\alpha(x))\] for $x\in N$.

\end{proof}

\noindent For the next result, recall that $\nu>0$ denotes the scaling constant of $(M,\Delta)$.

\begin{Lem}\label{lem5} The one-parametergroup $\tau_t^N$ satisfies
$\varphi_N\circ \tau_t^N= \nu^t \varphi_N$, and if $x\in
\mathscr{N}_{\varphi_N}$, then \[P_N^{it}\Lambda_N(x)=
\nu^{t/2}\Lambda_{N}(\tau_t^N(x)).\] \end{Lem}

\begin{proof} The first statement easily follows since
\begin{eqnarray*} \varphi_N\circ \tau_t^N &=&
((\iota\otimes \varphi)\circ \alpha)\circ \tau^N_t\\ &=&
\tau_t^N\circ ((\iota\otimes \varphi\circ \tau_t)\circ \alpha)\\ &=& \nu^{t}
\varphi_N.\end{eqnarray*}

\noindent By the first statement, $\textrm{Ad}(\widehat{\pi}_l(\widehat{J}\widehat{\delta}^{it}
\widehat{J}))(x) \in \mathscr{N}_{\varphi_N}$ for $x\in \mathscr{N}_{\varphi_N}$, and the second statement is equivalent with \[\nu^{t/2}
\widehat{\pi}_l(\widehat{J}\widehat{\delta}^{it} \widehat{J})\Lambda_{N} (x) =
\Lambda_{N} (\textrm{Ad}(\widehat{\pi}_l(\widehat{J}\widehat{\delta}^{it}
\widehat{J}))(x)).\] Taking an
arbitrary $y\in \mathscr{N}_{\varphi_N}$, we have \[G
(\nu^{t/2} \widehat{\pi}_l(\widehat{J}\widehat{\delta}^{it}
\widehat{J})\Lambda_{N} (x)\otimes \Lambda_{N}(y)) =
\nu^{t/2} (1\otimes\widehat{J}\widehat{\delta}^{it}
\widehat{J})( \Lambda_N\otimes \Lambda)
(\alpha(x)(y\otimes 1)).\] Since $\widehat{J}\widehat{\delta}^{it}\widehat{J}=\nabla^{it}P^{-it}$ and
$\alpha( \textrm{Ad}(\widehat{\pi}_l(\widehat{J}\widehat{\delta}^{it}
\widehat{J}))(x))=
(\iota\otimes \textrm{Ad}(\widehat{J}\widehat{\delta}^{it}
\widehat{J}) )(\alpha(x))$, the result
follows.\end{proof}

\begin{Cor}\label{cor1} We have the following commutation relations: \begin{enumerate} \item $\tilde{G}(\nabla_N^{it}\otimes \nabla_N^{it}) = (\delta^{-it}\widehat{\nabla}^{-it}\otimes \nabla_N^{it})\tilde{G}$,
\item $\tilde{G}(\nabla_N^{it}\otimes P_N^{it}) = (\nabla^{it}\otimes P_N^{it})\tilde{G}$,
\item $\tilde{G}(P_N^{it}\otimes P_N^{it}) = (P^{it}\otimes P_N^{it})\tilde{G}$.\end{enumerate}\end{Cor}
\begin{proof} The first identity follows immediately from Lemma \ref{lem26}, while the other two follow by using the definition of $G$, the implementation of Lemma \ref{lem5} and the identities in Lemma \ref{com2}.\end{proof}

\noindent Now consider $H^{it}=\tilde{G}^{*}(J\delta^{it}J\otimes
1)\tilde{G}$ in $N'\otimes N$.

\begin{Prop}\label{lem16} There exist non-singular $h,k\geq 0$ affiliated with respectively $N'$ and $N$ such that
 $H^{it}=h^{it}\otimes k^{it}$ for all $t\in \mathbb{R}$. Moreover, $\alpha(k^{it})=k^{it}\otimes \delta^{it}$ for $t\in \mathbb{R}$.\end{Prop}

\begin{proof} We show that $H^{it}(B(\mathscr{L}^2(N))\otimes 1)H^{-it}= B(\mathscr{L}^2(N))\otimes 1
$. Since $B(\mathscr{L}^2(N))=\rho(N\rtimes M)$, we only have to
show
 that $H^{it}(N\otimes 1)H^{-it}=(N\otimes 1)$ and $H^{it}(\widehat{\pi}_l(\widehat{M}')\otimes 1)H^{-it}=(\widehat{\pi}_l(\widehat{M}')\otimes 1)$.
 Now the first equality is clear as the first leg of $H^{it}$ lies in $N'$. As for the
 second equality, applying $\tilde{G}(\cdot)\tilde{G}^*$, this is equivalent
 with Ad$(J\delta^{it}J)(\widehat{M}')=\widehat{M}'$, which is easily seen to be true. \\

\noindent Denote by $h$ a positive operator which implements the
automorphism group Ad$(H^{it})$ on $B(\mathscr{L}^2(N))$: Ad$(H^{it})(x\otimes 1)=($Ad$(h^{it})(x))\otimes 1$ for all $x\in B(\mathscr{L}^2(N))$. Then $h$ is non-singular, with $h$ affiliated with $N'$, and $H^{it}=h^{it}\otimes k^{it}$ for a positive non-singular $k$ affiliated with $N$. \\

\noindent Note now that $\widehat{W}^*(J\delta^{it}J\otimes 1)\widehat{W} = J\delta^{it}J\otimes \delta^{it}$, which can be computed for example by Lemma 4.14 and the formulas in Proposition 4.17 of \cite{VDae2}. Then using the pentagon equation for $\tilde{G}$, we have \begin{eqnarray*} (\iota \otimes \alpha^{\textrm{op}})(H^{it}) &=& \tilde{G}_{23}H_{12}^{it}\tilde{G}_{23}^* \\ &=& \tilde{G}_{23}\tilde{G}_{12}^*(J\delta^{it}J\otimes 1\otimes 1)\tilde{G}_{12}\tilde{G}_{23}^*\\ &=& \tilde{G}_{13}^*\widehat{W}_{12}^*\tilde{G}_{23}(J\delta^{it}J\otimes 1\otimes 1)\tilde{G}_{23}^*  \widehat{W}_{12}\tilde{G}_{13}\\ &=& \tilde{G}_{13}^*(J\delta^{it}J\otimes \delta^{it}\otimes 1) \tilde{G}_{13}\\ &=& h^{it} \otimes \delta^{it}\otimes k^{it},\end{eqnarray*} so that $\alpha(k^{it})= k^{it}\otimes \delta^{it}$.

\end{proof}

\noindent The operator $k$ which appears in the proposition is determined up to a positive scalar. We will now fix some $k$, and call it $\delta_N$.

\begin{Def} We call $\delta_N$ the \emph{modular element} of $(N,\alpha)$.\end{Def}

\begin{Lem} \label{lem8} With the notation of the previous proposition, we have \begin{enumerate}
\item  $h=J_{N}\delta_N^{-1}J_{N}$,
\item
$\sigma_t^{N}(\delta_N^{is})=\nu^{ist}\delta_N^{is}$,
\item $\tau_t^N(\delta_N^{is})=\delta_N^{is}$. \end{enumerate}\end{Lem}

\begin{proof} Denoting again $H^{it}=\tilde{G}^{*}(J\delta^{it}J\otimes 1)\tilde{G}$, we first
prove that \[\Sigma (J_{N}\otimes
J_{N})H^{it}(J_{N}\otimes J_{N})\Sigma =
H^{it}.\] Using Lemma \ref{lem2}, the left hand side equals
\[\tilde{G}^* (\widehat{J}\otimes J_N)\Sigma U^*\Sigma
(J\delta^{it}J\otimes 1)\Sigma U\Sigma (\widehat{J}\otimes J_N)\tilde{G}.\]
As $U\in B(\mathscr{L}^2(N))\otimes M$, this reduces to
$\tilde{G}^*(\widehat{J}J\delta^{it}J\widehat{J}\otimes 1)\tilde{G}$.
 Since $J$ commutes with $\widehat{J}$ up to a scalar of modulus 1, and since $\delta^{it}$ commutes with $\widehat{J}$, we find that this expression reduces to
  $\tilde{G}^{*}(J\delta^{it}J\otimes 1)\tilde{G}=H^{it}$. So  \[J_N\delta_N^{it}J_N\otimes J_Nh^{it}J_N =h^{it}\otimes \delta_N^{it}, \] which implies that there exists a positive scalar $r$ such
that $h^{it} = r^{it} J_N\delta_N^{it} J_N$. But plugging this back into the above
equality, we find that $r^{2it}=1$ for all $t$, hence $r=1$. \\

\noindent  For the second statement, we easily get, using the first commutation relation of Corollary \ref{cor1}, that
   \[(\nabla_N^{it} \otimes \nabla_N^{it})(J_N\delta_N^{is}J_N \otimes \delta_N^{is})(\nabla_N^{-it}\otimes \nabla_N^{-it})=
   (J_N\delta_N^{is}J_N \otimes \delta_N^{is}). \] This
   implies that there exists a positive number $\tilde{\nu}$
   such that $\sigma_t^{N}(\delta_N^{is})= \tilde{\nu}^{ist} \delta_N^{is}$. We
   must show that $\tilde{\nu}=\nu$. \\

\noindent But we know now that $\delta_N^{is}$ is analytic with
respect to $\sigma_t^{N}$. So if $x\in
\mathscr{M}_{\varphi_N}$, then also $x\delta_N^{is}$ and
$\delta_N^{is}x$ are integrable.
 We have for such $x$ that, choosing some
 state $\omega\in N_*$, \begin{eqnarray*} \varphi_N(\delta_N^{is} x)
 &=& \varphi( (\omega\otimes \iota)(\alpha(\delta_N^{is}x)))\\
 &=& \varphi(\delta^{is} (\omega(\delta_N^{is}\cdot)\otimes \iota)(\alpha(x)))\\
 &=& \nu^{s} \varphi((\omega(\delta_N^{is}\cdot)\otimes \iota)(\alpha(x))\delta^{is} )
 \\ &=& \nu^{s} \varphi((\omega(\delta_N^{is}\cdot \delta_N^{-is})\otimes \iota)(\alpha(x\delta_N^{is}) ))\\
 &=& \nu^{s}\varphi_N(x\delta_N^{is}).\end{eqnarray*}  This shows $ \sigma_{-i}(\delta_N^{is})= \nu^s \delta_N^{is}$, which implies $\tilde{\nu} = \nu$. \\

\noindent As for the last statement, this follows from \begin{eqnarray*} \alpha(\tau_t^N\sigma_{-t}^N(\delta_N^{is})) &=& (\iota\otimes \tau_t\sigma_{-t})\alpha(\delta_N^{is})\\ &=& \delta_N^{is}\otimes \tau_t\sigma_{-t}(\delta^{is})\\ &=& \nu^{-ist} \alpha(\delta_N^{is}).\end{eqnarray*}

\end{proof}

\noindent By Connes' cocycle derivative theorem, we can now make the nsf weight
$\psi_N=\varphi_N(\delta_N^{1/2}\cdot \,\delta_N^{1/2})$, which is the deformation of $\varphi_N$ by the cocycle $w_t = \nu^{it^2/2}\delta_N^{it}$.

\begin{Theorem}\label{proprinv} The weight $\psi_N$ is invariant with respect to $\alpha$:
if $x\in \mathscr{M}_{\psi_N}^+$ and $\omega\in M_*^+$, then
\[\psi_N((\iota\otimes \omega)\alpha(x)) = \psi_N(x)\omega(1).\]
\end{Theorem}

\begin{proof}

\noindent Let $x\in N$ be a left multiplier of $\delta_N^{1/2}$ such that the closure of $x\delta_N^{1/2}$ is an element of $\mathscr{N}_{\varphi_N}$. Then $x\in \mathscr{N}_{\psi_N}$, and $\Lambda_{N,\delta_N}(x):=\Lambda_N(\overline{x\delta_N^{1/2}})$ determines a GNS-construction for $\psi_N$ (see the remark before Proposition 1.15 in \cite{Kus1}). Choose $\xi\in \mathscr{D}(\delta^{-1/2})$. Then for any $\eta\in \mathscr{L}^2(M)$, we have $(\iota\otimes\omega_{\xi,\eta})\alpha(x)$ a left multiplier of $\delta_N^{1/2}$, and the closure of $((\iota\otimes\omega_{\xi,\eta})\alpha(x))\delta_N^{1/2}$ equals $(\iota\otimes \omega_{\delta^{-1/2}\xi,\eta})\alpha(\overline{x\delta_N^{1/2}})$. By the formula (\ref{form}) for $U$ (after the statement of Theorem \ref{Theo1}), we conclude that this last operator is in $\mathscr{D}(\Lambda_N)$, and that its image under $\Lambda_N$ equals $(\iota\otimes \omega_{\xi,\eta})(U)\Lambda_{\psi_N}(x)$. Then by the closedness of $\Lambda_{N,\delta_N}$, we can conclude that for $x$ of the above form, $(\iota\otimes\omega)(\alpha(x))\in \mathscr{D}(\Lambda_{N,\delta_N})$ for every $\omega\in M_*$, with \[\Lambda_{N,\delta_N} ((\iota\otimes\omega)(\alpha(x))) = (\iota\otimes \omega)(U)\Lambda_{N,\delta_N}(x).\] Since such $x$ form a $\sigma$-strong-norm core for $\Lambda_{N,\delta_N}$, the same statement holds for a general $x\in \mathscr{N}_{N,\delta_N}$. From this, it is standard to conclude the invariance: take $\omega=\omega_{\xi,\xi}\in M_*^+$ and $x=y^*y\in
\mathscr{M}_{\psi_N}^+$. Let $\xi_i$ denote an orthonormal basis for
$\mathscr{L}^2(M)$. Then by the lower-semi-continuity of $\psi_N$,
we find
\begin{eqnarray*} \psi_N((\iota\otimes \omega_{\xi,\xi})(\alpha(y^*y))) &=&
\psi_N (\sum_n (\iota\otimes \omega_{\xi,\xi_n})(\alpha(y))^*((\iota\otimes \omega_{\xi,\xi_n})(\alpha(y))\\
&=& \sum_n \psi_N((\iota\otimes
\omega_{\xi,\xi_n})(\alpha(y))^*((\iota\otimes
\omega_{\xi,\xi_n})(\alpha(y)))\\ &=&
\sum_n \|\Lambda_{N,\delta_N}((\iota\otimes \omega_{\xi,\xi_n})(\alpha(y))\|^2 \\
&=& \sum_n \| (\iota\otimes
\omega_{\xi,\xi_n})(U)\Lambda_{N,\delta_N}(y))\|^2\\ &=& \langle
\Lambda_{N,\delta_N}(y),(\sum_n ((\iota\otimes
\omega_{\xi_n,\xi})(U^*)(\iota\otimes
\omega_{\xi,\xi_n})(U))\Lambda_{N,\delta_N}(y)\rangle
\\ &=& \langle \Lambda_{N,\delta_N}(y),(\iota\otimes \omega_{\xi,\xi})(U^*U)\Lambda_{N,\delta_N}(y)\rangle \\
&=& \psi_N(y^*y)\omega_{\xi,\xi}(1),\end{eqnarray*} hence
$\psi_N((\iota\otimes \omega)(\alpha(x)))= \psi_N(x)\omega(1)$.\\

\end{proof}

\noindent \textit{Remark:} It is natural to ask if there is a corresponding result for general Galois coactions. We briefly show that one can not expect \emph{too} much: for general Galois coactions, there does not have to exist an invariant nsf operator valued weight $T_{\psi_N}$, i.e. an operator valued weight $N^+\rightarrow (N^{\alpha})^{+,\textrm{ext}}$ such that $T_{\psi_N}((\iota\otimes \omega)\alpha(x))=\omega(1)T_{\psi_N}(x)$ for $\omega\in M_*^+$ and $x\in \mathscr{M}_{T_{\psi_N}}^+$. To give an explicit example, suppose $\alpha$ is an outer left coaction of a locally compact quantum group $(M,\Delta)$ on a factor $N$. Then by outerness, there is a \emph{unique} nsf operator valued weight $(N\rtimes M)^+\rightarrow \alpha(N)^{+,\textrm{ext}}$ (up to a scalar), namely $(\iota\otimes \widehat{\varphi})\widehat{\alpha}$, where $\widehat{\alpha}$ is the dual right coaction. But if $(\widehat{M},\widehat{\Delta})$ is not unimodular, then this operator valued weight is not invariant. On the other hand, this does \emph{not} rule out the possibility that there exists an invariant nsf \emph{weight}: for if the original coaction has an invariant nsf weight $\psi_N$ (for example, the coactions occurring in \cite{Vae6}), then one checks that $x\in (N\rtimes M)^+ \rightarrow \psi_{\widehat{M}'}((\psi_N\otimes \iota)(x)) \in \lbrack 0,+\infty\rbrack$ is a well-defined $\widehat{\alpha}$-invariant nsf weight on $N\rtimes M$. We do not know of any example of a Galois coaction without invariant weights.\\

\begin{Prop} \label{lem10} Denote $\widehat{\nabla}_N^{it} = P_N^{it}J_N\delta_N^{it}J_N$.
Then $\widehat{\nabla}_N^{-it}\widehat{\pi}_l(m)\widehat{\nabla}_N^{it} =
\widehat{\pi}_l(\sigma_t^{\widehat{\varphi}^{op}}(m))$ for $m\in \widehat{M}'$.
\end{Prop}

\begin{proof} First note that $\widehat{\nabla}_N^{it}$ is well-defined,
since $P_N^{it}$ is easily seen to commute with $J_N$ and $\delta_N^{it}$. Then also
$\widehat{\nabla}_N^{it} \Lambda_{\psi_N}(x) =
\Lambda_{\psi_N}(\tau_t^N(x)\delta_N^{-it})$ for $x\in \mathscr{N}_{\psi_N}$, by an easy adjustment
of Lemma  \ref{lem5} and using the relative invariance property of
$\delta_N^{it}$. If we apply $(\iota\otimes \omega)(U)$ to this with
$\omega\in M_*$, then, using the commutation rules between $\alpha$,$\tau_t^N$ and $\delta_N^{it}$, we get \[(\iota\otimes \omega)(U) \widehat{\nabla}_N^{it} \Lambda_{\psi_N}(x)= \Lambda_{\psi_N}
(\tau_{t}^N((\iota\otimes
\omega(\tau_t(\cdot)\delta^{-it}))\alpha(x))\delta_N^{-it}).\] This
shows  \[\widehat{\nabla}_N^{-it}\widehat{\pi}_l((\iota\otimes
\omega)(V))\widehat{\nabla}_N^{it} =  \widehat{\pi}_l((\iota\otimes
\omega(\tau_t(\cdot)\delta^{-it}))(V)).\] But this is exactly
$\widehat{\pi}_l(\sigma_t^{\widehat{\varphi}^{\textrm{op}}}((\iota\otimes \omega)(V)))$. Then of course the same holds with $(\iota\otimes \omega)(V)$ replaced by a general element of $\widehat{M}'$, thus proving the proposition.

\end{proof}

\begin{Prop} \label{lem6} \label{lem7} The following commutation relations hold: \begin{enumerate}\item $(\nabla^{it}\otimes \widehat{\nabla}_N^{it})\tilde{G}= \tilde{G}(\nabla_N^{it}\otimes \widehat{\nabla}_N^{it})$,

\item $(\widehat{\nabla}^{it}\otimes P_N^{it})\tilde{G}= \tilde{G}(\widehat{\nabla}_N^{it}\otimes P_N^{it}\delta_N^{it}).$\end{enumerate}
\end{Prop}

\begin{proof}

\noindent The first formula follows by the second formula in Corollary \ref{cor1}, and the fact that the second leg of $\tilde{G}$ lies in $N$. The second formula follows from the fact that also $\widehat{\nabla}^{it}= J\delta^{it}JP^{it}$, then using the third formula of Corollary \ref{cor1} and the first formula in Lemma \ref{lem8} together with the definition of $\delta_N$.

\end{proof}

\begin{Theorem}\label{propi} Up to a positive constant,  $\psi_N$ is the only invariant, and $\varphi_N$ the only $\delta$-invariant weight on $N$.\end{Theorem}

\begin{proof} The claim about $\varphi_N$ follows immediately by Lemma 3.9 of \cite{Vae1} and the fact that $\alpha$ is ergodic. The second statement can be proven in the same fashion.

\end{proof}

\noindent Before going over to the next section, we remark that of course all results hold as well  in the context of \emph{left Galois coactions}: if $(P,\Delta_P)$ is a locally compact quantum group, $N$ a von Neumann algebra, and $\gamma$ an integrable ergodic left coaction of $(P,\Delta_P)$ on $N$, we call $(N,\gamma)$ a \emph{left Galois object} if, with $\psi_N=(\psi\otimes \iota)\gamma$, the \emph{Galois map} \[\tilde{H}:\mathscr{L}^2(N)\otimes \mathscr{L}^2(N)\rightarrow \mathscr{L}^2(N)\otimes \mathscr{L}^2(P)\]\[ \Lambda_{\psi_N}(x)\otimes \Lambda_{\psi_N}(y)\rightarrow (\Lambda_{\psi}\otimes \Lambda_{\psi_N})(\gamma(x)(1\otimes y)),\quad x,y\in \mathscr{N}_{\psi_N},\] is a unitary. We will therefore use the proper analogous statements of this section in the left context without further proof. \\

\section{Reflecting across a Galois object}

\noindent In this section, we will construct another locally compact
quantum group given a Galois object $(N,\alpha)$ for a locally
compact quantum group $(M,\Delta)$. In fact, the new quantum group
will be a corner of a special kind of quantum groupoid, with
$(\widehat{M},\widehat{\Delta})$ in the other corner. This quantum groupoid picture turns out to be very useful, providing one with the right intuition on how to proceed. We use notation as in the previous section. For convenience, we will now treat also $\mathscr{L}^2(M)\otimes \mathscr{L}^2(N)$ as an $N\rtimes M$-bimodule (by applying Ad$(\Sigma)$ to the previous representations), so that we can call $\tilde{G}$ a bimodule map. \\

\noindent Denote as before by $\widehat{Q}=\left(\begin{array}{cc} \widehat{Q}_{11} & \widehat{Q}_{12}\\
\widehat{Q}_{21} & \widehat{Q}_{22}\end{array}\right)=\left(\begin{array}{cc} \widehat{P} & \widehat{N}\\
\widehat{O} & \widehat{M}\end{array}\right)$ the linking algebra between the right $\widehat{M}$-modules $\mathscr{L}^2(M)$ and $\mathscr{L}^2(N)$ (see the remark before Corollary \ref{corr}). We will sometimes denote the natural inclusion $\widehat{Q}\subseteq B\left(\begin{array}{c} \mathscr{L}^2(N) \\ \mathscr{L}^2(M)\end{array}\right)$ by $\pi^{\widehat{Q},2}= (\pi_{ij}^2)_{i,j}$ for emphasis. We will identify the $\widehat{Q}_{ij}$ with their parts in $\widehat{Q}$ (so for example if $x\in \widehat{Q}_{12}$, we identify it with $\left(\begin{array}{cc} 0 & x\\
0 & 0\end{array}\right)$), \textit{except} that we will write the unit 1 of $\widehat{Q}_{22}=\widehat{M}$ as $1_{\widehat{M}}$ when we see it as a projection in $\widehat{Q}$ (likewise for $\widehat{Q}_{11}=\widehat{P}$). As before, we denote the right $\widehat{M}$-module structure on $\mathscr{L}^2(N)$ by $\widehat{\pi}_r$, i.e. $\widehat{\pi}_r(m)=\widehat{\pi}_l(\widehat{J}m^*\widehat{J})=\rho((1\otimes \widehat{J}m^*\widehat{J}))$ for $m\in \widehat{M}$, where $\rho$ is the Galois homomorphism for $\alpha$. By $\widehat{\pi}_r$, we also denote the map $\widehat{\pi}_r$ w.r.t. the Galois object $(M,\Delta)$, i.e. the standard right representation $\widehat{\pi}_r(m)= \widehat{J}m^*\widehat{J}$, and by $\widehat{\theta}_r$ also the right representation $m\rightarrow \widehat{R}(m)$ of $\widehat{M}$ on $\mathscr{L}^2(M)$. This will, again, not lead to any ambiguities (and, again, we will only use $\widehat{\pi}_r$, in this sense, in the proof of the following lemma).

\begin{Lem}\label{gelem} We have $\tilde{G}^*(1\otimes \widehat{N})\widehat{W}\subseteq \widehat{N} \otimes \widehat{N}$.
\end{Lem}

\noindent \emph{Remark:} By $\widehat{N} \otimes \widehat{N}$, we mean the $\sigma$-weak closure of the algebraic tensor product of $\widehat{N}$ with itself inside $\widehat{Q}\otimes \widehat{Q}$. By the commutation theorem for tensor products of von Neumann algebras, this coincides with the space of intertwiners for the right $\widehat{M}\otimes \widehat{M}$-modules $\mathscr{L}^2(M)\otimes  \mathscr{L}^2(M)$ and $\mathscr{L}^2(N)\otimes \mathscr{L}^2(N)$.

\begin{proof}

\noindent Let $x$ be an element of $\widehat{N}$. As the first leg
of $\widehat{W}$ lies in $\widehat{M}$, and $\tilde{G}$ is a left
$\widehat{M}'$-module morphism, it is clear that for any $m\in
\widehat{M}$, we have \[\tilde{G}^*(1\otimes x)\widehat{W} (\widehat{\pi}_r(m)\otimes 1) =
(\widehat{\pi}_r(m)\otimes 1)\tilde{G}^*(1\otimes x)\widehat{W}.\] On the other
hand, we have to prove that for all $m\in \widehat{M}$,
\begin{equation}\label{eqn4}\tilde{G}^*(1\otimes x)\widehat{W}(1\otimes \widehat{\pi}_r(m)) =
(1\otimes \widehat{\pi}_r(m))\tilde{G}^*(1\otimes x)\widehat{W}.\end{equation} Now
as $\tilde{G}$ is a right $N\rtimes M$-map, we have
\[ (1\otimes \widehat{\pi}_r(m))\tilde{G}^* = \tilde{G}^* ((\widehat{\theta}_r\otimes \widehat{\pi}_r)\widehat{\Delta}(m)),\] using the fourth commutation relation of Lemma \ref{com1} in a slightly adapted form. Since also
\[\widehat{W}(1\otimes \widehat{\pi}_r(m)) =((\widehat{\theta}_r\otimes \widehat{\pi}_r)(\widehat{\Delta}(m))\widehat{W},\] the stated commutation follows from the intertwining property of $x$, as $x\widehat{\pi}_r(m) = \widehat{\pi}_r(m)x$.

\end{proof}

\noindent Denote the corresponding map by
\[\Delta_{\widehat{N}}: \widehat{N}\rightarrow \widehat{N}\otimes
\widehat{N}: x\rightarrow \tilde{G}^*(1\otimes x)\widehat{W}\] Then we can also define
\[\Delta_{\widehat{O}}:\widehat{O}\rightarrow \widehat{O}\otimes
\widehat{O}: x\rightarrow \Delta_{\widehat{N}}(x^*)^*,\] and
\[\Delta_{\widehat{P}}: \widehat{P}\rightarrow \widehat{P}\otimes \widehat{P}:
x\rightarrow \tilde{G}^*(1\otimes x)\tilde{G},\] since
$\widehat{Q}_{21}=(\widehat{Q}_{12})^*$ and the span of
$\widehat{Q}_{12}\widehat{Q}_{21}$ is $\sigma$-weakly dense in $\widehat{P}$. Finally, we denote by $\Delta_{\widehat{Q}}$
the map \[\widehat{Q}\rightarrow \widehat{Q}\otimes \widehat{Q}: x_{ij} \rightarrow
\widehat{\Delta}_{ij}(x_{ij}), \qquad x_{ij}\in \widehat{Q}_{ij},\] where we denote $\widehat{\Delta}_{11}=\Delta_{\widehat{P}}$,... (in the following, we will use both notations without further comment). Then $\Delta_{\widehat{Q}}$ is easily seen to be a
$^*$-homomorphism. \emph{However}, it is \emph{not} unital: $\Delta_{\widehat{Q}}(1_{\widehat{Q}})= 1_{\widehat{M}}\otimes 1_{\widehat{M}}+ 1_{\widehat{P}}\otimes 1_{\widehat{P}}$ does not equal $(1_{\widehat{M}}+1_{\widehat{P}})\otimes (1_{\widehat{M}}+1_{\widehat{P}})=1_{\widehat{Q}\otimes \widehat{Q}}$.

\begin{Lem} The map $\Delta_{\widehat{Q}}$ is coassociative.\end{Lem}
\begin{proof} This follows trivially by Proposition \ref{lem4}.\end{proof}

\noindent Since $J_N \widehat{\pi}_l(m)^*J_N = \widehat{\pi}_l(Jm^*J)$ for $m\in \widehat{M}'$,
we can define an anti-$^*$-isomorphism $R_{\widehat{Q}}:\widehat{Q}\rightarrow \widehat{Q}$ by
sending $x\in \widehat{Q}_{12}$ to $(J_NxJ)^*$, and then extending it in the
natural
way. \\

\begin{Lem}\label{lem14} We have $\Delta_{\widehat{Q}}(R_{\widehat{Q}}(x))=(R_{\widehat{Q}}\otimes R_{\widehat{Q}})\Delta_{\widehat{Q}}^{op}(x)$ for $x\in \widehat{Q}$.\end{Lem}
\begin{proof} We only have to check whether \[\tilde{G}^*(1\otimes J_NxJ)\widehat{W} =
(J_N\otimes J_N)\Sigma \tilde{G}^* (1\otimes x)\widehat{W}\Sigma
(J\otimes J)\] for $x\in \widehat{Q}_{12}$. But using Lemma \ref{lem2} twice, once for $N$ and once
for $M$ itself, the right
hand side reduces:
\begin{eqnarray*}(J_N\otimes J_N)\Sigma \tilde{G}^* (1\otimes x)\widehat{W}\Sigma
(J\otimes J)&=&  \tilde{G}^*(\widehat{J}\otimes J_N)\Sigma U^*\Sigma
(1\otimes x)\widehat{W}\Sigma (J\otimes J)\\ &=&
\tilde{G}^*(\widehat{J}\otimes J_N)(1\otimes x)\Sigma V^*\Sigma
\widehat{W}\Sigma (J\otimes J)\\ &=& \tilde{G}^*(1\otimes J_NxJ)\widehat{W}
.\end{eqnarray*}
\end{proof}

\noindent This $(\widehat{Q},\Delta_{\widehat{Q}})$ could be called a coinvolutive
Hopf-von Neumann algebraic quantum groupoid, and in particular
$(\widehat{P},\Delta_{\widehat{P}})$ is a coinvolutive Hopf-von
Neumann algebra. We proceed to show that $(\widehat{Q},\Delta_{\widehat{Q}})$ is in fact a
measurable quantum groupoid (and $(\widehat{P},\Delta_{\widehat{P}})$ a locally
compact quantum group). However, we first briefly return to the situation of a general Galois coaction: it is not difficult to see that up to this point, everything in this section could be done without assuming $\alpha$ ergodic. Of course, $\widehat{P}$ will then not be a quantum group, but a quantum groupoid. More precisely: we will have that $(\widehat{P},N^{\alpha},\pi_l,\pi_r,\Delta_{\widehat{P}})$ is a Hopf bimodule (in the sense of Definition 3.1 of \cite{Eno3}), with $\Delta_{\widehat{P}}(x)=\tilde{G}^*(1\otimes x)\tilde{G} \in P \underset{N^{\alpha}}{_{\pi_r}*_{\,\pi_l}} P$ for $x\in \widehat{P}$. We can even equip it with a `scaling group' and a unitary antipode. However, we do not know if $\widehat{P}$ can actually be made into a measured quantum groupoid in general.\\

%: the remark after Theorem \ref{proprinv} (together with the intuition of \cite{DeC4}) makes it unlikely that this will be the case without further assumptions.\\

\noindent We have shown in Proposition \ref{lem10} that the modular automorphism group of
$\widehat{\varphi}^{\textrm{op}}$ on $\widehat{M}'$ can be implemented on $\mathscr{L}^2(N)$ by the one-parametergroup $\widehat{\nabla}_N^{it}=
P_N^{it}J_{N}\delta_N^{it}J_{N}$. Then by Theorem IX.3.11 in \cite{Tak1}, we can
construct an nsf weight $\varphi_{\widehat{P}}$ on $\widehat{P}$ which has
$\widehat{\nabla}_N$ as spatial derivative with respect to
$\widehat{\varphi}^{\textrm{op}}$. Then we can also consider the balanced weight
$\varphi_{\widehat{Q}}=\varphi_{\widehat{P}}\oplus \widehat{\varphi}$ on $\widehat{Q}$. Its modular automorphism
group $\sigma^{\widehat{Q}}_t$ is then implemented by
$\widehat{\nabla}_N^{it}\oplus \widehat{\nabla}^{it}$ if we use the faithful
representation $\pi^{\widehat{Q},2}$ of $\widehat{Q}$ on $\mathscr{L}^2(N)\oplus \mathscr{L}^2(M)$.\\

\noindent We make the identification
\[(\mathscr{L}^2({\widehat{Q}}),\pi^{\widehat{Q}},\Lambda_{\widehat{Q}})\cong
(\left(\begin{array}{cc} \mathscr{L}^2(\widehat{P}) & \mathscr{L}^2(N) \\
 \overline{\mathscr{L}^2(N)} & \mathscr{L}^2(M)\end{array}\right), \pi^{\widehat{Q}},
(\widehat{\Lambda}_{ij}))\] of the natural semi-cyclic representations of $\widehat{Q}$ w.r.t. $\varphi_{\widehat{Q}}$, as in Lemma IX.3.5 of \cite{Tak1} and the remark above it. Here $\left(\begin{array}{cc} \mathscr{L}^2(\widehat{P}) & \mathscr{L}^2(N) \\
 \overline{\mathscr{L}^2(N)} & \mathscr{L}^2(M)\end{array}\right)$ is just the direct sum Hilbert space of its entries, written as a matrix to emphasize its left $\widehat{Q}$-module structure. Further, $\widehat{\Lambda}_{11}$ and
$\widehat{\Lambda}_{22}$ are the ordinary GNS-constructions for the weights $\varphi_{\widehat{P}}$ and $\widehat{\varphi}$; the map
$\widehat{\Lambda}_{12}: \widehat{Q}_{12}\cap \mathscr{N}_{\varphi_{\widehat{Q}}}\rightarrow \mathscr{L}^2(N)$ is determined by
\[\widehat{\Lambda}_{12}(L_\xi)=\xi\] for $\xi\in \mathscr{L}^2(N)$ left-bounded (so that the closure $L_\xi$ of the map  $\widehat{\Lambda}^{\textrm{op}}(m)=\widehat{\Lambda}^{\textrm{op}}(\widehat{J}m^*\widehat{J})\rightarrow \widehat{\pi}_r(m)\xi=\widehat{\pi}_l(\widehat{J}m^*\widehat{J})\xi$ for $m\in \mathscr{N}_{\widehat{\varphi}}^*$ is bounded); and the map $\widehat{\Lambda}_{21}$ is determined by $\widehat{\Lambda}_{21}(L_\xi^*) = \overline{\widehat{\nabla}_{N}^{1/2}\xi}$ for $\xi\in \mathscr{L}^2(N)$ left-bounded \emph{and} in the domain of $\widehat{\nabla}_{N}^{1/2}$. Then the restriction $\widehat{J}_{21}$ of the modular conjugation $J_{\widehat{Q}}$ of $\varphi_{\widehat{Q}}$ to a map $\Lambda_{\widehat{Q}}(\mathscr{N}_{\varphi_{\widehat{Q}}}\cap \widehat{Q}_{21})\rightarrow \Lambda_{\widehat{Q}}(\mathscr{N}_{\varphi_{\widehat{Q}}}\cap \widehat{Q}_{12})$ is simply the natural anti-unitary map $\overline{\mathscr{L}^2(N)}\rightarrow \mathscr{L}^2(N): \overline{\xi}\rightarrow \xi$. We will denote the inverse of this map by $\widehat{J}_{12}$. Further, $\pi^{\widehat{Q}}$ decomposes as $\pi^{\widehat{Q},1}\oplus \pi^{\widehat{Q},2}$, with $\pi^{\widehat{Q},i}$ acting on the $i$-th column, and we will then also write $\pi^{\widehat{Q},1}=(\pi^{1}_{ij})_{i,j}$.\\

\noindent We will now provide another formula for $\tilde{G}^*$.

\begin{Lem}\label{lem13}  If $m\in \mathscr{N}_{\widehat{\varphi}}$ and $x\in
\widehat{N}\cap \mathscr{N}_{\varphi_{\widehat{Q}}}$, then
$\Delta_{\widehat{N}}(x)(m\otimes 1)\in \mathscr{D}(\Lambda_{\widehat{N}}\otimes
\Lambda_{\widehat{N}})$ and
\[(\Lambda_{\widehat{N}}\otimes \Lambda_{\widehat{N}})(\Delta_{\widehat{N}}(x)(m\otimes 1))=
\tilde{G}^* (\widehat{\Lambda}(m)\otimes
\Lambda_{\widehat{N}}(x)).\]\end{Lem}

\begin{proof} Since \begin{eqnarray*} (\iota\otimes \varphi_{\widehat{Q}})((m^*\otimes 1)\widehat{\Delta}_{12}(x)^*\widehat{\Delta}_{12}(x)(m\otimes
1)) &=& (\iota\otimes \widehat{\varphi})((m^*\otimes
1)\widehat{\Delta}(x^*x)(m\otimes 1)) \\ &=& \varphi_{\widehat{Q}}(x^*x)m^*m\end{eqnarray*} for $x\in \widehat{Q}_{12}$ and $m\in \widehat{M}$, it is clear that $\widehat{\Delta}_{12}(x)(m\otimes 1)\in \mathscr{D}(\widehat{\Lambda}_{12}\otimes
\widehat{\Lambda}_{12})$ for $m\in \mathscr{N}_{\widehat{\varphi}}$ and $x\in
\widehat{Q}_{12}\cap \mathscr{N}_{\varphi_{\widehat{Q}}}$, and that the map \[
\widehat{\Lambda}(m)\otimes \widehat{\Lambda}_{12}(x)\rightarrow
(\widehat{\Lambda}_{12}\otimes \widehat{\Lambda}_{12})(\widehat{\Delta}_{12}(x)(m\otimes 1))\]
extends to a well-defined
isometry. We now show that it coincides with $\tilde{G}^*$. \\

\noindent Let $z$ be an element of
$\mathscr{N}_{\widehat{\varphi}^{\textrm{op}}}$. Then it is sufficient to prove
that \[ \widehat{\Delta}_{12}(x)(\widehat{\Lambda}(m)\otimes
\widehat{\Lambda}^{\textrm{op}}(z))= (1\otimes
\widehat{\pi}_l(z))\tilde{G}^*(\widehat{\Lambda}(m)\otimes
\widehat{\Lambda}_{12}(x)).\]

\noindent But $\widehat{\Delta}_{12}(x)= \tilde{G}^* (1\otimes
x)\widehat{W}$, and bringing $\tilde{G}$ to the other side,
$\tilde{G}(1\otimes \widehat{\pi}_l(z))\tilde{G}^*$ can be written as $\Sigma
U(1\otimes \widehat{J}\widehat{R}'(z)^*\widehat{J})U^*\Sigma$. Taking a scalar
product in the first factor, it is then sufficient to prove that for
$\omega\in \widehat{M}'_*$, we have \[ x(\omega\otimes
\iota)(\widehat{W})\widehat{\Lambda}^{\textrm{op}}(z)=(\iota\otimes
\omega)(U(1\otimes \widehat{J}\widehat{R}'(z)\widehat{J})U^*)\widehat{\Lambda}_{12}(x).\]
But now using again that $(\widehat{\pi}_l\otimes \iota)(V)=U$, it is
sufficient to show that \[(\iota\otimes \omega)(V(1\otimes
\widehat{J}\widehat{R}'(z)\widehat{J})V^*)\in \mathscr{N}_{\widehat{\varphi}^{\textrm{op}}}\] and
that applying $\widehat{\Lambda}^{\textrm{op}}$ to it gives $(\omega\otimes
\iota)(\widehat{W})\widehat{\Lambda}^{\textrm{op}}(z)$. We could check this
directly, but we can just as easily backtrack our arguments: we only
have to see if for $y\in \mathscr{N}_{\widehat{\varphi}}$, we have
\[y(\omega\otimes \iota)(\widehat{W})\widehat{\Lambda}^{\textrm{op}}(z) =
(\iota\otimes \omega)(V(1\otimes
\widehat{J}\widehat{R}'(z)\widehat{J})V^*)\widehat{\Lambda}(y)\] for any
$z\in \mathscr{N}_{\widehat{\varphi}^{\textrm{op}}}$. This is then seen to be the
same as saying that \[(\widehat{\Lambda}\otimes
\widehat{\Lambda})(\widehat{\Delta}(y)(m\otimes 1))= \widehat{W}^*
(\widehat{\Lambda}(m)\otimes \widehat{\Lambda}(y)),\] which is of
course true by definition.

\end{proof}

\begin{Lem} \label{lem11} Let $x$ be in $\mathscr{N}_{\varphi_N}\cap \mathscr{N}_{\varphi_N}^*$, and $a\in \mathscr{T}_{\varphi}$, the Tomita algebra for $\varphi$. Then
\[(\omega_{\Lambda_{N}(x^*),\Lambda(\sigma_{i}(a)^*)}\otimes
\iota)(\tilde{G}) =
(\omega_{\Lambda(a),\Lambda_N(x)}\otimes
\iota)(\tilde{G}^*).\]

\end{Lem}

\begin{proof}

Choose $\omega\in N_*$. Then

\begin{eqnarray*} \omega((\omega_{\Lambda_N(x^*),\Lambda(\sigma_{i}(a)^*)}\otimes
\iota)(\tilde{G})) &=& \varphi(\sigma_i(a)((\omega\otimes
\iota)(\alpha(x)^*)))\\ &=& \varphi(((\omega\otimes
\iota)(\alpha(x)^*))a)\\ &=& \langle \Lambda(a),
\Lambda((\overline{\omega}\otimes \iota)\alpha(x))\rangle \\
&=& \langle\Lambda(a), (\iota\otimes
\overline{\omega})(\tilde{G}) \Lambda_N(x)\rangle\\ &=& \omega(
(\omega_{\Lambda(a),\Lambda_N(x)}\otimes
\iota)(\tilde{G}^*))\end{eqnarray*}

\end{proof}

\begin{Lem} Let $x\in \mathscr{N}_{\varphi_N}$ and $y\in \mathscr{T}_{\varphi_N}$. Then writing $w=x\sigma_{-i}^N(y^*)$, we have that
$\Lambda_N(w)$ is left-bounded, and
\[L_{\Lambda_N(w)}=(\iota\otimes
\omega_{\Lambda_N(x),\Lambda_N(y)})(\tilde{G}^*).\]\end{Lem}

\begin{proof}

We have to prove that for $m\in \mathscr{N}_{\widehat{\varphi}^{op}}$,
we have \[(\iota\otimes
\omega_{\Lambda_N(x),\Lambda_N(y)})(\tilde{G}^*)
\widehat{\Lambda}^{\textrm{op}}(m) =
\widehat{\pi}_l(m)\Lambda_N(x\sigma_{-i}^N(y^*)).\] But using the square
(\ref{eqn2}) at the end of section 1 and Lemma \ref{lem12}, we get for any $z\in \mathscr{N}_{\varphi_N}$ that
\begin{eqnarray*} \langle (\iota\otimes
\omega_{\Lambda_{N}(x),\Lambda_{N}(y)})(\tilde{G}^*)
\widehat{\Lambda}^{\textrm{op}}(m), \Lambda_{N}(z)\rangle &=&
\langle \tilde{G}^*(\widehat{\Lambda}^{\textrm{op}}(m)\otimes
\Lambda_{N}(x)), \Lambda_N(z)\otimes \Lambda_N(y)\rangle \\
&=& \langle
\Lambda_{N_2}(\widehat{\pi}_l(m)x),\Lambda_{N}(z)\otimes
\Lambda_{N}(y)\rangle \\ &=& \langle
\widehat{\pi}_l(m)x\Lambda_{N}(\sigma_{-i}^N(y^*)),\Lambda_N(z)\rangle.\end{eqnarray*}

\end{proof}

\begin{Prop}\label{lem15} If $x\in \widehat{N}\cap \mathscr{N}_{\varphi_{\widehat{Q}}}$ and $y\in
\widehat{O}\cap \mathscr{N}_{\varphi_{\widehat{Q}}}$, then
$\Delta_{\widehat{O}}(y)(x\otimes 1)$ in
$\mathscr{D}(\widehat{\Lambda}\otimes \Lambda_{\widehat{O}})$, and
\[(\widehat{\Lambda}\otimes \Lambda_{\widehat{O}})(\Delta_{\widehat{O}}
(y)(x\otimes 1)) = (J\otimes \widehat{J}_{12})\tilde{G}(J_N\otimes
\widehat{J}_{21}) (\Lambda_{\widehat{N}}(x)\otimes \Lambda_{\widehat{O}}(y)).\]
\end{Prop}

\noindent \textit{Remark:} Compare this formula with the identity $(\widehat{J}\otimes
J)W(\widehat{J}\otimes J)= W^*$.\\

\begin{proof} This statement is equivalent
with proving for sufficiently many $y$ in $\widehat{Q}_{21}\cap
\mathscr{N}_{\varphi_{\widehat{Q}}}$ and $\omega\in \widehat{Q}_{12,*}$ that
$(\omega\otimes \iota)(\widehat{\Delta}_{21}(y))\in \mathscr{N}_{\varphi_{\widehat{Q}}}$,
and \[\widehat{\Lambda}_{21}((\omega\otimes \iota)(\widehat{\Delta}_{21}(y)))=
(\omega\otimes \iota)((J\otimes \widehat{J}_{12})\tilde{G}(J_N\otimes
\widehat{J}_{21}))\widehat{\Lambda}_{21}(y), \] which can be written as
\begin{equation}\label{eqn6} \widehat{J}_{21}\widehat{\Lambda}_{21}((\omega\otimes \iota)(\widehat{\Delta}_{21}(y))) =
(\overline{\omega}(J(\cdot)^*J_N)\otimes
\iota)(\tilde{G})\widehat{J}_{21}\widehat{\Lambda}_{21}(y).\end{equation}

\noindent Let $y\in \widehat{Q}_{21}\cap \mathscr{N}_{\varphi_{\widehat{Q}}}$ be in the Tomita algebra of $\varphi_{\widehat{Q}}$. Let $\omega$ be of the form $\omega_{\Lambda_N(x),\Lambda(a)}$ with $x,a$ in the Tomita algebra of respectively $\varphi_N$ and $\varphi$. Then by the first formula of Lemma \ref{lem7} (used both in the general case and the case where $N=M$),  we have that $(\omega\otimes \iota)(\widehat{\Delta}_{21}(y))$ will also be analytic for $\sigma_t^{\widehat{Q}}$, with \[\sigma_{-i/2}^{\widehat{Q}} ((\omega_{\Lambda_N(x),\Lambda(a)}\otimes \iota)(\widehat{\Delta}_{21}(y)) = (\omega_{\nabla_N^{1/2}\Lambda_N(x),\nabla^{-1/2}\Lambda(a)}\otimes \iota)(\widehat{\Delta}_{12}(\sigma_{-i/2}^{\widehat{Q}}(y))).\] Further, $(\omega\otimes \iota)(\widehat{\Delta}_{21}(y))^* = (\overline{\omega}\otimes \iota)(\widehat{\Delta}_{12}(y^*))$, which will be in $\mathscr{D}(\widehat{\Lambda}_{12})$ by Lemma \ref{lem13}, with \[\widehat{\Lambda}_{12}((\overline{\omega}\otimes \iota)(\widehat{\Delta}_{12}(y^*))) = (\overline{\omega}\otimes \iota)(\tilde{G}^*)\widehat{\Lambda}_{12}(y^*).\] This shows that $(\omega\otimes \iota)(\widehat{\Delta}_{21}(y))\in \mathscr{D}(\widehat{\Lambda}_{21})$.\\

\noindent Now by Lemma \ref{lem11}, we have then also
\[\widehat{\Lambda}_{12}((\overline{\omega}\otimes \iota)(\widehat{\Delta}_{12}(y^*))) = (\omega_{\Lambda_N(x^*),\Lambda(\sigma_{-i}(a^*))}\otimes \iota)(\tilde{G})\widehat{\Lambda}_{12}(y^*),\] and by Lemma \ref{lem7}, we have that $(\omega_{\Lambda_N(x^*),\Lambda(\sigma_{-i}(a^*))}\otimes \iota)(\tilde{G})$ is analytic for $\chi_t=\textrm{Ad}(\widehat{\nabla}_N^{it})$, with \[\chi_{-i/2}((\omega_{\Lambda_N(x^*),\Lambda(\sigma_{-i}(a^*))}\otimes \iota)(\tilde{G}))=(\omega_{J_N\Lambda_N(x),J\Lambda(a)}\otimes \iota)(\tilde{G}).\] So combining all this, we get \begin{eqnarray*} \widehat{J}_{21}\widehat{\Lambda}_{21}((\omega\otimes \iota)(\widehat{\Delta}_{21}(y)))&=& \nabla_{\widehat{Q}}^{1/2} \widehat{\Lambda}_{12} ((\omega\otimes \iota)(\widehat{\Delta}_{21}(y))^*) \\ &=& (\nabla_{\widehat{Q}}^{1/2} (\omega_{\Lambda_N(x^*),\Lambda(\sigma_{-i}(a^*))}\otimes \iota)(\tilde{G})\nabla_{\widehat{Q}}^{-1/2})\nabla_{\widehat{Q}}^{1/2}\widehat{\Lambda}_{12}(y^*) \\ &=& (\omega_{J_N\Lambda_N(x),J\Lambda(a)}\otimes \iota)(\tilde{G})\widehat{J}_{21}\widehat{\Lambda}_{21}(y)\\ &=&  (\overline{\omega}(J(\cdot)^*J_N)\otimes
\iota)(\tilde{G})\widehat{J}_{21}\widehat{\Lambda}_{21}(y).\end{eqnarray*}

\noindent Now by closedness of $\Lambda_{\widehat{Q}}$, this equality remains true for $\omega$ arbitrary. Since such $y$'s form a $\sigma$-strong$^*$-norm core for $\widehat{\Lambda}_{12}$, the equality is true for any $y\in \widehat{Q}_{21}\cap \mathscr{N}_{\varphi_{\widehat{Q}}}$.

\end{proof}

\begin{Theorem}\label{corinv} The weight $\varphi_{\widehat{P}}$ is left invariant.\end{Theorem}

\begin{proof} It follows from the last proposition that \[(\iota \otimes \varphi_{\widehat{P}})(\Delta_{\widehat{P}}(L_{\xi}L_{\xi}^*)) = \varphi_{\widehat{P}}(L_{\xi}L_{\xi}^*)\] for $\xi$ right-bounded and in the domain of $\widehat{\nabla}_N^{1/2}$. From Lemma IX.3.9 of \cite{Tak1}, it follows that also $(\iota \otimes \varphi_{\widehat{P}})(\Delta_{\widehat{P}}(b)) = \varphi_{\widehat{P}}(b)$ for $b\in \mathscr{M}_{\varphi_{\widehat{P}}}^+$. Indeed: that lemma implies that $b$ can be approximated from below by elements of the form $\sum_{i=1}^n L_{\xi_i}L_{\xi_i}^*$ with $\xi_i$ right-bounded, and since $b$ is integrable, every $\xi_i$ must be in $\mathscr{D}(\widehat{\nabla}_N^{1/2})$. So we can conclude by lower-semi-continuity.

\end{proof}

\noindent This proves that $(\widehat{P},\Delta_{\widehat{P}})$ is a locally compact quantum group, since $\varphi_{\widehat{P}}$ is a left invariant weight, and by Lemma \ref{lem14}, $\psi_{\widehat{P}}:= \varphi_{\widehat{P}}\circ R_{\widehat{Q}}$ will be a right invariant weight.\\

\begin{Def} If $(N,\alpha)$ is a Galois object for a locally compact quantum group $(M,\Delta)$, and $(\widehat{P},\Delta_{\widehat{P}})$ the locally compact quantum group constructed from it in the foregoing manner, then we call $(\widehat{P},\Delta_{\widehat{P}})$ the \emph{reflected locally compact quantum group} (or just the reflection) of $(\widehat{M},\Delta_{\widehat{M}})$ across $(N,\alpha)$. \end{Def}

\noindent To end this section, we show that $(\widehat{Q},\Delta_{\widehat{Q}})$ is a measured quantum groupoid. In fact, our set-up is closer in spirit to the formulation of the generalized Kac algebras of \cite{Yam1}, but this theory has no full generalization to the `locally compact' world. It is however well known that these approaches are equivalent in the finite-dimensional Kac case (cf. \cite{Nik1}). \\

\noindent Let $d$ be the natural imbedding of $\mathbb{C}^2$ in $\widehat{Q}$: \[d:\mathbb{C}^2\rightarrow \widehat{Q}: (w,z)\rightarrow\left(\begin{array}{ll} w & 0 \\ 0 & z\end{array}\right).\] Let $\epsilon$ denote the map \[\epsilon:\mathbb{C}^2\rightarrow \mathbb{C}: (w,z)\rightarrow w+z.\] Then we have natural identifications \begin{eqnarray*} \mathscr{L}^2(\widehat{Q})\underset{\epsilon}{_d\otimes _d} \mathscr{L}^2(\widehat{Q})&=& (\oplus_{i,j} \mathscr{L}^2(\widehat{Q}_{ij}))\underset{\epsilon}{_d\otimes _d}(\oplus_{l,k} \mathscr{L}^2(\widehat{Q}_{lk}))\\&\;\cong\;& \bigoplus_{i,j,k}^{2} (\mathscr{L}^2(\widehat{Q}_{ij})\otimes \mathscr{L}^2(\widehat{Q}_{ik}))\\& =& \Delta_{\widehat{Q}}(1)(\mathscr{L}^2(\widehat{Q})\otimes  \mathscr{L}^2(\widehat{Q})),\end{eqnarray*} since $\underset{\epsilon}{_d\otimes _d}$ is just the ordinary balanced tensor product of two $\mathbb{C}^2$-modules. (Note that $\mathbb{C}^2$ acts on the left on both the $\mathscr{L}^2(\widehat{Q})$ spaces, so we don't get ordinary `matrix multiplication compatibility' on the summands!) Under this identification we have \[\widehat{Q}\underset{\mathbb{C}^2}{_d*_d}\widehat{Q} \;\cong\; \Delta_{\widehat{Q}}(1)(\widehat{Q}\otimes \widehat{Q})\Delta_{\widehat{Q}}(1),\] where the expression left is the fibred product. Thus $\Delta_{\widehat{Q}}$ can be seen as a map \[\Delta_{\widehat{Q}}:\widehat{Q}\rightarrow \widehat{Q}\underset{\mathbb{C}^2}{_d*_d}\widehat{Q}.\] Note now that the expressions $1 \underset{\mathbb{C}^2}{_d\otimes _d} d(x)$ and $d(x) \underset{\mathbb{C}^2}{_d\otimes _d} 1$ coincide with respectively  $(1\otimes d(x))\Delta_{\widehat{Q}}(1)$ and $(d(x)\otimes 1)\Delta_{\widehat{Q}}(1)$  for $x\in \mathbb{C}^2$. Using also that $\iota\underset{\mathbb{C}^2}{_d*_d} \Delta_{\widehat{Q}}$ is just the restriction of $(\iota\otimes \Delta_{\widehat{Q}})$ to $\Delta_{\widehat{Q}}(1)(\widehat{Q}\otimes \widehat{Q})\Delta_{\widehat{Q}}(1)$, it is easy to see that $\Delta_{\widehat{Q}}$ satisfies the coassociativity conditions for a Hopf bimodule as in Definition 3.1 of \cite{Eno3}. Now the octuple $(\mathbb{C}^2,\widehat{Q}, d,d,\Delta_{\widehat{Q}}, \left(\begin{array}{cc}\varphi_{\widehat{P}}&0 \\ 0 & \widehat{\varphi}\end{array}\right),\left(\begin{array}{cc}\psi_{\widehat{P}}&0 \\ 0 & \widehat{\psi}\end{array}\right), \epsilon)$ will form a measured quantum groupoid as in Definition 3.7 of \cite{Eno1}: First of all, after the proper identifications, it is easy to see that $T_{\widehat{Q}}=\left(\begin{array}{cc}\varphi_{\widehat{P}}&0 \\ 0 & \widehat{\varphi}\end{array}\right)$ is a left invariant nsf operator valued weight onto $d(\mathbb{C}^2)$, and that $T_{\widehat{Q}}'=\left(\begin{array}{cc}\psi_{\widehat{P}}&0 \\ 0 & \widehat{\psi}\end{array}\right)$ is a right invariant nsf operator valued weight onto $d(\mathbb{C}^2)$. So we only have to check whether $\epsilon$ is relatively invariant with respect to $T_{\widehat{Q}}$ and $T_{\widehat{Q}}'$. Now since $\psi_{\widehat{P}}\oplus \widehat{\psi} = (\varphi_{\widehat{P}}\oplus \widehat{\varphi})\circ R_{\widehat{Q}}$, we have \[\sigma_t^{\psi_{\widehat{P}}\oplus \widehat{\psi}} = R_{\widehat{Q}} \circ  \sigma_t^{\varphi_{\widehat{P}}\oplus \widehat{\varphi}} \circ R_{\widehat{Q}}.\] If we look at the faithful representation $\pi^{\widehat{Q},2}$ of $\widehat{Q}$ on $\mathscr{L}^2(N)\oplus \mathscr{L}^2(M)$, then \[(\widehat{\nabla}_N^{it}\oplus \widehat{\nabla}^{it})\pi^{\widehat{Q},2}(x)(\widehat{\nabla}_N^{-it}\oplus \widehat{\nabla}^{-it}) = \pi^{\widehat{Q},2} (\sigma_t^{\varphi_{\widehat{P}}\oplus \widehat{\varphi}}(x))\] and \[(J_N\oplus J)\pi^{\widehat{Q},2}(x)^*(J_N\oplus J) = \pi^{\widehat{Q},2}(R_{\widehat{Q}}(x))\] for $x\in \widehat{Q}$, so that $\sigma_t^{\psi_{\widehat{P}}\oplus \widehat{\psi}}$ is implemented on $\mathscr{L}^2(N)\oplus \mathscr{L}^2(M)$ by $\widehat{\acnabla}^{\,\,\,it}_N\oplus \widehat{\acnabla}^{\,\,\,it}$, where $\widehat{\acnabla}^{\,\,\,it}_N= J_N \widehat{\nabla}_N^{-it}J_N$. Using the definition of $\widehat{\nabla}_N$ and the commutation rules between $\delta_N$, $J_N\delta_NJ_N$ and $P_N$, it is then easy to see that indeed $\sigma_t^{\psi_{\widehat{P}}\oplus \widehat{\psi}}$ commutes with $\sigma_s^{\varphi_{\widehat{P}}\oplus \widehat{\varphi}}$.\\

%\noindent We will further examine $(\widehat{Q},\Delta_{\widehat{Q}})$ and its dual in \cite{DeC4}. \\

\section{Twisting by 2-cocycles}

\noindent We now treat a specific way to create non-trivial Galois objects, namely the twisting by cocycles. Let $(M,\Delta)$ be a locally compact quantum group, and let $\Omega\in \widehat{M}\otimes \widehat{M}$ be a unitary 2-cocycle, i.e. a unitary element satisfying \[(1\otimes \Omega)(\iota\otimes \widehat{\Delta})(\Omega) = (\Omega\otimes 1)(\widehat{\Delta}\otimes \iota)(\Omega).\] Denote by $\check{\alpha}$ the trivial coaction $\mathbb{C}\rightarrow \widehat{M}\otimes \mathbb{C}$ of $\widehat{M}$. The following definitions and propositions will refer to the paper \cite{Vae3}. So $(\check{\alpha},\Omega)$ is a cocycle action in the terminology of Definition 1.1. Let \[N= \widehat{M}\underset{\Omega}{\ltimes}\mathbb{C}:=\lbrack (\omega\otimes \iota)(\widehat{W}\Omega^*)\mid \omega\in \widehat{M}_*\rbrack^{\sigma-\textrm{weak}}\] be the cocycle crossed product as in Definition 1.3 (actually, one should take the von Neumann algebra generated by elements of this last set, in stead of just the $\sigma$-weak closure, but it will follow from our Proposition \ref{prop412} and the following proposition that this is the same). Then, by Proposition 1.4, there is a canonical \textit{right} coaction $\alpha$ of $M$ on $N$, determined by $\alpha((\omega\otimes \iota)(\widehat{W}\Omega^*))=(\omega\otimes \iota\otimes \iota)(\widehat{W}_{13}\widehat{W}_{12}\Omega_{12}^*)$, $\omega\in \widehat{M}_*$, and by Theorem 1.11.1, it is ergodic. By the remark after Lemma 1.12, it is integrable, and by Proposition 1.15 we can take the GNS-construction for $\varphi_N$ in $\mathscr{L}^2(M)$, by defining $\Lambda_N((\omega\otimes \iota)(\widehat{W}\Omega^*))=\Lambda((\omega\otimes \iota)(\widehat{W}))$ for $\omega\in \widehat{M}_*$ well-behaved. Finally, $(N,\alpha)$ is a Galois object, since the unitary $\widehat{W}\Omega^*\in B(\mathscr{L}^2(M)) \otimes N $ satisfies $(\iota\otimes \alpha)(\widehat{W}\Omega^*)=\widehat{W}_{13}\widehat{W}_{12}\Omega_{12}^*$, so that $\alpha$ is semi-dual (see Proposition 5.12 of \cite{Vae1} in the setting of left coactions). In fact, not surprisingly:

 \begin{Prop} The Galois map $\tilde{G}$ equals $\widehat{W}\Omega^*$.\end{Prop}

 \begin{proof} Choose $\xi,\eta,\zeta\in \mathscr{L}^2(M)$, and an orthonormal basis $\xi_i$ of $\mathscr{L}^2(M)$. Further, let $m\in \widehat{M}$ be a Tomita element for $\widehat{\varphi}$, and denote $\omega'= \omega_{\zeta,\widehat{\Lambda}(m)}$. Then by Proposition 1.15 of \cite{Vae3}, $(\omega'\otimes \iota)(\widehat{W}\Omega^*)\in \mathscr{N}_{\varphi_N}$, $(\omega'\otimes \iota)(\widehat{W})\in \mathscr{N}_{\varphi}$  and \[\Lambda_N( (\omega'\otimes \iota)(\widehat{W}\Omega^*))=\Lambda( (\omega'\otimes \iota)(\widehat{W})).\] So  \begin{eqnarray*} (\iota\otimes \omega_{\xi,\eta})(\tilde{G})&&\!\!\!\!\!\!\!\!\!\!\!\!\!\!\Lambda((\omega'\otimes \iota)(\widehat{W}))\\ &=& (\iota\otimes \omega_{\xi,\eta})(\tilde{G})\Lambda_N((\omega'\otimes \iota)(\widehat{W}\Omega^*)) \\ &=& \Lambda((\omega_{\xi,\eta}\otimes \iota)(\alpha((\omega'\otimes \iota)(\widehat{W}\Omega^*))))\\ &=& \Lambda((\omega'\otimes \omega_{\xi,\eta}\otimes \iota)(\widehat{W}_{13}\widehat{W}_{12}\Omega_{12}^*))\\ &=& \Lambda(\sum_i (\omega'\otimes \omega_{\xi,\xi_i}\otimes \omega_{\xi_i,\eta}\otimes \iota)(\widehat{W}_{14}\widehat{W}_{13}\Omega_{12}^*)),\end{eqnarray*} where the sum is taken in the $\sigma$-strong-topology.\\

 \noindent On the other hand, using Result 8.6 of \cite{Kus1}, adapted to the von Neumann algebra setting, we get \begin{eqnarray*} (\iota\otimes \omega_{\xi,\eta})(\widehat{W}\Omega^*) &&\!\!\!\!\!\!\!\!\!\!\!\!\!\!\Lambda((\omega'\otimes \iota)(\widehat{W}))\\&=& \sum_i (\iota\otimes \omega_{\xi_i,\eta})(\widehat{W})(\iota\otimes \omega_{\xi,\xi_i})(\Omega^*)\Lambda((\omega'\otimes \iota)(\widehat{W}))\\ &=& \sum_i \Lambda((\omega_{\xi_i,\eta}\otimes \iota)\Delta((\omega'(\cdot (\iota\otimes \omega_{\xi,\xi_i})(\Omega^*))\otimes \iota)(\widehat{W})))\\ &=& \sum_i \Lambda((\omega'\otimes \omega_{\xi,\xi_i}\otimes \omega_{\xi_i,\eta}\otimes \iota)(\widehat{W}_{14}\widehat{W}_{13}\Omega_{12}^*)),\end{eqnarray*} so that the result follows by the closedness of $\Lambda$ and the density of elements of the form $\Lambda( (\omega'\otimes \iota)(\widehat{W}))$ in $\mathscr{L}^2(M)$.

 \end{proof}

\begin{Theorem}\label{cor3} The $\Omega$-twisted Hopf-von Neumann algebra $(\widehat{M},\widehat{\Delta}_{\Omega})$ is a
locally compact quantum group.\end{Theorem}

\begin{proof}

\noindent Recall that the $\Omega$-twisted Hopf-von Neumann algebra is the algebra $\widehat{M}$ with the comultiplication $\widehat{\Delta}_{\Omega}(m)=\Omega\widehat{\Delta}(m)\Omega^*$. But the representation of $\widehat{M}'$ on $\mathscr{L}^2(N)$ equals the ordinary representation on $\mathscr{L}^2(M)$ (since it's easy to see that the unitary implementation of the coaction \emph{equals} the right regular representation $V$), so we can identify the underlying algebra of the reflected quantum group $(\widehat{P},\Delta_{\widehat{P}})$ with $\widehat{M}$, and then \begin{eqnarray*} \Delta_{\widehat{P}}(m) &=& \tilde{G}^*(1\otimes m)\tilde{G} \\ &=& \Omega \widehat{W}^*(1\otimes m)\widehat{W} \Omega^*\\ &=& \widehat{\Delta}_{\Omega}(m),\end{eqnarray*} which proves the theorem.
\end{proof}

\noindent \textit{Remark:} This answers negatively a question of \cite{Szy1}: the 2-pseudo-cocycles $\Omega_q$ of \cite{Szy1} are \emph{not} 2-cocycles, since $SU_0(2)$ is not a quantum group. This of course does not rule out the possibility that the $SU_q(2)$ \emph{are} cocycle twists of each other in some other way.\\

\noindent We will keep notation as in the previous sections, so we keep writing $(\widehat{P},\Delta_{\widehat{P}})$ for $(\widehat{M},\widehat{\Delta}_{\Omega})$.\\

\noindent Denote by $u_t = \widehat{\nabla}_N^{it}\widehat{\nabla}^{-it} \in \widehat{M}$ the cocycle derivative of $\varphi_{\widehat{P}}$ with respect to $\widehat{\varphi}$, so that $u_{s+t}=u_s\widehat{\sigma}_s(u_t)$. Denote $v_t= \nabla_N^{it}\nabla^{-it}$. Then also $v_t\in \widehat{M}$, since $\nabla_N^{it}$ and $\nabla^{it}$ implement the same automorphism on $\widehat{M}'$. Finally, denote $X= J_NJ$, then $X\in \widehat{M}$ for the same reason.

\begin{Prop} \begin{enumerate} \item The one-parametergroup $v_t$ is a cocycle with respect to $\widehat{\tau}_t$.
\item The 2-cocycles $\Omega$ and $(\widehat{\tau}_t\otimes \widehat{\tau}_t)(\Omega)$ are cohomologous by the coboundary $v_t$.
\item The 2-cocycles $\Omega$ and $\tilde{\Omega}=(\widehat{R}\otimes \widehat{R})(\Sigma \Omega^*\Sigma)$ are cohomologous by the coboundary $X$.

\end{enumerate}\end{Prop}

\noindent \emph{Remark:} The third statement of this proposition was noted for 2-cocycles in the von Neumann group algebra of a compact group in \cite{Was2}.\\

\begin{proof} By Lemma \ref{lem6}, we have \[(\nabla^{it}\otimes u_t\widehat{\nabla}^{it})(\widehat{W}\Omega^*) = (\widehat{W}\Omega^*)(\nabla_N^{it}\otimes u_t\widehat{\nabla}^{it}).\] Since $\nabla^{it}\otimes \widehat{\nabla}^{it}$ commutes with $\widehat{W}$ and $\nabla^{it}$ implements $\widehat{\tau}_t$ on $\widehat{M}$, the left hand side can be rewritten as $(1\otimes u_t)\widehat{W}(\widehat{\tau_t}\otimes \widehat{\sigma}_t)(\Omega^*) (\nabla^{it}\otimes \widehat{\nabla}^{it})$, and so, bringing $\widehat{W}$ and $(\nabla^{it}\otimes \widehat{\nabla}^{it})$ to the other side, we obtain \[\widehat{\Delta}(u_t)(\widehat{\tau}_t\otimes \widehat{\sigma}_t)(\Omega^*) = \Omega^*(v_t \otimes u_t).\]  Hence \begin{eqnarray*} v_{s+t}\otimes u_{s+t}  &=& \Omega \widehat{\Delta}(u_{s+t})(\widehat{\tau}_{s+t}\otimes \widehat{\sigma}_{s+t})(\Omega^*)\\ &=& \Omega \widehat{\Delta}(u_s\widehat{\sigma}_s(u_t))(\widehat{\tau}_{s+t}\otimes \widehat{\sigma}_{s+t})(\Omega^*) \\&=& \Omega \widehat{\Delta}(u_s)(\widehat{\tau}_{s}\otimes \widehat{\sigma}_{s})(\Omega^*) \cdot (\widehat{\tau}_{s}\otimes \widehat{\sigma}_{s}) ( \Omega \widehat{\Delta}(u_t)(\widehat{\tau}_{t}\otimes \widehat{\sigma}_{t})(\Omega^*)) \\ &=& v_s \widehat{\tau}_s(v_t)\otimes u_s\widehat{\sigma}_s(u_t),\end{eqnarray*} from which the cocycle property of $v_t$ follows.\\

\noindent Now note that $v_t$ also equals $P_N^{it}P^{-it}$ (by definition of $P_N$). So using  the third equality of Corollary \ref{cor1}, \[\widehat{W}\Omega^* (v_t\otimes v_t)(P^{it}\otimes P^{it}) = (P^{it}\otimes v_t P^{it})\widehat{W}\Omega^*.\] Using that $P^{it}=\widehat{P}^{it}$, taking $\widehat{W}$ and $P^{it}\otimes P^{it}$ to the other side, we arrive at \[\Omega^* (v_t\otimes v_t)= \widehat{\Delta}(v_t)(\widehat{\tau}_t\otimes \widehat{\tau}_t)(\Omega^*),\] which proves the second statement. \\

\noindent Finally, as mentioned already, the unitary implementation of $\alpha$ is just $V$ itself. So by Lemma \ref{lem2}, we have $\widehat{W}\Omega^*(J_N\otimes J_N)\Sigma = \Sigma V\Sigma (\widehat{J}\otimes J_N) \widehat{W}\Omega^*$. Multiplying to the right with $(J\otimes J)\Sigma$, we get \begin{eqnarray*} \widehat{W}\Omega^*(X\otimes X) &=& \Sigma V\Sigma (1\otimes X)(\widehat{J}\otimes J)\widehat{W}\Omega^*(J\otimes J)\Sigma \\ &=& \Sigma V\Sigma (1\otimes X) (\widehat{J}\otimes J)\widehat{W}(J\otimes J)\Sigma \tilde{\Omega}^* \\ &=& \Sigma V\Sigma (1\otimes X)(\widehat{J}\otimes J)\Sigma V\Sigma (\widehat{J}\otimes J)\widehat{W}\tilde{\Omega}^*\\ &=& \Sigma V\Sigma (1\otimes X)\Sigma V^*\Sigma \widehat{W}\tilde{\Omega}^* \\ &=& (1\otimes X)\widehat{W}\tilde{\Omega}^*,\end{eqnarray*} from which $\Omega^* (X\otimes X) = \widehat{\Delta}(X)\tilde{\Omega}^*$ immediately follows.
\end{proof}

\noindent We have the following formula for the multiplicative unitary $\widehat{W}_{\Omega}$ for $(\widehat{M},\widehat{\Delta}_{\Omega})$:

\begin{Prop} $\widehat{W}_{\Omega}= (J_N\otimes \widehat{J})\Omega\widehat{W}^*(J\otimes \widehat{J})\Omega^* $.\end{Prop}

\noindent \textit{Remark:} This is to be compared with the formula for the multiplicative unitary in \cite{Fim1}. \\

\begin{proof}

\noindent We will use notation as in the previous section. As already noted, the unitary implementation of $\alpha$ is the multiplicative unitary $V$, and the identification of $\mathscr{L}^2(N)$ with $\mathscr{L}^2(M)$ is an identification of left $\widehat{M}'$-modules. Hence the left module structure of $\left(\begin{array}{cc} \widehat{M} & \widehat{M}\\ \widehat{M}&\widehat{M}\end{array}\right)$ on $\left(\begin{array}{cc} \mathscr{L}^2(M) & \mathscr{L}^2(M)\\ \overline{\mathscr{L}^2(M)}&\mathscr{L}^2(M)\end{array}\right)$ is explicitly known: if we identify $\overline{\mathscr{L}^2(M)}$ with $\mathscr{L}^2(M)$ by the map $\widehat{J}\widehat{J}_{21}$, then the module structure is just ordinary matrix multiplication, the module structure on all summands being the standard one. Also, $\widehat{\Lambda}_{12}$ becomes $\widehat{\Lambda}$, and $\widehat{\Lambda}_{21}$ becomes $\Lambda_{\varphi_{\widehat{P}}}$. Then from Lemma \ref{lem15}, and the fact that $\widehat{\Delta}_{21}(x)=\widehat{\Delta}(x)\Omega^*$, it is easy to conclude that $\widehat{W}_{\Omega}\Omega =\tilde{G}_{J}^* = (J_N\otimes \widehat{J})(\Omega\widehat{W}^*)(J\otimes \widehat{J})$. The proposition follows.

\end{proof}

\section{Galois objects and coactions on type I factors}

\noindent We now look at a possible way to create examples. One of the major motivations for me to study Galois objects was the paper \cite{BDV1}. In this article, the authors consider examples of Galois objects (there termed `ergodic coactions of full quantum multiplicity') which were not induced by a 2-cocycle. This was surprising, as Wassermann had shown in \cite{Was2} that for compact groups, any Galois object for the function algebra must come from a 2-cocycle of the dual (a result which was in turn based on the work in \cite{Was1}, and ultimately on the fundamental results of \cite{Lan1}). In fact, in \cite{BDV1} all Galois objects for the compact quantum groups $SU_q(2)$ are classified. There is a whole family of them, parametrized by orthogonal matrices which satisfy some relation w.r.t. $q$, even though there are no non-trivial cocycles for the dual of $SU_q(2)$.\\

\noindent To obtain examples in our wider setting, the following construction would seem to be very helpful. It is a generalization of the fact that any action (and by the work of A. Wassermann, any coaction (\cite{Was2}, Theorem 3)) of a compact group on a type I-factor comes from a cocycle representation. We need some terminology.\\

\begin{Def} Let $(N,\alpha)$ be a (right) Galois object for a locally compact quantum group $(M,\Delta)$. Denote by $\widehat{N}$ the space of $\mathscr{L}^2(M)_{\widehat{M}}$-$\mathscr{L}^2(N)_{\widehat{M}}$-intertwiners as before. Let $\mathscr{H}$ be a Hilbert space. A (unitary) \emph{left $(N,\alpha)$-corepresentation} for $(\widehat{M},\widehat{\Delta})$ is a unitary $\mathcal{G}\in \widehat{N}\otimes B(\mathscr{H})$ such that $(\Delta_{\widehat{N}}\otimes \iota)(\mathcal{G})=\mathcal{G}_{13}\mathcal{G}_{23}$. By a \emph{projective corepresentation} for $(\widehat{M},\widehat{\Delta})$, we mean a left $(N,\alpha)$-corepresentation for $(\widehat{M},\widehat{\Delta})$ and some Galois object $(N,\alpha)$.\end{Def}

\noindent For any Galois object $(N,\alpha)$, there is a regular left $(N,\alpha)$-corepresentation, namely the unitary $(J_N\otimes \widehat{J}_{12})\tilde{G}^*(J\otimes \widehat{J}_{21})$. In case $(M,\Delta)=(\mathscr{L}(\mathfrak{G}),\Delta)$ is the group von Neumann algebra of an ordinary locally compact group $\mathfrak{G}$, and $(N,\alpha)$ is the twisted convolution algebra by a cocycle $\Omega\in \mathscr{L}^{\infty}(\mathfrak{G})\otimes \mathscr{L}^{\infty}(\mathfrak{G})$, we just get back the ordinary notion of a cocycle representation. Of course, one can also easily adapt the definition to find the notion of a right $(N,\alpha)$-corepresentation.\\

\noindent \begin{Theorem}\label{proptypI} Let $(M,\Delta)$ be a locally compact quantum group. If $(N,\alpha)$ is a Galois object for $(M,\Delta)$, then any left $(N,\alpha)$-corepresentation of $(\widehat{M},\widehat{\Delta})$ gives rise to a left coaction of $(\widehat{M},\widehat{\Delta})$ on a type-$I$-factor. Conversely, any left coaction on a type-$I$-factor is induced by a projective corepresentation.\end{Theorem}

\begin{proof} The first statement is easy: if $\mathcal{G}$ is such a corepresentation, define \[\Upsilon: B(\mathscr{H})\rightarrow \widehat{M}\otimes B(\mathscr{H}): x\rightarrow \mathcal{G}^*(1\otimes x)\mathcal{G},\] then this is a coaction by the defining property of $\mathcal{G}$. \\

\noindent Now let $\mathscr{H}$ be a Hilbert space, and $\Upsilon: B(\mathscr{H})\rightarrow \widehat{M}\otimes B(\mathscr{H})$ be a coaction of $(\widehat{M},\widehat{\Delta})$. Denote by $N$ the relative commutant of $\Upsilon(B(\mathscr{H}))$ inside $\widehat{M}\ltimes B(\mathscr{H})$. Then we have a canonical isomorphism $\Phi:\widehat{M}\ltimes B(\mathscr{H})\rightarrow   N\otimes B(\mathscr{H})$. We claim that the dual (right) coaction $\widehat{\Upsilon}: \widehat{M}\ltimes B(\mathscr{H})\rightarrow (\widehat{M}\ltimes B(\mathscr{H}))\otimes M$ restricts to a coaction of $M$ on $N$. Indeed: choose an orthonormal basis $\xi_{i}$ of $\mathscr{H}$, with respective matrix unit system $\{e_{ij}\}$. Then for $x\in N$, we have $x=\sum_{k} \Upsilon(e_{k1})x\Upsilon(e_{1k})$ $\sigma$-strongly. Applying $\widehat{\Upsilon}$, we get $\widehat{\Upsilon}(x)= \sum_k (\Upsilon(e_{k1})\otimes 1)\widehat{\Upsilon}(x)(\Upsilon(e_{1k})\otimes 1)$, whose first leg clearly commutes with $\Upsilon(B(\mathscr{H}))$. \\

\noindent We now show that $(N,\alpha)$ is a Galois object. Ergodicity is clear, since $1\otimes B(\mathscr{H})$ is the fixed point algebra of $\textrm{Ad}(\Sigma)_{23}(\alpha\otimes\iota)=(\Phi\otimes \iota)\widehat{\Upsilon}\circ \Phi^{-1}$. Also integrability follows easily by this, $\widehat{\Upsilon}$ being integrable. Since we have a canonical isomorphism $(\widehat{M}\ltimes B(\mathscr{H}))\rtimes M \cong  (N\rtimes M)\otimes B(\mathscr{H})$, and the first space is $\cong B(\mathscr{H})\otimes B(\mathscr{L}^2(M))$, also $N\rtimes M$ must be a type I factor, from which it follows that the Galois homomorphism for $\alpha$ is necessarily an isomorphism.\\

\noindent We now show that the original coaction is implemented by an $(N,\alpha)$-corepresentation. Denote by $\textrm{Tr}$ the ordinary trace on $B(\mathscr{H})$, by $\widehat{\textrm{Tr}}$ the dual weight on $\widehat{M}\ltimes B(\mathscr{H})$ with respect to $\textrm{Tr}$, and by $\varphi_N$ the weight $(\iota\otimes \varphi)\alpha$ on $N$. Then we have $\widehat{\textrm{Tr}}= (\varphi_N \otimes \textrm{Tr})\circ \Phi$. Hence we obtain a unitary \[u: \mathscr{L}^2(M)\otimes \mathscr{L}^2(B(\mathscr{H}))\rightarrow \mathscr{L}^2(N)\otimes \mathscr{L}^2(B(\mathscr{H}))\] which sends $\Lambda(m)\otimes \Lambda_{\textrm{Tr}}(x)$ to $(\Lambda_N\otimes \Lambda_{\textrm{Tr}})(\Phi(m\otimes 1)(1\otimes x))$ for $m\in \mathscr{N}_{\varphi}$ and $x$ Hilbert-Schmidt. But identifying $\mathscr{L}^2(B(\mathscr{H}),\textrm{Tr})$ with $\mathscr{H}\otimes \overline{\mathscr{H}}$, and observing that $u$ is right $B(\mathscr{H})$-linear, we must have that $u=\mathcal{G}\otimes 1$ for some unitary \[\mathcal{G}: \mathscr{L}^2(M)\otimes \mathscr{H}\rightarrow \mathscr{L}^2(N)\otimes \mathscr{H}.\]

\noindent We proceed to show that $\mathcal{G}$ is indeed an $(N,\alpha)$-corepresentation implementing $\Upsilon$. First of all, it is not difficult to see that $\mathcal{G}\in \widehat{Q}_{12}\otimes B(\mathscr{H})$, since for $m\in \mathscr{N}_{\varphi}$ and $x$ Hilbert-Schmidt, and $\xi,\eta \in \mathscr{L}^2(M)$ with $\xi\in \mathscr{D}(\delta^{-1/2})$, we have, putting $\omega=\omega_{\xi,\eta}$ and $\omega_{\delta}=\omega_{\delta^{-1/2}\xi,\eta}$, denoting by $U$ the unitary corepresentation belonging to $\alpha$ and by $V$ the right regular representation for $(M,\Delta)$, \begin{eqnarray*}u ((\iota\otimes \omega)(V)\otimes 1)(\Lambda(m)\otimes  \Lambda_{\textrm{Tr}}(x)) &=& u (\Lambda((\iota\otimes \omega_{\delta})(\Delta(m)))\otimes\Lambda_{\textrm{Tr}}(x)) \\ &=& (\Lambda_N\otimes \Lambda_{\textrm{Tr}})(\Phi(((\iota\otimes \omega_{\delta})(\Delta(m)))\otimes 1)(1\otimes x)) \\  &=& (\Lambda_N\otimes \Lambda_{\textrm{Tr}})(\Phi((((\iota\otimes \omega_{\delta})(\Delta(m)))\otimes 1)\Upsilon(x))) \\ &=& (\Lambda_N\otimes \Lambda_{\textrm{Tr}})(\Phi((\iota\otimes \omega_{\delta})(\widehat{\Upsilon}((m\otimes 1)\Upsilon(x))))) \\ &=& (\Lambda_N\otimes \Lambda_{\textrm{Tr}})((\iota\otimes \omega_{\delta}\otimes \iota)(\alpha\otimes \iota) \Phi((m\otimes 1)\Upsilon(x))) \\ &=& ((\iota\otimes \omega)(U)\otimes 1)(\Lambda_N\otimes \Lambda_{\textrm{Tr}})( \Phi((m\otimes 1)\Upsilon(x)))\\ &=& ((\iota\otimes \omega)(U)\otimes 1) u (\Lambda(m)\otimes  \Lambda_{\textrm{Tr}}(x)),\end{eqnarray*} so that $\mathcal{G}((\iota\otimes \omega)(V)\otimes 1) = ((\iota\otimes \omega)(U)\otimes 1)\mathcal{G}$, which is sufficient to conclude that the first leg of $\mathcal{G}$ is in $\widehat{Q}_{12}$.\\

\noindent Also, it is easy to see that $\mathcal{G}$ implements $\Upsilon$: since $u\Upsilon(x)=(1\otimes x)u$ on $\mathscr{L}^2(M)\otimes \mathscr{L}^2(B(\mathscr{H}))$, we have $\mathcal{G} \Upsilon(x) = (1\otimes x)\mathcal{G}$ on $\mathscr{L}^2(M)\otimes \mathscr{H}$. \\

\noindent So the only thing left to show, is that $\mathcal{G}$ satisfies $(\widehat{\Delta}_{12}\otimes \iota)(\mathcal{G})=\mathcal{G}_{13}\mathcal{G}_{23}$. Writing out $\widehat{\Delta}_{12}$ and tensoring by $1_{\overline{\mathscr{H}}}$ to the right, this translates into proving that $\tilde{G}_{12}^*u_{23}\widehat{W}_{12}= u_{13}u_{23}$, with $\tilde{G}$ the Galois unitary for $(N,\alpha)$. Moving $\tilde{G}$ to the other side, and multiplying to the left with $\Sigma_{12}$, this becomes $u_{13}W^*_{12}\Sigma_{12}= \Sigma_{12}\tilde{G}_{12}u_{13}u_{23}$. This can again be proven using a simple matrix algebra argument: we can write $\Phi(m\otimes 1)= \sum_{i,j} \Phi_{ij}(m)\otimes e_{ij}$ with $\Phi_{ij}(m)=\sum_{k} \Upsilon(e_{ki})(m\otimes 1)\Upsilon(e_{jk})\in N$, where the sums are in the $\sigma$-strong topology. Then for $m,n\in \mathscr{N}_{\varphi}$ and $x$ Hilbert-Schmidt, we have

\begin{eqnarray*} &&\!\!\!\!\!\!\!\!\!\!\!\!\!\! u_{13}W^*_{12}\Sigma_{12} (\Lambda(m)\otimes\Lambda(n)\otimes \Lambda_{\textrm{Tr}}(x))\\ &=& u_{13} (\Lambda\otimes \Lambda\otimes \Lambda_{\textrm{Tr}}) (\Delta(m)(n\otimes 1)\otimes x)\\ &=& (\Lambda_N\otimes \Lambda_M\otimes \Lambda_{\textrm{Tr}})(\sum_{i,j}((\Phi_{ij}\otimes \iota)(\Delta(m)(n\otimes 1)) \otimes e_{ij}x)),\end{eqnarray*}
 while \begin{eqnarray*} &&\!\!\!\!\!\!\!\!\!\!\!\!\!\!\Sigma_{12}\tilde{G}_{12}u_{13}u_{23}(\Lambda(m)\otimes \Lambda(n)\otimes \Lambda_{\textrm{Tr}}(x))\\ &=& \Sigma_{12}\tilde{G}_{12} u_{13} (\Lambda\otimes \Lambda_N\otimes \Lambda_{\textrm{Tr}})(\sum_{i,j} m\otimes \Phi_{ij}(n)\otimes e_{ij}x)\\&=& \Sigma_{12}\tilde{G}_{12} (\Lambda_N\otimes \Lambda_N\otimes \Lambda_{\textrm{Tr}})(\sum_{i,j,r}(\Phi_{ri}(m)\otimes \Phi_{ij}(n)\otimes e_{rj}x))\\ &=& (\Lambda_N\otimes \Lambda_M\otimes \Lambda_{\textrm{Tr}})(\sum_{i,j,r}((\alpha(\Phi_{ri}(m))\otimes 1)(\Phi_{ij}(n)\otimes 1\otimes e_{rj}x)))\\ &=& (\Lambda_N\otimes \Lambda_M\otimes \Lambda_{\textrm{Tr}})(\sum_{i,j,r} ((\Phi_{ri}\otimes \iota)(\Delta(m))\otimes 1)(\Phi_{ij}(n)\otimes  1\otimes e_{rj}x))\\ &=&  (\Lambda_N\otimes \Lambda_M\otimes \Lambda_{\textrm{Tr}})(\sum_{j,r}((\Phi_{rj}\otimes \iota)(\Delta(m)(n\otimes 1))\otimes e_{rj}x)),\end{eqnarray*} where we have used $\sum_{i} \Phi_{ri}(m)\Phi_{ij}(n) = \Phi_{rj}(mn)$ for $m,n\in M$ in the last step. So we are done.

\end{proof}

\noindent \textit{Remark:} Starting from an $(N_1,\alpha_1)$-corepresentation $\mathcal{G}_1$, one thus obtains a Galois object $(N_2,\alpha_2)$ and an $(N_2,\alpha_2)$-corepresentation $\mathcal{G}_2$. As is to be expected, these Galois objects are isomorphic, in such a way that the corepresentations correspond to each other. Indeed: it's easy to see that $\mathcal{G}_1\mathcal{G}_2^* = v\otimes 1$ for some unitary $v: \mathscr{L}^2(N_2)\rightarrow \mathscr{L}^2(N_1)$. Since $v$ is a right $\widehat{M}$-module map, we can extend the (well-defined) map $\widehat{Q}_{2,12}\rightarrow \widehat{Q}_{1,12}: z\rightarrow vz$ to an isomorphism $\Psi$ of the linking algebras $\widehat{Q}_2$ and $\widehat{Q}_1$. From the fact that $\mathcal{G}_1$ and $\mathcal{G}_2$ are corepresentations, it is easy to deduce that $\widehat{\Delta}_{1,12}(vz)=(v\otimes v)\widehat{\Delta}_{2,12}(z)$ for $z\in \widehat{Q}_{2,12}$. Hence $\Psi$ preserves the comultiplication structure, and thus $(N_1,\alpha_1)$ and $(N_2,\alpha_2)$ are isomorphic by a map $\widehat{\Psi}$, and moreover $(\Psi\otimes \iota)(\mathcal{G}_2)=\mathcal{G}_1$.\\

\noindent Recall that two coactions $\Upsilon_1$ and $\Upsilon_2$ of $(\widehat{M},\widehat{\Delta})$ on a von Neumann algebra $Y$ are called outer equivalent if there exists a unitary element $v\in \widehat{M}\otimes Y$ which satisfies \[(\widehat{\Delta}\otimes \iota)(v)= v_{23}(\iota\otimes \Upsilon_1)(v),\] (i.e., $v$ is an $\Upsilon_1$-cocycle) and such that $\Upsilon_2(x)=v\Upsilon_1(x)v^*$ for $x\in Y$. Then it is easy to see that also the following classical result still holds true:

\begin{Theorem} Suppose $(M,\Delta)$ is a locally compact quantum group for which $M$ has a separable predual. Then there is a natural one-to-one correspondence between outer equivalence classes of coactions of $(\widehat{M},\widehat{\Delta})$ on $B(\mathscr{H})$, with $\mathscr{H}$ a separable infinite-dimensional Hilbert space, and isomorphism classes of right Galois objects (with separable predual) for $(M,\Delta)$.

\end{Theorem}

\begin{proof} First suppose that $\Upsilon_1$ and $\Upsilon_2$ are two coactions on $B(\mathscr{H})$ which are outer equivalent by a unitary $v$. Then we get an isomorphism \[\Phi: \widehat{M}\underset{\Upsilon_1}{\ltimes} B(\mathscr{H})\rightarrow \widehat{M}\underset{\Upsilon_2}{\ltimes} B(\mathscr{H}): z\rightarrow vzv^*,\] which obviously sends $\Upsilon_1(B(\mathscr{H}))$ to $\Upsilon_2(B(\mathscr{H}))$. Hence if $(N_i,\alpha_i)$ denotes the Galois object constructed from $\Upsilon_i$ as in the previous Theorem, $N_1$ gets sent to $N_2$ by $\Phi$. But $\Phi$ also preserves the dual right coaction, since $V_{13}v_{12}= v_{12}V_{13}$. So $\Phi_{\mid N_1}$ gives an $(M,\Delta)$-equivariant isomorphism from $(N_1,\alpha_1)$ to $(N_2,\alpha_2)$. \\

\noindent Conversely, suppose that $\Upsilon_1$ and $\Upsilon_2$ are two coactions on $B(\mathscr{H})$, which are induced by respective $(N,\alpha)$-corepresentations $\mathcal{G}_1$ and $\mathcal{G}_2$ for some Galois object $(N,\alpha)$ for $(M,\Delta)$. Put $v=\mathcal{G}_2^*\mathcal{G}_1 \in \widehat{M} \otimes B(\mathscr{H})$. Then $v$ is an $\Upsilon_1$-cocycle: \begin{eqnarray*} (\widehat{\Delta}\otimes \iota)(v) &=& \mathcal{G}_{2,23}^*\mathcal{G}_{2,13}^*\mathcal{G}_{1,13}\mathcal{G}_{1,23}\\ &=& v_{23}\mathcal{G}_{1,23}^*v_{13}\mathcal{G}_{1,23}\\ &=& v_{23}(\iota\otimes \Upsilon_1)(v),\end{eqnarray*} and obviously $\Upsilon_2(x)= v\Upsilon_1(x)v^*$ for $x\in B(\mathscr{H})$. Hence $\Upsilon_1$ and $\Upsilon_2$ are outer equivalent.\\

\noindent Now for any right Galois object $(N,\alpha)$ with separable predual, there exists a coaction on $B(\mathscr{H})$ which has $(N,\alpha)$ as its associated Galois object: for example, one can take $\mathscr{H}\cong \overline{\mathscr{L}^2(N)}\otimes \mathscr{H}$ and equip it with the coaction \[\Upsilon: B(\overline{\mathscr{L}^2(N)}\otimes \mathscr{H})\rightarrow \widehat{M}\otimes \overline{\mathscr{L}^2(N)}\otimes \mathscr{H}:\]
\[\Upsilon(x) = (((J_N\otimes \widehat{J}_{21})\tilde{G}(J\otimes \widehat{J}_{12}))\otimes 1) (1\otimes x)(((J_N\otimes \widehat{J}_{12})\tilde{G}^*(J\otimes \widehat{J}_{21}))\otimes 1),\] i.e., take an amplification of the coaction coming from the regular left projective corepresentation of a Galois object. This observation then ends the proof of the proposition.\end{proof}

%\noindent In \cite{DeC4}, we show that contrary to the classical situation, compact quantum groups can have infinite-dimensional irreducible projective corepresentations. We then exploit this fact to establish the following surprising result: there exists a compact quantum group $(\widehat{M},\widehat{\Delta})$ and a 2-cocycle $\Omega\in \widehat{M}\otimes \widehat{M}$, such that the cocycle twisted quantum group $(\widehat{M},\widehat{\Delta}_{\Omega})$ is no longer compact.\\

\noindent \textbf{Acknowledgements:} I would like to thank my thesis advisor Alfons Van Daele who gave me the opportunity (and initial motivation) to study these problems. I would also like to thank L. Va\u{\i}nerman at the university of Caen, where part of this work was made. Finally, I would like to thank P. Hajac, for informing me about the Ehresmann construction, which provides one with the right geometrical intuition in these matters.

\end{document}